\documentclass[11pt,reqno]{amsart}
\usepackage[english]{babel}
\usepackage[a4paper,twoside,marginparwidth=50pt]{geometry}
\usepackage{microtype}
\usepackage{enumerate}
\usepackage{csquotes}
\usepackage{longtable}
\usepackage{bbm}
\usepackage{nicefrac}
\usepackage{array}
\usepackage{marginnote}
\usepackage{tcolorbox}

\usepackage{xspace}
\patchcmd{\subsection}{-.5em}{.5em}{}{}
\usepackage{caption} 
\newcolumntype{C}[1]{>{\centering\let\newline\\\arraybackslash\hspace{0pt}}m{#1}}

\usepackage[hyperfootnotes=false]{hyperref} 
\usepackage{color}
\hypersetup{
	colorlinks = true,
	linkcolor = {blue},
	urlcolor = {red},
	citecolor = {blue}
}

\makeatletter
\newcommand{\mysubsubsection}[1]{\subsubsection*{\bfseries #1}}
\renewcommand{\tocsection}[3]{
  \indentlabel{\@ifnotempty{#2}{\ignorespaces#1 #2\quad}}\bfseries#3}
\renewcommand{\tocsubsection}[3]{
  \indentlabel{\@ifnotempty{#2}{\ignorespaces#1 #2\quad}}#3}

\newcommand\@dotsep{4.5}
\def\@tocline#1#2#3#4#5#6#7{\relax
  \ifnum #1>\c@tocdepth
  \else
    \par \addpenalty\@secpenalty\addvspace{#2}
    \begingroup \hyphenpenalty\@M
    \@ifempty{#4}{
      \@tempdima\csname r@tocindent\number#1\endcsname\relax
    }{
      \@tempdima#4\relax
    }
    \parindent\z@ \leftskip#3\relax \advance\leftskip\@tempdima\relax
    \rightskip\@pnumwidth plus1em \parfillskip-\@pnumwidth
    #5\leavevmode\hskip-\@tempdima{#6}\nobreak
    \leaders\hbox{$\m@th\mkern \@dotsep mu\hbox{.}\mkern \@dotsep mu$}\hfill
    \nobreak
    \hbox to\@pnumwidth{\@tocpagenum{\ifnum#1=1\bfseries\fi#7}}\par
    \nobreak
    \endgroup
  \fi}
\AtBeginDocument{
\expandafter\renewcommand\csname r@tocindent0\endcsname{0pt}
}
\def\l@subsection{\@tocline{2}{0pt}{2.5pc}{5pc}{}}
\makeatother

\usepackage{amsmath}
\usepackage{amsthm}
\usepackage{amssymb}
\usepackage{mathrsfs} 
\usepackage{eucal}
\usepackage{mathtools}

\usepackage{graphicx}
\usepackage{tikz}
\usetikzlibrary{cd}

\newcounter{results}[section]

\theoremstyle{plain}
\newtheorem{theorem}[results]{Theorem}

\newtheorem{lemma}[results]{Lemma}
\newtheorem{proposition}[results]{Proposition}
\newtheorem{corollary}[results]{Corollary}

\theoremstyle{remark}
\newtheorem{remark}[results]{Remark}

\theoremstyle{definition}
\newtheorem{definition}[results]{Definition}
\newtheorem{experiment}{Experiment}[section]

\numberwithin{equation}{section}

\usepackage{algorithm,algpseudocode,tabularx}

\makeatletter
\newcommand{\multiline}[1]{%
  \begin{tabularx}{\dimexpr\linewidth-\ALG@thistlm}[t]{@{}X@{}}
    #1
  \end{tabularx}
}
\makeatother

\newcommand{\NN}{\mathcal{NN}}
\newcommand{\ds}{\mathfrak{D}}
\newcommand{\norm}[1]{ \left \| #1 \right \|}

\newcommand{\FW}{\FO_{\vartheta}^{W^p_p}}
\newcommand{\FWT}{\FO_{\vartheta}^{W^2_2}}
\newcommand{\FWtt}{\FO_{\vartheta}^{W_2}}
\newcommand{\GW}[1]{\GO_{\vartheta}^{#1}}
\newcommand{\GV}[1]{\GVO_{\vartheta}^{#1}}
\newcommand{\LR}[1]{\widetilde{\GVO}_{\vartheta}^{#1}}
\newcommand{\ALL}[1]{\mathbf{A}_{\vartheta}^{#1}}
\newcommand{\bmu}[1]{\mathbf{b}_{\vartheta}^{#1}}
\newcommand{\ttrain}{\mathcal{T}_{\mathrm{train}}}
\newcommand{\ttest}{\mathcal{T}_{\mathrm{test}}}
\newcommand{\Hot}{H^{1,2}(\prob_2(\R^d), W_2, \mm)}
\newcommand{\Cot}{\ccyl{\prob(\R^d)}{\rmC_b^1(\R^d)}}
\newcommand{\Lot}{L^2(\prob_2(\R^d), W_2, \mm)}
\newcommand{\dprod}[1]{\left<#1\right>}

\DeclareMathOperator{\FO}{F}
\DeclareMathOperator{\GO}{G}

\DeclareMathOperator{\GVO}{L}

\DeclareMathOperator{\LO}{L}

\newcommand{\E}{\mathbb{E}}

\newcommand{\N}{\mathbb{N}}

\newcommand{\R}{\mathbb{R}}

\renewcommand{\AA}{\mathscr{A}}

\newcommand{\CC}{\mathscr{C}}

\newcommand{\EE}{\mathscr{E}}

\newcommand{\HH}{\mathscr{H}}

\newcommand{\PP}{\mathbb{P}}
\renewcommand{\SS}{\mathscr{S}}

\newcommand{\mm}{{\mbox{\boldmath$m$}}}

\newcommand{\pp}{{\mbox{\boldmath$p$}}}

\newcommand{\xx}{{\boldsymbol x}}
\newcommand{\yy}{{\boldsymbol y}}

\newcommand{\ggamma}{{\mbox{\boldmath$\gamma$}}}

\newcommand{\pphi}{{\boldsymbol \phi}}

\newcommand{\sfd}{{\sf d}}

\newcommand{\rmC}{{\mathrm C}}

\newcommand{\rmB}{{\mathrm B}}
\newcommand{\rmD}{{\mathrm D}}

\newcommand{\Kliminf}{K\kern-3pt-\kern-2pt\mathop{\rm lim\,inf}\limits}  
\newcommand{\supp}{\mathop{\rm supp}\nolimits}   
\newcommand{\diam}{\mathop{\rm diam}\nolimits}   
\newcommand{\argmin}{\mathop{\rm argmin}\limits}   
\newcommand{\Lip}{\mathop{\rm Lip}\nolimits}          
\newcommand{\Lipb}{\mathop{\rm Lip}_b\nolimits}          

\newcommand{\lip}{\mathop{\rm lip}\nolimits}          

\renewcommand{\d}{{\mathrm d}}

\newcommand{\restr}[1]{\lower3pt\hbox{$|_{#1}$}}

\newcommand{\la}{{\langle}}                  
\newcommand{\ra}{{\rangle}}

\newcommand{\eps}{\varepsilon}  
\newcommand{\nachi}{{\raise.3ex\hbox{$\chi$}}}
\newcommand{\weakto}{\rightharpoonup}
\newcommand{\JJ}{{\mathcal J}}

\newcommand{\prob}{\mathcal P}

\newcommand{\de}{{\rm DE}}

\renewcommand{\mm}{\mathfrak m}
\renewcommand{\pp}{\mathfrak p}
\newcommand{\bmm}{\boldsymbol{\mathfrak m}}

\newcommand{\nc}{\normalcolor}

\newcommand{\J}I
\newcommand{\CE}{\mathsf{C\kern-1pt E}}
\newcommand{\NE}{\mathsf{N\kern-2.5pt E}}
\newcommand{\wCE}{\mathsf{wC\kern-1pt E}}

\newcommand{\pCE}{\mathsf{pC\kern-1pt E}}

\newcommand{\uphi}{\pphi}

\newcommand{\ccyl}[2]{\mathfrak C\big (#1,#2 \big)}

\newcommand{\rsqm}[1]{\mathsf m_p(#1)} 
\newcommand{\mres}{\mathbin{\vrule height 1.6ex depth 0pt width
0.13ex\vrule height 0.13ex depth 0pt width 1.3ex}}

\newcommand{\W}{\mathbb W}

\newcommand{\lin}[1]{\mathsf L_{#1}}
\newcommand{\prbt}{{\prob_p(\R^d)}}

\renewcommand{\de}{\, \mathrm d}

\title[Computing on the Wasserstein Space]{Approximation Theory, Computing, and Deep Learning on the Wasserstein Space}

\author{Massimo Fornasier}
\address{Massimo Fornasier: TUM School of Computation, Information, and Technlogy - Department of Mathematics, Boltzmannstrasse 3, 85748 Garching bei M\"unchen (Germany) \newline \& 
TUM Institute for Advanced Studies
\newline \&
Munich Data Science Institute \newline \& Munich Center for Machine Learning}
\email{massimo.fornasier@cit.tum.de}

\author{Pascal Heid}
\address{Pascal Heid: TUM School of Computation, Information, and Technlogy - Department of Mathematics, Boltzmannstrasse 3, 85748 Garching bei M\"unchen (Germany) \newline \& Munich Center for Machine Learning}
\email{hepa@ma.tum.de}

\author{Giacomo Enrico Sodini}
\address{Giacomo Enrico Sodini: Institut für Mathematik - Fakultät für Mathematik - Universität Wien, Oskar-Morgenstern-Platz 1, 1090 Wien (Austria)}
\email{giacomo.sodini@univie.ac.at}


\begin{document} 

\begin{abstract}
The challenge of approximating functions in infinite-dimensional spaces from finite samples is widely regarded as  formidable. In this study, we delve into the challenging problem of the numerical approximation of Sobolev-smooth functions defined on probability spaces. Our particular focus centers on the Wasserstein distance function, which serves as a relevant example.
In contrast to the existing body of literature focused on approximating efficiently pointwise evaluations, we chart a new course to define functional approximants by adopting three machine learning-based approaches:
\begin{itemize}
\item[1.] Solving a finite number of optimal transport problems and computing the corresponding Wasserstein potentials.
\item[2.] Employing empirical risk minimization with Tikhonov regularization in Wasserstein Sobolev spaces.
\item[3.]Addressing the problem through the saddle point formulation that characterizes the weak form of the Tikhonov functional's Euler-Lagrange equation.
\end{itemize}
As a theoretical contribution, we furnish explicit  and quantitative bounds on generalization errors for each of these solutions. In the proofs, we leverage the theory of metric Sobolev spaces and we combine it with techniques of optimal transport, variational calculus, and large deviation bounds.
In our numerical implementation, we harness appropriately designed neural networks to serve as  basis functions. These networks undergo training using diverse methodologies. This approach allows us to obtain approximating functions that can be rapidly evaluated after training. Consequently, our constructive solutions significantly enhance at equal accuracy the evaluation speed, surpassing that of state-of-the-art methods by several orders of magnitude. This allows evaluations over large datasets several times faster, including training, than traditional optimal transport algorithms. Moreover, our analytically designed deep learning architecture slightly outperforms the test error of state-of-the-art CNN architectures on datasets of images.\nc
\end{abstract}

\maketitle
\tableofcontents
\thispagestyle{empty}

\section{Introduction}

In this work we are concerned with the efficient numerical approximation of Sobolev-smooth functions defined on spaces of probability measures from the information obtained by a finite number of point evaluations.
Such approximation problems are considered already very challenging for functions defined on high dimensional Euclidean spaces, but they become particularly intriguing and formidable as the domain is a metric space of infinite dimensional nature.
As an inspiring and motivating example, we focus in particular - although not solely, see Section \ref{sec:cons} and Section \ref{sec:eulerlagrange} - on the study of the approximation of the Wasserstein distance function $\mu \mapsto W_p(\mu,\vartheta)$, where $\vartheta \in \mathcal{P}(K)$ is a given reference measure on a compact set $K \subset \mathbb{R}^d$, both from a theoretical and computational point of view. We recall that the Wasserstein distance arises as the solution of an optimal transport problem; cf.~Section~\ref{sec:wasserstein} below. Our theoretical approximation bounds are largely based on the theory of Wasserstein Sobolev spaces, especially on their Hilbertian structure, the Cheeger energy, and the algebra of cylinder functions; those notions are recalled in Section~\ref{sec:introwss}, see also \cite{FSS22, S22}. Concerning the computational framework, we make use of the fact that the algebra of cylinder functions is dense in the Wasserstein Sobolev space and that every cylinder function can be approximately realized by deep neural networks.

Hence, from a foundational point of view, with this paper we contribute to pioneer the connection of metric measure space theory and metric Sobolev spaces \cite{AGS14I, Bjorn-Bjorn11,Cheeger99,Shanmugalingam00} with numerical computations and machine learning. As we tread  this novel route, we draw some conceptual inspiration from previous works such as \cite{AMBROSIO2021108968,TSG_2017-2019__35__197_0,Zhangkai23} that enable the spectral embedding of RCD spaces into $L^2$ spaces.  Especially in cases where the spectrum of the Laplacian is discrete, as in the compact case, it offers an intriguing tool for the Euclidean embedding of metric space data points. From a computational point of view, the strength of our approach is the fast evaluation of Sobolev-smooth functions such as the Wasserstein distance $\mu \to W_p(\mu,\vartheta)$ once a deep neural network is setup and suitably trained. In particular, if the Wasserstein distance has to be computed pairwise on vast amounts of data (made of distributions), our approach outperforms the commonly employed methods in perspective of the evaluation time. 

Let us recall the relevance of the Wasserstein distance. Indeed, it is an important tool to compare (probability) measures. Besides its crucial role in the theory of optimal transport \cite{santambrogio,Villani:09}, it has had by now far reaching uses in numerical applications \cite{PeyreCuturi:2019}. One of the first data-driven applications of a discrete Wasserstein metric was in image retrieval and can be found in~\cite{Rubner2000TheEM}. Since then, the Wasserstein metric has been succesfully applied in various areas such as image processing, computer vision, statistics, and machine learning; we refer to~\cite{Kolouri:2017} for a comprehensive overview of specific applications and further references. 
Consequently, computational methods for the \emph{pointwise} evaluation of the Wasserstein distance and similar functions acting on spaces of probability measures have gained significant prominence in recent years. Concerning the computation of the Wasserstein distance of two measures, which amounts to the solution of an optimal transport problem, we may point to algorithms from linear programming such as the Hungarian and auction method, the superior Sinkhorn algorithm, or the more recent approach by the linear optimal transport framework to only name a few; an extensive review of computational schemes can be found in the Appendix~\ref{sec:app1}. 

In contrast to the existing body of literature, we move from the commonly addressed problem of approximating the Wasserstein distance pointwise to the more abstract one of approximating it as a function, which is relevant given a great deal of data. In particular, our manuscript differs from the standard setting of the works mentioned in the Appendix in~\ref{sec:app1}, where the task is to compute the Wasserstein distance for a pair rather than for a whole dataset. Indeed, our work takes a fresh direction by adopting a machine learning approach for computing the Wasserstein distance as a function on datasets. To our knowledge, this is the first work in this direction. Specifically, our objective is to compute this distance by training a straightforward approximation function using a finite training set of Wasserstein distance evaluations from a dataset, while ensuring a minimal relative test error across the remaining dataset. Additionally, we aim for the approximating function to possess a simplicity that allows for rapid numerical evaluations after training, which could improve of several orders of magnitude the speed of evaluation with respect to state-of-the-art methods. It shall be pointed out, however, that this does not apply if the Wasserstein distance of two measures from a statistically uncorrelated dataset has to be computed, but only for two data points drawn from the same statistics of the given training set, which is the essence of machine learning. Nonetheless, as will be highlighted by a numerical test, {\it the total computational time for the pairwise Wasserstein distance evaluation of a large amount of data can still be significantly improved by leveraging our novel approach, despite the cost of training.} As we elaborate later on, our approach extends beyond the specific case of the Wasserstein distance to encompass more general Wasserstein Sobolev functions.

Throughout the paper, we examine three closely interconnected machine learning approaches for approximating Wasserstein Sobolev functions by
\begin{itemize}
\item[1.]  solving a finite number of optimal transport problems and computing the corresponding Wasserstein potentials, cf. Section \ref{sec:conofpot} and Appendix \ref{sec:4} (this approach is limited to the approximation of the Wasserstein distance); this trainable approach will constitute the baseline for comparison with our next methods;
\item[2.]  empirical risk minimization using a Tikhonov regularization in Wasserstein Sobolev spaces, cf. Section \ref{sec:cons};
\item[3.] solving the saddle point problem that describes the Euler-Lagrange equation in weak form of the Tikhonov functional, cf. Section \ref{sec:eulerlagrange}.
\end{itemize}
In all these solutions we employ suitably defined neural networks to implement the basis functions. 

In the following we consider both generic non-negative and finite Borel measures $\mm$ over the space $\prob(K)$ and probability measures $\pp \in \prob(\prob(K))$ in order to describe data distributions; to keep things as simple as possible,  let us assume the domain to be a compact subset $K \Subset \R^d$. The choice between these two options depends on whether the results are deterministic or are based on randomization of the input.
The starting point for our analysis are the results of \cite{FSS22, S22} concerning the approximability of the Wasserstein distance and other Wasserstein Sobolev functions by so called \emph{cylinder functions}. Let us also mention that, even if we limit our analysis to the space of probability measures, many of the techniques developed in \cite{FSS22, S22} can also be applied to the space of non-negative and finite measures endowed with the Hellinger-Kantorovich distance (\cite{marc, LMS18}). For that reason we expect that the analysis developed herein should be robust enough so that it can be further generalized in this direction.  

As a starting and relevant example, the function that we would like to approximate is  given by
\begin{align} \label{eq:WDFintro}
\FWtt(\mu):= W_2(\mu, \vartheta), \quad \mu \in \prob(K),
\end{align}
where $\vartheta \in \prob(K)$  is a fixed reference probability measure and  $W_2$ is the $2$-Wasserstein distance between probability measures (while in the work we often consider a parameter $p$ possibly different from $2$ for the Wasserstein distance, we prefer to consider the simpler case $p=2$ at this introductory level), arising as the solution of the optimal transport problem
\[ W_2(\mu, \vartheta) := \left ( \inf \left \{ \int_{K \times K } |x-y|^2 \de \ggamma(x,y) : \ggamma \in \prob(K\times K) \text{ with marginals $\mu$ and $\vartheta$} \right \} \right )^{1/2}.\]

One of the main results of \cite{FSS22} states that, given any non-negative and finite Borel measure $\mm$ on $\prob(K)$, it is possible to find a sequence $F_n: \prob(K) \to \R$ of functions having the form
\begin{align} \label{eq:cylinderintro}
F_n(\mu) := \psi_n \left ( \int_K \phi^n_1 \de \mu, \int_K \phi^n_2 \de \mu, \dots, \int_K \phi_{N_n}^n \de \mu \right ), \quad \mu \in \prob(K)
\end{align}
such that 
\begin{equation}\label{eq:convergence}
F_n \to W_2(\cdot, \vartheta) \mbox{ in } L^2(\prob(K), \mm),
\end{equation}
 where  $\psi_n: \R^{N_n} \to \R$ and $\phi_i^n: K \to \R$ are smooth functions for $1 \le i \le N_n$, $n \in \N$. A function of the form~\eqref{eq:cylinderintro} is called \emph{cylinder function}, see also \eqref{eq:acyl}. For later use, we denote by $\Cot$ the set of all cylinder functions. We anticipate now that, as long as the function $\psi_n$ is realized or well-approximated by a neural network, the cylinder function $F_n$ can itself be considered a neural network over $\prob(K)$ with first-layer weights $\phi_i^n$ used to integrate the input $\mu$. The efficient approximation of high-dimensional functions, such as $\psi_n$, by neural networks is underlying our numerical results (cf.~Section \ref{sec:num}), but it is not central or exhaustively studied in our theoretical analysis. Instead, we do refer to the large body of more specific literature dedicated to the approximation theory by neural networks \cite{MR0111809, Cybenko1989ApproximationBS, HORNIK1989359, HORNIK1991251, BGKP:2019, Elbrchter2019DeepNN, KidgerLyons2020, MR4362469, devore_hanin_petrova_2021}, which is far from an exhaustive list.
 To provide a more explicit disclaimer regarding the use of neural networks, it is important to emphasize that we do not specifically target ensuring global optimization when dealing with neural networks. Given the inherent non-convex nature of the problem, addressing generalization errors in neural network optimization is widely recognized as one of the most challenging aspects of machine learning. In this paper, our focus does not extend to tackling this intricate issue. Instead, in our numerical experiments, we make the simplifying assumption that we are aiming at approximating minimization, and our primary goal is to achieve values of the objective functions that are sufficiently small.
 \\

The approximation result by cylinder functions has deep consequences in terms of the structure of the so called \emph{metric Sobolev space} on $(\prob(K), W_2, \mm)$ that we briefly discuss in Section \ref{sec:introwss} (and we refer the reader to \cite{FSS22,S22} for a more extended analysis), but here we are more interested in the practical consequences of this approximation property. In particular, while computing numerically the Wasserstein distance between two measures may be a challenging and heavy task, the computation of the above functions $F_n$ consists simply in the calculation of a finite number of integrals and then in the evaluation of a function of the resulting values. The simplicity of the latter procedure is thus the main reason to study how to provide an explicit and numerically efficient construction of the functions $F_n$, also because the convergence proof elaborated in \cite{FSS22} and further developed in \cite{S22} is not fully constructive.

This is the first problem that we address in Section \ref{sec:conofpot}. In particular, we show (see Proposition \ref{prop:laprima}) that, up to knowing a \emph{countable} number of functions (called \emph{Kantorovich potentials}), it is possible to build maps $\psi_n$ and $\phi_i^n$ satisfying the convergence \eqref{eq:convergence}; the main idea comes from the celebrated Kantorovich duality theorem, see, e.g., \cite[Theorem 6.1.4]{AGS08}, stating that, for every pair of sufficiently well-behaved probability measures $\mu, \vartheta \in \prob(K)$, there exists a pair of potentials $\varphi^*, \psi^*:K \to \overline{\R}$ such that 
\[ W_2(\mu, \vartheta) =\left(\int_K \varphi^* \de \mu + \int_K \psi^* \de \vartheta\right)^{1/2} = \sup_{\substack{\varphi(x) +\psi(y) \leq |x-y|^2, \\ \varphi, \psi \in \rmC_b(K)}} \left(\int_K \varphi \de \mu + \int_K \psi \de \vartheta\right)^{1/2}. \]
In general, however, the pair $(\varphi^*, \psi^*)$ depends on $\vartheta$ and $\mu$ so that it is not possible to simply consider the function $F:\prob(K) \to \R$ built as
\[ F( \nu ):= \left(\int_K \varphi^* \de \nu+ \int_K \psi^* \de \vartheta \right)^{1/2}, \quad \nu \in \prob(K), \]
as a cylinder function approximating the Wasserstein distance between $\nu$ and $\vartheta$.
This first issue can be fixed by considering a dense and countable subset of measures $(\sigma^h)_h \subset \prob(K)$ and the supremum over their corresponding Kantorovich potentials. The second issue concerns the regularity of such potentials, which, in general, are neither smooth nor bounded; it is thus useful to consider specific truncations and regularization arguments in order to produce a suitable sequence of potentials. 

All in all, the first result we obtain is Proposition \ref{prop:laprima}, which we report here in a simplified version for the sake of this introduction.

\begin{proposition}\label{prop:laprimaintro} Let $K \subset \R^d$ be a compact set and let $\vartheta \in \prob(K)$. There exist a sequence of smooth functions $(\eta_{ k})_{ k} \subset \rmC_b^1(\R^{ k})$ and smoothed versions of Kantorovich potentials $(v_i^{ k})_{ k} \subset \rmC_b^1(K)$, $1 \le i \le { k}$, (built starting from a dense subset $(\sigma^h)_h \subset \prob(K)$) such that the function 
\[ F_{ k}(\mu) := \eta_{ k} \left ( \int_{K} v_1^{ k} \de\mu , \int_{K} v_2^{ k} \de\mu, \dots,  \int_{K} v_{k}^{k} \de \mu \right ), \quad \mu \in \prob(K),\]
converges pointwise monotonically from below to $W_2(\vartheta, \cdot)$ in $\prob(K)$ as ${ k} \to \infty$.
\end{proposition}

The pointwise convergence obtained above in turn implies $$F_{ k} \to W_2(\cdot, \vartheta)  \mbox{ in }  L^2(\prob(K),  \mm) \mbox{ as }  k \to \infty,$$ for any measure $\mm$. Hence, it is a result that ensures approximability in $L^2(\prob(K), \mm)$ in a universal manner. It is certainly a solid first step in the direction of the computability of the smooth approximation of the Wasserstein distance by cylinder functions, but it still suffers from two limitations: first of all it requires the knowledge of an infinite number of Kantorovich potentials and, secondly, it does not come with any quantitative convergence rate.

We obtain both these improvements by trading such a deterministic construction and its universality with a probabilistic approach that depends on the choice of the underlying measure $\mm$; namely, in Section \ref{sec:32}, we introduce the notion of random subcovering of a metric measure space and we use it to quantify the convergence of (a randomly adapted version of) the functions $F_{ k}$ as above.

Let us first of all introduce the relevant definition: given a probability measure $\pp$ on a complete and separable metric space $(\SS, \sfd)$ (as a concrete example we have the metric space $(\prob(K), W_2)$ in mind), we call the quantity
$$
p_{\varepsilon, k}:=\mathbb{P} \bigg ( X \in \cup_{i=1}^{ k} B(X_i,\varepsilon)\bigg ),
$$
the \emph{$(\varepsilon,{ k})$-subcovering probability}, where $X, X_1, \dots, X_{ k}$ are drawn i.i.d.~according to $\pp$, $\eps \in (0,1)$ and ${ k} \in \N$. This is simply the probability that a (random) point $X \in \SS$ belongs to at least one of the (random) $\eps$-balls centered at points $X_1, \dots, X_{ k}$ and serves as a quantification of the concentration of the measure $\pp$.
 The value of $p_{\varepsilon,{ k}}$ provides insight into how {\it locally concentrated} the measure $\pp$ is on the space $\SS$: if $p_{\varepsilon,{ k}}$ is close to $1$ for relatively small $\eps$ and moderate ${ k}$, this indicates that the measure $\pp$ is highly concentrated on few small regions of the space.  In contrast, if, for very small $\eps$, $p_{\varepsilon,{ k}}$ is small (i.e., close to zero) unless ${ k}$ is large reflects that the measure $\pp$ is spread out.\nc

After providing a few results concerning the study of the quantity $p_{\eps, { k}}$, see in particular Lemma \ref{lem:expconv} showing that $p_{\eps, { k}} \to 1$  exponentially fast as ${ k} \to + \infty$, the main result of Section \ref{sec:32} is the following ``randomized and quantified" version of the above Proposition \ref{prop:laprimaintro}. Again, the version reported here is simplified for the sake of clarity, and we refer to Proposition \ref{prop:seconda} for the complete statement.

\begin{proposition}\label{prop:lasecondaintro} Let ${ k} \in \N$ and $K \subset \R^d$ be a compact set. Consider $\mu_1, \dots, \mu_{ k}$ i.i.d.~random variables on $\prob(K)$ distributed according to $\pp \in \prob(\prob(K))$ and let $\vartheta \in \prob(K)$ be fixed. Let $(v_1, \dots, v_{ k})$ be smoothed versions of Kantorovich potentials computed starting from the measures $\mu_1, \dots, \mu_{ k}$ and let $F_{ k}: \prob(K) \to \R$ be defined as 
\[ F_{ k}(\mu) := \eta_{ k} \left ( \int_{K} v_1 \de\mu , \int_{K} v_2 \de\mu, \dots,  \int_{K} v_{ k} \de \mu \right ), \quad \mu \in \prob(K).\]
Then
\[ \E \left [ \int_{\prob(K)} |F_{ k}(\mu)-\FWtt(\mu)|^2 \,\de\pp(\mu) \right ] \le C\left [ (1-p_{\eps,{k}})+p_{\eps,{k}}({k}^{-2}+\eps^{2})\right ],\]
where $C$ is a constant depending only on $K$.
\end{proposition}
It is clear that the computation of the above function $F_{k}$ is drastically simpler than the one of $F_n$ as in Proposition \ref{prop:laprimaintro}, requiring the knowledge of only a finite number of Kantorovich potentials, which are drawn randomly; we can also see an explicit order of convergence in terms of the $(\eps, {k})$-subcovering probabilities. (We re-iterate that thanks to the compactness of $K$, for $\varepsilon>0$ fixed, we have that $p_{\eps,{k}} \to 1$ exponentially fast as ${k}\to \infty$, see Lemma \ref{lem:expconv}.) We further mention that the estimate in Proposition \ref{prop:lasecondaintro} can be regarded as a {\it generalization error} as it quantifies in mean-squares the misfit of the approximant integrated over the entire measure $\pp$.

\medskip
This result is further studied in the specific context of a discrete base domain $\mathfrak{D}$, i.e., the compact set $K$ is replaced by a discrete set $\mathfrak{D}$; we refer to Appendix~\ref{sec:4}\nc. The restriction to a discrete domain allows for a simpler construction of the approximating sequence by cylinder functions, which still relies on pre-computed Kantorovich potentials. Additionally, we show how to numerically and efficiently implement the procedure on a computer via neural networks (whose architecture is used also for training). In particular, for a fixed reference measure $\vartheta \in \mathcal{P}(\mathfrak{D})$ and a dense and countable subset $(\mu_k)_{k \in \mathbb{N}}$ of $\mathcal{P}(\mathfrak{D})$, denote by $(\varphi_k,\psi_k)$, $k \in \mathbb{N}$, the corresponding pairs of Kantorovich potentials; i.e., we have that
\begin{align} \label{eq:WDF2intro}
\FWT(\mu_k):=W_2^2(\vartheta, \mu_k)=\int_{\mathfrak{D}} \varphi_k \, \de \mu_k + \int_{\mathfrak{D}} \psi_k \, \de \vartheta.
\end{align}
Based thereon, for $I \subset \mathbb{N}$, we define the cylinder function 
\begin{align} \label{eq:Gintro}
\GW{I}(\mu):=\max_{k \in I}\int_{\mathfrak{D}} \varphi_k \, \de \mu+\int_{\mathfrak{D}} \psi_k \, \de \vartheta.
\end{align}
We show (cf.~Theorem \ref{thm:eps}) that, for any $\varepsilon>0$ there exists a finite index set $I_\epsilon \subset \mathbb{N}$ such that
\begin{align} \label{eq:approxintro}
\sup_{\mu \in \mathcal{P}(\mathfrak{D})} \left|\FWT(\mu)-\GW{I_\varepsilon}(\mu)\right| \leq \varepsilon.
\end{align}
Thereupon we verify that the cylinder function $\GW{I}$ from~\eqref{eq:Gintro} can be realized as a deep neural network. Indeed, since 
\[\mu \mapsto \int_{\mathfrak{D}} \varphi_k \, \de \mu+\int_{\mathfrak{D}} \psi_k \, \de \vartheta\]
is an affine function, it can be represented in the finite dimensional setting by a (weight) matrix plus a (bias) vector. Moreover, it is known that the maxima function can be realized by a deep neural network, see, e.g.,~\cite{BJK:2022,GJS:2022,JR:2020}, and the composition of those functions can be obtained by the concatenation of the neural networks (\cite[Prop.~2.14]{JR:2020}).

\medskip

The approach described by Proposition \ref{prop:lasecondaintro} does apply quite specifically to the approximation of the Wasserstein distance function, thanks to the Kantorovich duality. In order to allow for efficient approximations of more general functions, we study the construction of approximants by empirical risk minimization.
In particular, in Section \ref{sec:cons}, our specific goal is to elucidate the conditions and the method through which we can achieve the best possible approximation of the Wasserstein distance function as in equation \eqref{eq:WDFintro} and other Sobolev-smooth functions. We aim to do this within the context of the space $(\prob(K), W_2, \mm_n)$, where $\mm_n$ represents an empirical estimate of the measure $\mm$. Furthermore, we introduce a degree of regularization using the pre-Cheeger energy as additional term to cope with noisy data. Ultimately, our objective is to showcase how this approximation converges robustly to the sought function within the space $(\prob(K), W_2, \mm)$.
We reiterate that this theory allows us to extend the scope of our results to a larger class of functions (beyond the Wasserstein distance function), which are elements of  suitably defined Wasserstein Sobolev spaces. 
Before stating the precise result we need to briefly introduce the concept of (pre-)Cheeger energy and Sobolev norm in this context (see Section \ref{sec:introwss} for a more comprehensive explanation and references to the literature). 

The fundamental idea is that to every cylinder function $F: \prob(K) \to \R$, thus having the form
\[ F(\mu):= \psi \left ( \int_K \phi_1 \de \mu, \dots, \int_K \phi_N \de \mu \right ), \quad \mu \in \prob(K),\]
we can associate a notion of derivative given by
\begin{equation}
\rmD F(\mu,x):=\sum_{n=1}^N \partial_n \psi \left ( \lin \uphi(\mu)
                 \right ) \nabla \phi_n(x), \quad (\mu,x) \in \prob(\R^d) \times \R^d,   
\end{equation}
where $ \lin \uphi(\mu) := \left (\int_K \phi_1 \de \mu, \dots, \int_K \phi_N \de \mu \right )$.
Besides being intuitively the correct object to be rightfully considered as the derivative of $F$, it can be seen that the $L^2(K, \mu)$-norm of $\rmD F(\mu, \cdot)$ corresponds to the so called \emph{asymptotic Lipschitz constant} of $F$ at $\mu$ (cf.~Proposition \ref{prop:introwss}), a purely metric object that generalizes the notion of derivative.

This leads to the definition of the \emph{pre-Cheeger} energy of the function $F$ given by
\[ \pCE_{2, \mm}(F) := \int_{\prob(K)} \int_K |\rmD F(\mu,x)|^2 \de \mu(x) \de \mm(\mu) = \int_{\prob(K)} \|\rmD F(\mu, \cdot)\|^2_{L^2(K,\mu)} \de \mm(\mu). \]
The completion of the space of cylinder functions with respect to the norm
\begin{align} \label{eq:prenorm} |F|_{pH^{1,2}_\mm}^2 := \|F\|_{L^2(\prob(K),\mm)}^2 + \pCE_{2, \mm}(F)
\end{align}
is the Sobolev space $H^{1,2}(\prob(K), W_2, \mm)$ with norm 
\[ |F|_{H^{1,2}_\mm}^2:= \|F\|_{L^2(\prob(K), \mm)} + \CE_{2, \mm}(F), \quad F \in H^{1,2}(\prob(K), W_2, \mm),\]
where $\CE_{2, \mm}$ is called \emph{Cheeger energy} of $F$; the specific definition of this notion is, for the moment, not relevant, but is introduced below in~\eqref{eq:relaxintroche}.

The main result of Section \ref{sec:cons1} is the following (we refer to Theorem \ref{thm:moscow} and Corollary \ref{cor:gamma} for the detailed and complete statements). 

\begin{theorem}\label{theo:sec5} Let $(\mm_n)_n$ be a sequence of non-negative Borel measures on $\prob(K)$ weakly converging to the measure $\mm$ and let $\lambda \ge 0$. For a given Lipschitz continuous function $F: \prob(K) \to \R$, we further define the functionals
\begin{align*}
\JJ_n(G) &:= \begin{cases} \|G-F\|^2_{L^2(\prob_2(K), \mm_n)} + \lambda \pCE_{2, \mm_n}(G), \quad &\text{ if } G \in \mathcal{U}_n\nc, \\
    + \infty \quad &\text{ if } G \in H^{1,2}(\prob(K), W_2, \mm) \setminus \mathcal{U}_n\nc,\end{cases}\\
    \overline{\JJ}(G) &:= \|G-F\|^2_{L^2(\prob_2(K), \mm)} + \lambda \CE_{2, \mm}(G), \qquad \quad \,\,G \in H^{1,2}(\prob(K), W_2, \mm).
\end{align*}
Then there exists a sequence of \emph{finite dimensional subsets} of cylinder functions $\mathcal{U}_n\nc$ such that $\JJ_n$ $\Gamma$-converges to $\overline{\mathcal{J}}$ as $n \to + \infty$. In particular, the minimizers of $\JJ_n$ converge to the minimizer of $\overline{\JJ}$ as $n \to + \infty$ in $L^2(\prob(K), \mm)$. For $\lambda>0$, the convergence is even in $H^{1,2}(\prob(K), W_2, \mm)$.
\end{theorem}
This result applies for instance, but not exclusively, to the Wasserstein distance function $F=\FWtt$. We further note that the above convergence is completely deterministic and it applies to any measure $\mm$; i.e., it is not restricted to probability measures. However, it does not provide any convergence rate and it does not offer yet a proper analysis for data corrupted by noise. Moreover the subsets $\mathcal{V}_n$ are not easily implementable as they require bounds on higher order derivatives, see Definition \ref{def:vn}.
For these reasons, we discuss in Section \ref{sec:cohen} a different approach, which leads to an analogous kind of convergence. 
The route of these results follows and generalizes \cite{cohen}. 
However, it is based again on a randomization for the choice of $\mm_n$ and we gain an explicit error estimate, albeit the convergence is only in mean and no longer deterministic. As a relevant element of novelty with respect to \cite{cohen} that focused on pure least squares, let us also stress that we generalize the entire argument to Tikhonov-type functionals such as $\JJ_n$ and that we work on the non-smooth space of probability measures on a compact set $\prob(K)$ rather than the Euclidean space. Hence, we had to devise appropriate methods to deal properly with the non-standard pre-Cheeger energy regularization term, which eventually contributes to control the propagation of the data noise on the final generalization error. (While we were concluding this paper, we came aware also of related results in the very recent preprint \cite{sonnleitner2023} that focuses again on a pure least squares rather than on a Tikhonov regularization model as we do here).
Our main result in this direction is the following (see Theorem \ref{teo:cohen2}).

\begin{theorem}  Let $K \subset \R^d$ be a compact set, $\pp$ be a probability measure on $\prob(K)$ and $\pp_N$, $N \in \N$, be the empirical measure associated to random points $\mu_1, \dots, \mu_N$ which are i.i.d. with respect to $\pp$. Let $F: \prob(K) \to [-M,M]$, $M>0$,  be a bounded and measurable function and let $ \widetilde{F}^\star $ be defined as 
\[  \widetilde{F}^\star   := -M \vee S_N^{\lambda,n}( \widetilde{F} ) \wedge M,\]
where $S_N^{\lambda,n}(\widetilde{F})$ is the minimizer of the functional 
\[ \JJ_{N, F}^\lambda(G):= \| \widetilde{F}-G\|^2_{L^2(\prob(K), \pp_N)} + \lambda \pCE_{2,\pp_N}(G), \quad G \in \mathcal{V}_n,\]
with $\mathcal{V}_n$, $n \in \N$, being an arbitrary $n$-dimensional subspace of cylinder functions, and $$
\widetilde{F}(\mu_j):= F(\mu_j)+\eta_j, \quad j=1, \dots, N,
$$
is a noisy version of $F$ obtained with i.i.d.~random variables $(\eta_j)_{j=1}^N$ with variance bounded by $\sigma^2>0$. If $r>0$ and $M,n \in \N$ are suitably chosen (see \eqref{eq:thecond})\nc, then 
\begin{align*}
\mathbb{E} \left [ \|F- \widetilde{F}^\star \|_{L^2(\prob(K), \pp)}^2\right ] &\le 2 e(F,n) \left ( 1 + \frac{ c_{1/2}}{\log(N)(1+r)(1/2+\lambda  \mu_{\min,n} )^2} \right ) + 8\lambda \pCE_{2, \pp}(P_n F) + \\
& \quad + 4\frac{\sigma^2}{(1+\lambda \mu_{\min,n} )^2}\frac{n}{N}+ 2  M^2 N^{-r},
\end{align*}
where $P_n F$ is the ${L^2(\prob(K), \pp)}$-orthogonal projection of $F$ onto $\mathcal{V}_n$,
\[ e(F,n) := \|P_n F-F\|^2_{L^2(\prob(K), \pp)} \]
is the $\mathcal{V}_n$-best approximation error with respect to the $L^2(\prob(K), \pp)$-norm, and $ \mu_{\min,n} $ is the minimal value of the pre-Cheeger energy evaluated on a suitable orthonormal basis of $\mathcal{V}_n$.
\end{theorem}

Up to scaling the parameter $\lambda$ and choosing a subspace $\mathcal{V}_n$ that entails a bound on the pre-Cheeger energy w.r.t.~$\pp$, the above estimate ensures the convergence of $\widetilde{F}^\star$ to $F$ in expectation. Moreover, the explicit generalization error bound completely reveals the interplay between the number $N$ of the noisy data, the dimension $n$ of the space $\mathcal{V}_n$, the regularization parameter $\lambda$, and the noise level $\sigma$. 
In particular the contribution of the noise to the error bound can be controlled by either choosing a larger number $N$ of data or by considering a larger regularization parameter $\lambda$ that one needs to optimize.
For $\lambda =0$ one simply obtains again an analogous result as in \cite{cohen}. Note also that the subspace $\mathcal{V}_n$ is arbitrary and may be chosen in such a way that the projection error $e(F,n)$ is very small, since we know that $F$ can be approximated arbitrarily well by cylinder functions, see Theorem \ref{thm:mainold}. 

\medskip
As a final approach to the finite sample approximation problem, let us consider again the energy functional
\[
\overline{\JJ}(G):=\norm{G-\widetilde{F}}^2_{\Lot}+ \lambda \CE_{2,\mm}(G), \qquad G \in \Hot,
\]
where $\widetilde{F}$ is a possibly noisy version of a given function $F \in \Lot$.
Since $\overline{\JJ}$ is a strictly convex functional on $\Hot$, its minimizer is equivalently the unique solution of the Euler--Lagrange equation specified in~\eqref{eq:ELdual}. Subsequently, stating the Euler--Lagrange equation as an operator equation in the dual space and employing the density of cylinder functions in the Wasserstein Sobolev space $\Hot$ leads to a saddle point problem of the form  
\[
\inf_{G \in \Cot} \sup_{\substack{H \in \Cot \\ H \neq 0}} \frac{(G-\widetilde{F},H)_{\Lot} +\lambda \pCE_{2,\mm}(G,H)}{|H|_{pH^{1,2}_\mm}},
\]
where $\Cot$ denotes again the space of all cylinder functions and $\pCE_{2,\mm}(\cdot,\cdot)$ is defined as in~\eqref{eq:preip} below. To solve this problem in applications, we have to replace the measure $\mm$ by an empirical approximation $\mm_N$. Moreover, since we cannot efficiently numerically implement the entire set of all cylinder functions (especially in very high dimension), we restrict to a subset of cylinder functions represented by deep neural networks. In particular, similar as in~\cite{Zangetal:20}, we use an adversarial network for the computation of the supremum and a solution network for the infimum in the saddle point problem above; this gives rise to the adversarial training Algorithm~\ref{alg:Adversarial}.  

\medskip 

We conclude our paper with a series of numerical experiments that test the three approaches illustrated above on the approximation of the Wasserstein distance between probability measures derived from datasets.
Specifically, we employ two widely recognized datasets: MNIST, comprising handwritten characters, and CIFAR-10, featuring coloured $32 \times 32$ pixels images from $10$ classes. In our approach, we interpret the elements within these datasets as probability measures. Consequently, we regard the normalized images within these datasets as representations of positive probability densities. First of all, we investigate how the approximation accuracy of~\eqref{eq:Gintro} improves for an increasing number of precomputed potentials. Subsequently, we model the function~\eqref{eq:Gintro} as a deep neural network, and examine how the error further decays if we optimize the network parameters by a suitable training procedure. Finally, we run some experiments for the solution of the Euler--Lagrange equation as briefly described above. 
We mention that the numerical experiments require considerable computational effort, which limits the number of training data that we consider and end with relatively large test errors. Yet, we obtain approximation errors, which are competitive in accuracy with respect to state-of-the-art methods for computing the Wasserstein distance, with an improvement in computational time after training of several orders of magnitude, see Table \ref{table:comptime}.  The numerical results of this paper follow the principles of reproducible research and the software to reproduce them is available at \url{https://github.com/heipas/Computing-on-WassersteinSpace}.

\medskip 

\textbf{Outline.} In Section~\ref{sec:introwss} we first introduce some basic notions and results about metric Sobolev spaces, Wasserstein spaces, and Wasserstein Sobolev spaces, which are fundamental for the present manuscript. Then, Section~\ref{sec:conofpot} deals with the explicit construction of cylinder functions that approximate the Wasserstein distance to a given reference measure, the random $\eps$-subcovering of a complete and separable metric space, and the empirical approximation of the Wasserstein distance based on a finite number of samples of the distribution on the Wasserstein Sobolev space. In Appendix~\ref{sec:4}, we translate those results to the case of a discrete base space, which might be of interest in applications. In Section~\ref{sec:cons1}, we verify that a sequence of cylinder functions obtained by empirical risk minimization converges to the sought solution, which could be the Wasserstein distance up to some arbitrarily small regularization term. In the subsequent Section~\ref{sec:cohen}, at the expense of determinism, we even obtain a convergence rate up to high probability of the empirical risk minimization approach. Then, in Section~\ref{sec:eulerlagrange} we introduce a novel approach to compute the Wasserstein distance based on the Euler-Lagrange formulation of the risk minimization. The work is finally rounded off by some numerical experiments in the context of the MNIST and CIFAR-10 datasets. In that setting, we also relate cylinder functions to deep neural networks.

\mysubsubsection{Acknowledgments.} M.F.~and G.E.S.~gratefully acknowledge the support of the Institute for Advanced Study of the Technical University of Munich, funded by the German Excellence Initiative. M.F.~and P.H.~acknowledge the support of the Munich Center for Machine Learning. The authors thank Felix Krahmer and Giuseppe Savaré for various useful conversations on the subject of this paper.

\section{Wasserstein Sobolev spaces} \label{sec:introwss}
This section is devoted to the introduction of some basic tools in the theory of metric Sobolev spaces (Section \ref{sec:mss}), the definition of the Wasserstein space of probability measures as a relevant example of a metric space (Section \ref{sec:wasserstein}), and, finally, (Section \ref{sec:wss}) to the presentation of some of the results of \cite{FSS22} regarding Wasserstein Sobolev spaces. 

\subsection{The relaxation approach to metric Sobolev spaces}\label{sec:mss}
In the last years, several notions of metric Sobolev spaces on metric measure spaces were studied, such as, e.g.,~the Newtonian approach in \cite{Shanmugalingam00, Bjorn-Bjorn11}. Here we focus on the one presented by Ambrosio, Gigli and Savaré \cite{AGS14I}, where, starting from the ideas of Cheeger \cite{Cheeger99}, they introduce the following notion of $q$-relaxed ($q \in (1,+\infty))$ gradient: given a metric measure space $(\SS, \sfd, \mm)$ (this is to say that $(\SS,\sfd)$ is a complete and separable  metric space and $\mm$ is a finite non-negative Borel measure on $(\SS, \sfd)$), a function $G \in L^q(\SS, \mm)$ is a $q$-relaxed gradient (see \cite[Definition 3.1.5]{Savare22}) of $f \in L^0(\SS, \mm)$ 
if there exist a sequence of bounded Lipschitz functions $(f_n)_n \subset \Lipb(\SS,\sfd)$ and $\tilde{G} \in L^q(\SS, \mm)$ such that
\begin{enumerate}
    \item $f_n \to f$ in $L^0(\SS, \mm)$ and $\lip_\sfd f_n \weakto \tilde{G}$ in $L^q(\SS, \mm)$,
    \item $\tilde{G} \le G$ $\mm$-a.e.~in $\SS$;
\end{enumerate}
here, for a function $h:\SS \to \R$, the asymptotic Lipschitz constant is defined as
\begin{equation}\label{eq:aslipintro}
  \lip_\sfd h(x):=
  \limsup_{y,z\to x,\ y\neq z}\frac{|h(y)-h(z)|}{\sfd(y,z)}, \quad x \in \SS,
\end{equation}
and $\Lip_b(\SS, \sfd)$ denotes the space of $\sfd$-Lipschitz and bounded functions on $\SS$. In particular, for $f \in \Lip_b(\SS, \sfd)$, we have that
\begin{equation}\label{eq:thelip}
    \Lip(f,\SS):= \sup_{x,y \in \SS, \, x \ne y} \frac{|f(x)-f(y)|}{\sfd(x,y)} < \infty.
\end{equation}
It is not difficult to see that if $f \in L^0(\SS, \mm)$ admits at least one $q$-relaxed gradient, then it has a minimal one (both in $L^q$-norm and in the $\mm$-a.e.~sense) which is denoted by $|\rmD f|_{\star, q}$ and provides a notion of (norm of the) derivative for $f$. We can thus define the $q$-Cheeger energy of $f \in L^0(\SS, \mm)$ as
\[ \CE_{q,\mm}(f) := \begin{cases}  \int_\SS |\rmD f|_{\star, q}^q \de\mm \quad &\text{ if $f$ admits a $q$-relaxed gradient},\\ + \infty \quad &\text { else}. \end{cases}\]
This quantity can be also obtained as the relaxation (see e.g.~\cite[Corollary 3.1.7]{Savare22}) of the so called pre-$q$-Cheeger energy 
\begin{equation}\label{eq:prech}
\pCE_{q,\mm}(f):= \int_\SS (\lip_\sfd f)^q \de\mm, \quad f \in \Lipb(\SS, \sfd),  
\end{equation}
meaning that
\begin{equation}\label{eq:relaxintroche}
  \CE_{q,\mm}(f) = \inf \left \{ \liminf_{n \to + \infty}
    \pCE_{q,\mm}(f_n) : (f_n)_n \subset \Lipb(\SS, \sfd), \, f_n \to f \text{ in } L^0(\SS,\mm)
  \right \}.
\end{equation}
The resulting Sobolev space $H^{1,q}(\SS, \sfd, \mm)$ is then the vector space of functions $f\in L^q(\SS, \mm)$ with finite Cheeger energy endowed with the norm 
\begin{align} \label{eq:CheegerEnergy} |f|^q_{H^{1,q}(\SS, \sfd, \mm)}:= \int_\SS |f|^q \de\mm + \CE_{q,\mm}(f),
\end{align}
which makes it a Banach space (cf.~\cite[Theorem 2.1.17]{pasquabook}). However, in general, even for $q=2$, $H^{1,q}(\SS, \sfd, \mm)$ is not a Hilbert space. For instance, for $\SS=\R^d$, $\sfd(x,y):=|x-y|_\infty$ is the distance induced by the infinity norm, and $\mm$ is a Gaussian measure on $\R^d$, this is indeed not a Hilbert space.

\subsection{Wasserstein spaces} \label{sec:wasserstein}
We briefly collect here the main definitions related to the notion of Wasserstein spaces of probability measures. Given a metric space $(U, \varrho)$, we denote by $\prob(U)$ the set of Borel probability measures on $U$ and by $\mathsf{m}_{p,x_0}(\mu)$ the $p$-th moment, $p \in (1,+\infty)$, of a measure $\mu \in \prob(U)$ defined as
\begin{equation}\label{eq:pmom}
    \mathsf{m}_{p,x_0} (\mu) = \left ( \int_{U} \varrho(x,x_0)^p \de\mu(x) \right )^{1/p},
\end{equation}
where $x_0 \in U$ is a fixed point. Usually $U$ is a subset of $\R^d$ and, in this case, we consider as distance $\varrho$ the one induced by the Euclidean norm and we take as $x_0$ the origin in $\R^d$, also removing the subscript $x_0$ in the notation for the moment. 

We consider the space $\prob_p(U)$, which is the subset of  $\prob(U)$ of probability measures with finite $p$-th moment:
\[  \prob_p(U):= \left \{ \mu \in  \prob(U) : \mathsf{m}_{p,x_0} (\mu) < + \infty \right \}.\]
Notice that the above definition doesn't depend on the point $x_0$ that has been used.
The $p$-Wasserstein distance between two points $\mu, \nu \in \prob_p(U)$ is defined as
\[ W_p(\mu, \nu):= \left ( \inf \left \{ \int_{ U \times U }  \varrho(x,y)^p \de \ggamma(x, y) : \ggamma \in \Gamma(\mu, \nu) \right \}\right )^{1/p}, \]
where $\Gamma(\mu, \nu)$ denotes the subset of Borel probability measures on  $U \times U$ having as marginals $\mu$ and $\nu$, also called \emph{transport plans} or simply \emph{plans} between $\mu$ and $\nu$.

It is well known that the infimum above is attained on a non empty and compact subset $\Gamma_o(\mu, \nu) \subset \Gamma(\mu, \nu)$. Notice that, in general, $\Gamma_o(\mu, \nu)$ also depends on $p$ and $\varrho$, but we omit this  dependence since it is always clear from the context which value of $p$  and which distance $\varrho$ we are referring to. Elements of $\Gamma_o(\mu, \nu)$ are called \emph{optimal transport plans} or simply \emph{optimal plans} between $\mu$ and $\nu$ and thus they satisfy
\begin{equation}\label{eq:optplan}
W_p^p(\mu, \nu) = \int_{ U \times U }  \varrho(x,y)^p \de \ggamma(x,y), \quad \text{ for every } \ggamma \in \Gamma_o(\mu, \nu).  
\end{equation}
It turns out that $(\prob_p(U), W_p)$ is a metric space which is complete (resp.~separable) if $(U, \varrho)$ is complete (resp.~separable), see \cite[Proposition 7.1.5]{AGS08}. We also recall that $\prob_p(U)$ is compact if and only if $(U,\varrho)$ is compact (see \cite[Remark 7.1.8]{AGS08}) and, in this case, we obviously have $\prob_p(U)=\prob(U)$ for every $p \in (1,+\infty)$.

We refer, e.g.,~to \cite{Villani:09,AGS08,Villani03} for a comprehensive treatment of the theory of Optimal Transport and Wasserstein spaces.

\subsection{Density of subalgebras of Lipschitz functions in Wasserstein Sobolev spaces}\label{sec:wss}
The starting point for the present work is the density result of \cite{FSS22}. Sticking for a short moment to the abstract framework of a general complete and separable metric space $(\SS,\sfd)$, it is shown in \cite{FSS22} that a sufficiently rich algebra of functions $\AA \subset \Lipb(\SS,\sfd)$ is \emph{dense in energy} in the Sobolev space $f \in H^{1,q}(\SS,\sfd, \mm)$. In particular, under suitable assumptions on the algebra $\AA$, we have that for every $f \in H^{1,q}(\SS,\sfd, \mm)$ there exists a sequence $(f_n)_n \subset \AA$ such that
\begin{equation}\label{eq:sap}
    f_n \to f \quad \text{and} \quad  \lip_\sfd f_n \to |\rmD f|_{\star, q} \quad \text{ in } L^q(\SS, \mm).
\end{equation}
Such a property is particularly relevant in case the algebra $\AA$ satisfies some interesting properties, which might not be clear for the whole set of bounded Lipschitz functions, and in turn can be transferred to the whole Sobolev space. For instance, this is the case for cylinder functions on the space of probability measures with the $p$-Wasserstein distance analyzed in \cite{FSS22} provided that the base space is a Hilbert space and $p=q=2$, and subsequently further developed in \cite{S22} to the case of a Polish metric base space and $p,q \in (1,+\infty)$.

Let us recall the definition of cylinder function: a function $F:\prob(\R^d) \to \R$ is a cylinder function if it is of the form
\begin{equation}\label{eq:acyl}
F = \psi \circ \lin{\pphi},  
\end{equation}
where $\psi \in \rmC_b^1(\R^N)$, $\pphi=(\phi_1, \phi_1, \dots \phi_N) \in (\rmC_b^1(\R^d))^N$, and $\lin{\pphi}:\prob(\R^d) \to \R^N$ being defined as
\begin{equation}\label{eq:lin}
    \lin{\pphi}(\mu):= \left ( \int_{\R^d} \phi_1 \de\mu, \dots, \int_{\R^d} \phi_N \de\mu \right ), \quad \mu \in \prob(\R^d).
\end{equation}

The set of cylinder functions is denoted by $\ccyl{\prob(\R^d)}{\rmC_b^1(\R^d)}$.

The ensuing result states that the asymptotic Lipschitz constant (cf.~{\eqref{eq:aslipintro}) of a cylinder function has a simple expression, see \cite[Proposition 4.7]{S22} and \cite[Proposition 4.9]{FSS22}} for the proof of the statement.

\begin{proposition}\label{prop:introwss} Let $F \in \ccyl{\prob(\R^d)}{\rmC_b^1(\R^d)}$ be as in \eqref{eq:acyl} and let us define
\begin{equation}\label{eq:thediff}
\rmD F(\mu,x):=\sum_{n=1}^N \partial_n \psi \left ( \lin \uphi(\mu)
                 \right ) \nabla \phi_n(x), \quad (\mu,x) \in \prob(\R^d) \times \R^d.   
\end{equation}
Then, for every $\mu \in \prob_p(\R^d)$, we have that
\begin{equation}\label{eq:pcerep}
    (\lip_{W_p} F (\mu))^{p'} =  \int_{\R^d} \left |\rmD F (\mu,x)\right |^{p'} \de \mu(x) =: \|\rmD F [\mu]\|_{L^{p'}_\mu}^{p'},
\end{equation} 
where $p':=p/(p-1)$.
\end{proposition}

The following statement is the density result mentioned above, see \cite[Theorem 4.10]{FSS22} and \cite[Theorem 4.15]{S22}, and which is of relevance for the work presented herein.

\begin{theorem}\label{thm:mainold} Let $\mm$ be any non-negative and finite Borel measure on $(\prob_p(\R^d), W_p)$. Then the algebra $\ccyl{\prob(\R^d)}{\rmC_b^1(\R^d)}$ is dense in $q$-energy in $H^{1,q}(\prob_p(\R^d), W_p, \mm)$, $p,q\in (1,+\infty)$. In particular, for every $F \in H^{1,q}(\prob_p(\R^d), W_p, \mm)$ there exists a sequence $(F_n)_n \subset \ccyl{\prob(\R^d)}{\rmC_b^1(\R^d)}$ such that
\[ F_n \to F \text{ in } L^q(\prob_p(\R^d), \mm), \quad \pCE_{q,\mm}(F_n) \to \CE_{q,\mm}(F) \quad \text{ as } n \to + \infty.\]
Moreover, the Banach space $H^{1,q}(\prob_p(\R^d), W_p, \mm)$ is reflexive and uniformly convex for every $p,q \in (1,+\infty)$. Finally, if $p=q=2$, then $\CE_{2,\mm}$ is a quadratic form and $H^{1,2}(\prob_2(\R^d), W_2, \mm)$ is a Hilbert space.
\end{theorem}

Since we use it below, let us also emphasize that the structure of cylinder functions and the result above allow to extend the notion of vector-valued gradient to general functions in $\Hot$. In particular, we need below the properties summarized in the following remark.

\begin{remark} \label{rem:precheeger}
For any $F,G \in \Hot$ there exist unique vector fields $\rmD_\mm F, \rmD_\mm G \in L^2(\prob_2(\R^d) \times \R^d, \bmm; \R^d)$ such that
\begin{align*}
\CE_{2,\mm}(F,G)=\int_{\prob_2(\R^d) \times \R^d} \rmD_\mm F \cdot \rmD_\mm G \de \bmm = \int_{\prob_2(\R^d)} \int_{\R^d} \rmD_\mm F(x,\mu) \cdot \rmD_\mm G (x,\mu)\de \mu(x)\de \mm(\mu),
\end{align*}
where $\bmm := \int_{\prob_2(\R^d)} \delta_\mu \de \mm(\mu)$; we refer to~\cite[Section 5]{FSS22} for the proof\nc. Notice that even for $F \in \ccyl{\prob(\R^d)}{\rmC_b^1(\R^d)}$, in general, $\rmD_\mm F \ne \rmD F$, where the latter is defined as in~\eqref{eq:thediff}\nc. This equality would correspond to the \emph{closability} of the Dirichlet form induced by the Cheeger energy, which is not guaranteed in general. However it holds that
\begin{equation}\label{eq:useeq}
    \int_{\prob_2(\R^d) \times \R^d} \rmD_\mm F \cdot \rmD_\mm G \de \bmm = \int_{\prob_2(\R^d) \times \R^d} \rmD F \cdot \rmD_\mm G \de \bmm 
\end{equation}
 for every $F \in \ccyl{\prob(\R^d)}{\rmC_b^1(\R^d)}$ and $G \in \Hot$, see \cite[Formula 5.36]{FSS22}.
\end{remark}

The above result, Theorem~\ref{thm:mainold}, has a twofold importance: from a theoretical point of view, in case $p=q=2$, the Hilbertianity property is crucial, being the starting point for a very rich theory \cite{Lott-Villani07,Sturm06I, Sturm06II, AGS14I, Gigli15-new}. As a relevant example of application, in this work we make heavy use of the Hilbertianity of $H^{1,2}(\prob_2(\R^d), W_2, \mm)$ in Section \ref{sec:eulerlagrange}, where we compute the Euler-Lagrange equations of quadratic functionals over this space. 

Moreover, again from an applied perspective, this theorem guarantees the approximability of a wide class of functions with the class of smooth cylinder functions. For instance, for any given $k \in (0,+\infty)$, it can be easily shown that the $k$-truncated $p$-Wasserstein distance from a fixed measure $\vartheta \in \prob_p(\R^d)$,
\[ \mu \mapsto W_p(\mu,\vartheta) \wedge k,\]
is an element of $H^{1,q}(\prob_p(\R^d), W_p, \mm)$ for any non-negative Borel measure $\mm$ on $\prob_p(\R^d)$; we thus know that there exists a sequence $(F_n)_n \subset \ccyl{\prob(\R^d)}{\rmC_b^1(\R^d)}$ such that 
\begin{equation}\label{eq:approx1}
 F_n \to W_p(\vartheta, \cdot) \wedge k \quad \text{and} \quad \lip_{W_p} F_n \to |\rmD (W_p(\vartheta, \cdot)\wedge k)|_{\star, q} \quad \text{ in } L^q(\prob_p(\R^d), \mm)
\end{equation}
as $n \to +\infty$. Moreover, if the measure $\mm$ is concentrated on measures in $\prob_p(\R^d)$ with the same compact support $K \subset \R^d$, \eqref{eq:approx1} holds even with $k=+\infty$.

Let us remark that the result of Theorem \ref{thm:mainold} is not constructive, in the sense that in \cite{FSS22,S22} there is no explicit form for the sequence of cylinder functions approximating a given $F \in H^{1,q}(\prob_p(\R^d), W_p, \mm)$. This is the content of the ensuing section.

\section{Approximability of Wasserstein distances by cylinder maps} \label{sec:conofpot}
This section is devoted to the explicit approximation by cylinder maps of the Wasserstein distance function from a fixed reference point. This is done first in a completely deterministic way in Section \ref{sec:31} without providing any convergence rates. These are then elaborated in Section \ref{sec:32} at the expense of determinism, employing the concept of random $\eps$-subcoverings. 

\subsection{Distribution independent approximation}\label{sec:31}
We first present a general strategy for the pointwise approximation of the $p$-Wasserstein distance through Kantorovich potentials, which is similar to the one proposed in \cite{DelloSchiavo20, FSS22}.
In particular we show that the Wasserstein distance $W_p(\cdot, \vartheta)$, for a given reference measure $\vartheta \in \prob_p(K)=\prob(K)$ with $K \subset \R^d$ being compact, can be approximated pointwise by a well-constructed uniformly bounded sequence of cylinder functions $(F_{k } )_{k} \subset \ccyl{\prob(K)}{\rmC_b^1(\R^d)}$; i.e.,~we have that
\begin{equation}\label{eq:pwc}
\lim_{k\to \infty} F_{k}(\mu) = W_p(\mu, \vartheta).
\end{equation}
As a consequence of the pointwise convergence and dominated convergence theorem, for any non-negative Borel measure $\mm$ on $\prob(K)$ and any $q \in [1,+\infty)$, we obtain the convergence in $L^q(\prob(K), \mm)$:
\begin{equation}\label{eq:Lq}
\lim_{{k}\to \infty} \int_{\prob(K)} |F_{k}(\mu) - W_p(\mu, \vartheta)|^q \de \mm(\mu)=0.
\end{equation}
\\

We start reporting a result (the proof of which can be found in \cite{S22}) which is the fundamental link between Wasserstein distance and cylinder functions. In the following, we denote by $\prob_p^r(\R^d)$ the subset of $\prob_p(\R^d)$ of \emph{regular measures}; i.e.,~those elements $\mu \in \prob_p(\R^d)$ such that $\mu \ll \mathcal{L}^d$, where $\mathcal{L}^d$ denotes the Lebesgue measure on $\R^d$.
\begin{theorem} \label{thm:ot} Let $\mu, \nu \in \prob_p^r(\R^d)$
  with $\supp{\nu} = \overline{\rmB(0,R)}$ for some $R>0$, where $\rmB(0,R)$ denotes the Euclidean ball of radius $R$ centered at the origin. There exists a unique pair of locally Lipschitz functions $\phi \in L^1(\rmB(0,R),\nu)$ and $\psi \in L^1(\R^d, \mu)$ such that
  \begin{enumerate}
      \item[(i)] $\displaystyle \phi(x)+ \psi(y) \le \frac{1}{p}|x-y|^p$ for every $(x,y) \in \rmB(0,R) \times \R^d$,
\item[(ii)]
  \begin{align}
    \label{eq:160}
    \displaystyle \phi(x) &= \inf_{y \in \R^d}  \left \{ \frac{1}{p}|x-y|^p - \psi(y)\right \}  &&\text{for every $x \in \rmB(0,R)$,}\\ \label{eq:160bis}
    \displaystyle \psi(y) &= \inf_{x \in \rmB(0,R)}  \left \{ \frac{1}{p}|x-y|^p - \phi(x)\right \}  &&\text{for every $y \in \R^d$,}
\end{align}
\item[(iii)] $\displaystyle \psi(0)=0$,
\item[(iv)] $\displaystyle \int_{\rmB(0,R)} \phi\, \de\nu + \int_{\R^d} \psi\, \de\mu =  \frac{1}{p}W_p^p(\mu, \nu)$.
\end{enumerate}
Such a unique pair is denoted by $(\Phi(\nu, \mu), \Phi^\star(\nu, \mu))$. Finally, the function $\psi:=\Phi^\star(\nu, \mu)$ satisfies the following estimates: there exists a constant $K_{p,R}$, depending only on $p$ and $R$, such that
\begin{align}
    \left |\psi(y')-\psi(y'') \right | &\le |y'-y''| 2^{p-1}(2R^{p-1} + |y'|^{p-1}+|y''|^{p-1}) \quad &&\text{ for every } y',y'' \in \R^d, \label{eq:phiest1}\\
|\psi(y)| &\le K_{p,R}(1+|y|^{p}) \quad &&\text{ for every } y \in \R^d. \label{eq:phiest2}
\end{align}
\end{theorem}

\begin{remark}\label{rem:above} There exists a constant $D_{p,R}$ depending only on $p$ and $R$ such that, if $\psi:\R^d \to \R$ is a function satisfying \eqref{eq:phiest1}, then
\begin{equation}\label{eq:w2lipphi}
\int_{\mathbb{R}^d} \psi \de(\mu'-\mu'') \le D_{p,R} W_p(\mu', \mu'')(1+\rsqm{\mu'}+\rsqm{\mu''}) \quad \text{ for every } \mu', \mu'' \in \prob_p(\R^d),    
\end{equation}
where $\rsqm{\cdot}$ is as in \eqref{eq:pmom}. Indeed, by \eqref{eq:phiest1}, if we take any optimal plan $\ggamma \in \Gamma_o(\mu', \mu'')$ as in \eqref{eq:optplan} (for the distance induced by the Euclidean norm), we have
\begin{align*}
    \int_{\R^d} \psi \de(\mu'-\mu'') &= \int_{\R^d \times \R^d} \left (\psi(y')-\psi(y'') \right ) \de\ggamma (y',y'') \\
    & \le 2^{p-1}\left ( \int_{\R^d \times \R^d} |y'-y''|^p \de\ggamma(y',y'') \right )^{1/p} \\
    &\quad \left ( \int_{\R^d \times \R^d} \left (2R^{p-1} + |y'|^{p-1}+|y''|^{p-1} \right )^{p'} \de\ggamma(y',y'') \right )^{1/p'}
\end{align*}
so that \eqref{eq:w2lipphi} follows.
\end{remark}

\begin{remark}\label{rem:below} Let $\{\psi_h\}_{h=1}^{k}$, ${k} \in \N_0$, be a set of functions satisfying \eqref{eq:phiest1} and let 
\[
G_{k}(\mu):= \max_{1 \le h \le {k}} \int_{\R^d} \psi_h \de\mu, \quad \mu \in \prob_p(\R^d).
\]
Then 
\[ |G_{k}(\mu') - G_{k}(\mu'') | \le D_{p,R} W_p(\mu', \mu'')(1+\rsqm{\mu'}+\rsqm{\mu''}) \quad \text{ for every } \mu', \mu'' \in \prob_p(\R^d),\]
where $D_{p,R}$ is the constant from Remark \ref{rem:above}. Indeed, by Remark \ref{rem:above} and the definition of $G_{k}$ we have that 
\[ \int_{\R^d} \psi_h \de\mu' - G_{k}(\mu'') \le \int_{\R^d} \psi_h \de\mu' - \int_{\R^d} \psi_h \de\mu'' \le  D_{p,R} W_p(\mu', \mu'')(1+\rsqm{\mu'}+\rsqm{\mu''}) \]
for every $1 \le h \le {k}$. Hence, passing to the maximum in $h$ we obtain that
\[ G_{k}(\mu')-G_{k}(\mu'') \le D_{p,R} W_p(\mu', \mu'')(1+\rsqm{\mu'}+\rsqm{\mu''}).\]
Reversing the roles of $\mu'$ and $\mu''$ gives the sought inequality.
\end{remark}

The main result of this subsection is a slightly more explicit and constructive version of the approximation results contained in \cite{DelloSchiavo20, FSS22}. Before stating it, we need to fix some notation.
\medskip

Let $\kappa \in \rmC_c^{\infty}(\R^d)$ be such that $\supp{\kappa}=\overline{\rmB(0,1)}$, $\kappa(x) \ge 0$ for every $x \in \R^d$, $\kappa(x)>0$ for every $x \in \rmB(0,1)$, $\int_{\R^d} \kappa \de\mathcal{L}^d=1$, and $\kappa(-x)=\kappa(x)$ for every $x \in \R^d$. Then, for any $0<\eps<1$, we consider the standard mollifiers
\[ \kappa_\eps (x) := \frac{1}{\eps^d} \kappa (x/\eps), \quad x \in \R^d.\]
 Given $\sigma \in \prbt$, $0<\eps<1$, and $R>0$, we further define
\[
\sigma_\eps := \sigma \ast \kappa_\eps, \quad 
 \hat{\sigma}_{\eps,R} := \frac{\sigma_\eps \mres \rmB(0,R) +
                      \eps^{d+p+1} \mathcal{L}^d \mres
                      \rmB(0,R)}{\sigma_\eps(\rmB(0,R)) +
                      \eps^{d+p+1} \mathcal{L}^d(\rmB(0,R))};
\]
here, $\ast$ and $\mres$ are the convolution and restriction operators, respectively. Notice that $\sigma_\eps, \hat{\sigma}_{\eps,R} \in \prob_p^r(\R^d)$ with 
\[ \supp \sigma_\eps \subset \supp \sigma + \rmB(0,\eps), \quad \supp \hat{\sigma}_{\eps,R}= \overline{\rmB(0, R)},  \quad \hat \sigma_\eps \ge \frac{\eps^{d+p+1}}{1+\eps^{d+p+1} \omega_{R,d}}\mathcal{L}^d \mres\rmB(0,R),\] where $\omega_{R,d}:=\mathcal{L}^d(\rmB(0,R))$.
We have also that $W_p(\sigma, \sigma_\eps) \le \eps \rsqm{\kappa \mathcal{L}^d}$ (see \cite[Lemma 7.1.10]{AGS08}) so that $W_p(\sigma_\eps,\sigma) \to 0$ as $\eps\to 0$. 
Moreover, if $\sigma, \sigma' \in \prbt$, we have (cf.~\cite[Lemma 5.2]{santambrogio})
\begin{equation}\label{eq:ineqconv}
    W_p(\sigma_\eps, \sigma'_\eps) \le W_p(\sigma, \sigma') \quad \text{ for every } 0<\eps<1.
\end{equation}
We can also see that $W_p(\hat{\sigma}_{\eps,1/\eps}, \sigma) \to 0$ as $\eps \to 0$, which follows from the well known fact (cf.~\cite[Proposition 7.1.5]{AGS08}) that the convergence in $W_p$ is equivalent to the convergence
\[ \int_{\R^d} \varphi \de \hat{\sigma}_{\eps,R} \to \int_{\R^d} \varphi \de \sigma\]
for every continuous function $\varphi:\R^d \to \R$ with less than $p$-growth. This condition is easily seen to be satisfied in our case, also noticing that $\sigma_\eps(\rmB(0,1/\eps)) \to 1$: indeed for every $\delta>0$ we can find $R_\delta>0$ such that $\sigma(\rmB(0, R_\delta)) > 1- \delta$ so that, if $\eps < 1/R_{\delta}$, we have
\[\sigma_\eps(\rmB(0, 1/\eps)) \ge \sigma_\eps(\rmB(0,R_\delta))\]
hence we get
\[ \liminf_{\eps \downarrow 0} \sigma_\eps(\rmB(0,1/\eps)) \ge \liminf_{\eps \downarrow 0}\sigma_\eps(\rmB(0,R_\delta)) \ge \sigma(\rmB(0, R_\delta)) > 1-\delta\]
and being $\delta>0$ arbitrary we conclude.

\medskip

In case $\sigma$ has compact support, we can also obtain an explicit rate of convergence of $\hat{\sigma}_{\eps,R}$ to $\sigma$ for a particular choice of $R$: we can take any $R>1$ such that $\supp \sigma \subset \rmB(0,R-1)$, thus obtaining $\supp \sigma_\eps \subset \rmB(0,R)$ and, in turn, $\sigma_\eps \mres \rmB(0,R)= \sigma_\eps$. Using the convexity of the Wasserstein distance and the triangle inequality, it is not difficult to see that
\begin{equation}\label{eq:somerate}
  W_p(\hat{\sigma}_{\eps,R},\sigma) \le \eps \rsqm{\kappa \mathcal{L}^d} + \eps^{d+p+1} \omega_{R,d} (\rsqm{\sigma}+ R) \le (1 +2\omega_{R,d}R) \eps.
\end{equation}

\begin{proposition}\label{prop:laprima} Let $K \subset \R^d$ be a compact set, $\vartheta \in \prob_p(K)=\prob(K)$, and $(\sigma^h)_h \subset \prob(K)$ be a dense subset in $\prob(K)$ with respect to the the p-Wasserstein distance. Moreover, let $R>1$ be such that $K \subset \rmB(0,R-1)$, and for every ${k},h \in \N$ let us consider the functions
\[ \varphi_{k}^h := \Phi(\hat{\vartheta}_{1/{k},R}, \sigma^h_{1/{k}}), \quad \psi_{k}^h := \Phi^\star(\hat{\vartheta}_{1/{k},R}, \sigma^h_{1/{k}}), \quad u_{k}^h(x):= \psi_{k}^h(x) +\int_{\R^d} \varphi_{k}^h \de\hat{\vartheta}_{1/{k},R}, \quad x \in \R^d,\]
where $\Phi$ and $\Phi^\star$ are as in Theorem \ref{thm:ot}.
We further define, for every ${k} \in \N$, the compact sets 
\[ C_{k} := [-\|u_{k}^1 \ast \kappa_{1/{k}}\|_\infty, \|u_{k}^1 \ast \kappa_{1/{k}}\|_\infty] \times \dots \times [-\|u_{k}^{k} \ast \kappa_{1/{k}}\|_\infty, \|u_{k}^{k} \ast \kappa_{1/{k}} \|_\infty] \subset \R^{k},\]
where $\|\cdot\|_\infty$ denotes the infinity norm on $K$, and, for every ${k} \in \N$, choose a smooth  approximation $\eta_{k}:\R^{k} \to \R$ of the function 
\[ (x_1, \dots, x_{k}) \mapsto \left (p \max_{1 \le h \le {k}} x_h \right )^p\]
such that
\begin{equation}\label{eq:1k}
 0 \le \left (p \max_{1 \le h \le {k}} x_h \right )^{1/p} - \eta_{k}(x_1, \dots, x_{k}) \le \frac{1}{{k}} \quad \text{ for every } (x_1, \dots, x_{k}) \in C_{k}. 
\end{equation}
Then, the cylinder map
\[ F_{k}(\vartheta, \mu) := \eta_{k} \left ( \int_{\R^d} (u_{k}^1 \ast \kappa_{1/{k}}) \de\mu , \int_{\R^d} (u_{k}^2 \ast \kappa_{1/{k}}) \de\mu, \dots,  \int_{\R^d} (u_{k}^{k} \ast \kappa_{1/{k}}) \de\mu \right ), \quad \mu \in \prob(K),\]
converges pointwise from below to $W_p(\vartheta, \cdot)$ in $\prob(K)$.
\end{proposition}

\begin{proof} Let us define 
\[ G_{k}(\mu) := \max_{1 \le h \le {k}} \int_{\R^d} u_{k}^h \de\mu_{1/{k}}, \quad  \mu \in \prob(K)\]
and observe that, thanks to Theorem~\ref{thm:ot},
\[ G_{k}(\sigma^h) = \frac{1}{p}W_p^p(\hat{\vartheta}_{1/{k},R}, \sigma^h_{1/{k}}) \quad \text{ for every } 1 \le h \le {k}.\]
In contrast, for the cylinder function $F_{k}$ we only have that $F_{k} \le W_p(\cdot, \vartheta)$. Let us fix an arbitrary $\mu \in \prob(K)$ and ${k} \in \N$; by employing the triangle inequality, we find that
\begin{align*}
|F_{k}(\vartheta, \mu) - W_p(\mu, \vartheta) |
&\le \left | F_{k}(\vartheta, \mu) - \left ( p \max_{1 \le h \le {k}} \int_{\R^d} (u^h_{k} \ast \kappa_{1/{k}}) \de\mu \right )^{1/p} \right |  \\
& \quad + \left | (pG_{k}(\mu))^{1/p} - (p G_{k}(\sigma^h))^{1/p} \right | + \left | W_p(\hat{\vartheta}_{1/{k},R}, \sigma^h_{1/{k}}) - W_p(\hat{\vartheta}_{1/{k},R}, \mu_{1/{k}}) \right | \\
& \quad + \left | W_p(\hat{\vartheta}_{1/{k},R}, \mu_{1/{k}}) - W_p(\hat{\vartheta}_{1/{k},R}, \mu) \right | + \left | W_p(\hat{\vartheta}_{1/{k},R}, \mu) - W_p(\vartheta, \mu) \right |. 
\end{align*}
Invoking~\eqref{eq:1k} and~\eqref{eq:ineqconv}, together with the triangle inequality, further leads to
\begin{align*}
|F_{k}(\vartheta, \mu) - W_p(\mu, \vartheta) |
& \le 1/{k} + p^{1/p}\left |G_{k}(\mu) - G_{k}(\sigma^h) \right |^{1/p} + W_p(\sigma^h, \mu) \\
& \quad + W_p(\mu_{1/{k}}, \mu) + W_p(\hat{\vartheta}_{1/{k},R}, \vartheta).
\end{align*}
Hence, by Remark~\ref{rem:below}, we arrive at
\begin{align*}
|F_{k}(\vartheta, \mu) - W_p(\mu, \vartheta) | & \le 1/{k} +  p^{1/p}D_{p,R}^{1/p} W_p^{1/p}(\mu, \sigma^h)(1+\rsqm{\mu}+ \rsqm{\sigma^h} )^{1/p} \\
& \quad + W_p(\sigma^h, \mu) + W_p(\mu_{1/{k}}, \mu) + W_p(\hat{\vartheta}_{1/{k},R}, \vartheta).
\end{align*}
Passing first to the limit as ${k} \to \infty$ and then taking the infimum with respect to $h \in \mathbb{N}$ implies the claim.
\end{proof}
\begin{remark} 
For the sake of simplicity, consistency, and computational complexity in the numerical experiments, in the rest of the paper we keep focusing on the study of  approximations  of $\mu \mapsto W_p(\mu, \vartheta)$, for  fixed $\vartheta$. Nevertheless, we point out that the same approach as we employed in Proposition \ref{prop:laprima} could be used to provide approximations of the Wasserstein distance as a function of two variables, namely,~it can be verified that, if $K \subset \R^d$ is compact, then there exists a sequence $(F_{k})_{k}$ of the form
    \[ F_{k}(\mu, \nu )= \eta_{k} \left ( \int_{\R^d} \varphi_{k}^1 \de \mu + \int_{\R^d} \psi^1_{k} \de \nu, \dots,    \int_{\R^d} \varphi_{k}^{k} \de \mu + \int_{\R^d} \psi^{k}_{k} \de \nu\right), \quad (\mu, \nu) \in \prob(K)\times \prob(K) \]
    for smooth functions $\eta_{k}, \varphi^i_{k}, \psi^i_{k}$, such that $F_{k} \to W_p$ pointwise on $\prob(K)\times \prob(K)$. 
The approximation results of the next sections can also be adapted for providing approximations of $(\mu, \nu)\mapsto W_p(\mu, \nu)$.
\end{remark} 
On the one hand, the pointwise convergence \eqref{eq:pwc} is a deterministic and universal result, because it renders \eqref{eq:Lq} valid for any suitable $\mm$ and it does not depend on $\mm$. On the other hand, it does not provide any rate of convergence. In the following we shall trade  deterministic and universal approximation with more control on the rate of convergence. For that we need to introduce the concept of random subcovering.

\subsection{A probabilistic approach based on random subcoverings}\label{sec:32}

In this section we introduce random subcoverings in the general context of a complete and separable metric space $(\SS,\sfd)$ endowed with a Borel probability measure $\pp$ (we use the notation $\pp$ to distinguish it from a general non-negative Borel measure $\mm$). We say in short that $(\SS, \sfd, \pp)$ is a Polish metric-probability space. We often deal with points $X_1, \dots, X_{k}$, ${k} \in \N$, in $\SS$ drawn randomly according to $\pp$, meaning that we have fixed a probability space $(\Omega, \mathcal{E}, \PP)$ and $(X_i)_\sharp \PP=\pp$ for $i=1, \dots, {k}$, where $\sharp$ denotes the push forward operator. We recall that, for measurable spaces $(E_1, \mathcal{E}_1)$ and $(E_2, \mathcal{E}_2)$, if $\mu$ is a measure on ($E_1, \mathcal{E}_1)$ and $f:E_1 \to E_2$ is $\mathcal{E}_1-\mathcal{E}_2$ measurable, the measure $f_\sharp \mu$ on $(E_2, \mathcal{E}_2)$ is defined as
\[ (f_\sharp \mu)(B) := \mu (f^{-1}(B)), \quad B \in \mathcal{E}_2.\]
We also use the notation $\E$ to denote the expected value on $(\Omega, \EE, \PP)$; i.e.,~the integral w.r.t.~the probability measure $\PP$.
\begin{definition}
We introduce the following definitions.
\begin{itemize}
\item For $\varepsilon>0$ and ${k}\in \N$, we define a \emph{random $\varepsilon$-subcovering of size ${k}$ of $(\SS, \sfd, \pp)$} as a finite collection of ${k}$ radius $\varepsilon$-balls  $\CC=\{B(X_i,\varepsilon):i=1,\dots,{k}\}$, whose centers  $X_1,\dots,X_{k}$ are drawn i.i.d. at random according to $\pp$. 
\item For $\varepsilon>0$ and ${k}\in \N$, we call the probability
$$
p_{\varepsilon,{k}}:=\mathbb{P} \bigg ( X \in \cup_{i=1}^{k} B(X_i,\varepsilon)\bigg ),
$$
where $X, X_1, \dots, X_{k}$ are drawn i.i.d.~according to $\pp$, the \emph{$(\varepsilon,{k})$-subcovering probability}\nc .
\item Finally, for $\delta,\varepsilon>0$ we define the  $\varepsilon$-subcovering number $N_{\delta,\varepsilon}(\SS, \sfd, \pp)$ as the minimal size of random $(\varepsilon,{k})$-subcoverings to cover $\SS$ with probability $1-\delta$; i.e., 
\[ N_{\delta,\varepsilon}(\SS, \sfd, \pp):= \min \left \{ {k} \in \N : p_{\varepsilon,{k}} \geq 1- \delta \right \}.\]
\end{itemize}
\end{definition}
Notice that, while for a non-compact metric space $\SS$ the classical deterministic covering number is not finite, the $(\delta,\varepsilon)$-subcovering number may be instead finite and small. For instance, if the support of the measure $\pp$ is itself contained in a ball $B(\bar x,\varepsilon/2)$, then obviously $p_{\varepsilon,{k}}=1$ for all ${k} \in \N$, and therefore $N_{\delta,\varepsilon}(\SS)=1$. The $\eps$-subcovering probabilities and relative $(\delta,\eps)$-subcovering number essentially measure how locally concentrated the measure $\pp$ is. To some extent it is a form of quantification of the entropy of the measure.
Let us now compute explicitly the $\varepsilon$-subcovering probabilities $p_{\varepsilon,{k}}$ and estimate the $(\delta,\eps)$-subcovering number, respectively.

\begin{proposition} Let $(\SS, \sfd, \pp)$ be a Polish metric-probability space.
If ${k} \in \N$ and $\eps>0$,
then we have
\begin{equation}\label{eq:covprob} p_{\varepsilon,{k}} = 1 - \int_{\SS} (1- \pp(B(x,\varepsilon))^{k} \d\pp(x) =1 - \mathbb E[ \pp(\cap_{i=1}^{k} B^c(X_i,\varepsilon))].
\end{equation}
As a consequence the $(\delta,\varepsilon)$-subcovering number can be estimated from above as
\begin{equation}\label{eq:covnum}
N_{\delta,\varepsilon}(\SS,\sfd, \pp) \leq \bigg  \lceil{\frac{\log(\delta)}{\int_{\SS} \log(\pp(B^c(x,\varepsilon)) \de\pp(x)}} \bigg\rceil.
\end{equation}
\end{proposition}

\begin{proof}
Let us first consider
$$
p_{\varepsilon,{k}} = 1- \mathbb{P} \biggl ( X \in \bigcap_{i=1}^{k}\rmB^c(X_i, \eps) \biggr ),
$$
and compute  
\begin{align*}
\mathbb{P} \biggl ( X \in \bigcap_{i=1}^{k}\rmB^c(X_i, \eps) \biggr ) &= \int_{\SS^{{k}+1}} \nachi_{\{|x-x_i| \ge \eps \, \forall i=1, \dots, {k}\}}(x, x_1, \dots, x_{k}) \de (\otimes^{{k}+1} \pp)(x,x_1, \dots, x_{k})\\
&= \int_\SS \int_{\SS^{{k}}}\nachi_{\{|x-x_i| \ge \eps \, \forall i=1, \dots, {k}\}}(x, x_1, \dots, x_{k}) \de (\otimes^{k} \pp)(x_1, \dots, x_{k}) \de \pp(x)\\
&= \int_\SS \prod_{i=1}^{k} \mathbb{P}(X_i \in B^c(x,\eps)) \de \pp(x) \\
&= \int_\SS (1-\pp(\rmB(x,\eps)))^{k} \de \pp(x),
\end{align*}
where we have denoted by $\otimes^{m} \pp$ the product measure on $\SS^m$ obtained multiplying $m$ copies of $\pp$, for $m={k}, {k}+1$. Thus \eqref{eq:covprob} is verified. In turn, this further shows that $p_{\varepsilon,{k}} \geq 1- \delta$ if and only if
$$
\int_{\SS} ( 1- \pp(\rmB(x, \eps))^{k} \de\pp(x)=\mathbb{P} \biggl ( X \in \bigcap_{i=1}^{k}\rmB^c(X_i, \eps) \biggr ) \leq \delta.
$$
By applying the logarithm on both sides and using Jensen's inequality we obtain
$${k} \geq \frac{\log(\delta)}{\int_{\SS} \log(\pp(B^c(x,\varepsilon)) \de\pp(x)}.
$$
Hence, one deduces \eqref{eq:covnum}.
\end{proof}
\begin{lemma}\label{lem:expconv} Let $(\SS, \sfd, \pp)$ be a Polish metric-probability space.
We have that $0\leq p_{\varepsilon,{k}}\leq 1$, $p_{\varepsilon,0}=0$,  $p_{\varepsilon,1}=\int_{\SS} \pp(B(x,\varepsilon)) \de \pp(x)$,
and $\lim_{{k}\to \infty} p_{\varepsilon,{k}}=1$. Moreover, $\supp\pp$ is compact if and only if 
\begin{equation}\label{eq:cond}
\underline{\pp}_\eps:= \inf_{x \in \supp \pp } \pp(\rmB(x,\eps)) >0 \quad \text{ for every } \eps >0.
\end{equation}
 In this case, we have that
\begin{equation}\label{eq:rate}
p_{\varepsilon,{k}} \geq 1- (1-\underline {\pp}_\eps)^{k},
\end{equation}
and $p_{\varepsilon,{k}}$ converges exponentially fast to $1$ for ${k}\to \infty$. 
\end{lemma}

\begin{proof}
If $x \in \supp \pp$, then, for any $\varepsilon>0$, we have that $\pp(B(x,\varepsilon))>0$ and in turn \[\lim_{{k}\to \infty} (1-\pp(B(x,\varepsilon)))^{k}=0.\]
Consequently, by~\eqref{eq:covprob} and the dominated convergence theorem we deduce that $\lim_{{k}\to \infty} p_{\varepsilon,{k}}=1$. Assume now that $\supp \pp$ is compact, then \eqref{eq:cond} is satisfied because of the lower semicontinuity of the map $x \mapsto \pp(\rmB(x,\eps))$. On the other hand, we prove that if $\supp \pp$ is not compact, then \eqref{eq:cond} is always violated. So let us assume that $\supp \pp$ is not compact and $\underline{\pp}_\eps>0$. Since $\supp \pp$ is not compact there exists some $\eps>0$ such that any finite union of balls with radii equal to $\eps$ does not cover $\supp \pp$. Moreover, since $\pp$ is tight, we can find a compact $K \subset \supp \pp$ such that $\pp(\supp \pp\setminus K)<\underline {\pp}_{\eps/2}$. In turn, by the compactness of $K$, there exists an integer ${M} \in \N$ and points $x_1, \dots, x_{M} \in K$ such that
\[ K \subset \bigcup_{i=1}^{M} \rmB(x_i, \eps/2).\]
Let $x \in \supp \pp \setminus \bigcup_{i=1}^{M} \rmB(x_i, \eps)$. Then $\sfd(x,x_i)\ge \eps$ for every $i=1, \dots, {M}$ and thus $\rmB(x,\eps/2) \cap K = \emptyset$. But this means that $\pp(\rmB(x,\eps/2)) < \underline {\pp}_{\eps/2}$, a contradiction.
\end{proof}
\begin{remark}
We observe that it may happen that $\underline {\pp}_\eps \to 0$ very fast for $\eps \to 0$. Hence, the rate of convergence \eqref{eq:rate} has a meaning as long as $\eps>0$ is fixed and not too small. In some applications, as the ones we consider below, one may not need to reach arbitrary small accuracy, but be content with a sufficiently small approximation.
\end{remark}
The following result shows that $(\eps,{k})$-subcovering probabilities depend continuously on the distribution.

 \begin{lemma}
 Let $(\SS, \sfd, \pp)$ be a Polish metric-probability space. Assume that $\pp^{m} \ll \pp$, i.e., $\pp^{m}$ is absolutely continuous with respect to $\pp$, and $\rho_{m}:=\frac{\de\pp^{m}}{\de\pp} \to 1$ in $L^1(\SS,\pp)$ as ${m}\to \infty$. Then
 \begin{equation}\label{eq:unif}
 \lim_{{m} \to \infty} \int_\SS \varphi(x) \de\pp^{m}(x)   = \int_\SS \varphi(x) \de\pp(x),
 \end{equation}
 for any $\varphi \in L^\infty(\SS)$. Hence, $\pp^{m}$ converges narrowly to $\pp$. Additionally, for any fixed $\eps>0$ and ${k} \in \N$ we have 
 \begin{align} \label{eq:penconv}
 \lim_{{m} \to \infty} p_{\eps,{k}}(\pp^{{m}}) = p_{\eps,{k}}(\pp).
 \end{align}
 \end{lemma}
\begin{proof}
The first limit follows from
$$
\bigg |\int_\SS \varphi(x) (\rho_{m}(x) - 1)\de\pp(x) \bigg| \leq \|\varphi\|_\infty \int_\SS  |\rho_{m}(x) - 1|\de\pp(x) \to 0, \quad {m}\to \infty.
$$
Concerning~\eqref{eq:penconv}, employing~\eqref{eq:covprob} and the triangle inequality yields that
\begin{align*}
|p_{\eps,{k}}(\pp^{m}) -p_{\eps,{k}}(\pp) | &\leq  \int_\SS \bigg (\int_\SS \chi_{B(x,\eps)^c}(y)\rho_{m}(y) \d\pp(y)\bigg )^{k} |\rho_{m}(x) - 1| \d\pp(x) \\
& \quad +\int_\SS \bigg | \bigg (\int_\SS \chi_{B(x,\eps)^c}(y)\rho_{m}(y) \d\pp(y)\bigg )^{k} - \pp(B(x,\eps)^c)^{k}  \bigg| \d\pp(x) \\
&\leq \| \rho_{m} - 1\|_{L^1(\SS,\pp)} \\
& \quad + \int_\SS \bigg | \bigg (\int_\SS \chi_{B(x,\eps)^c}(y)\rho_{m}(y) \d\pp(y)\bigg )^{k} - \pp(B(x,\eps)^c)^{k}  \bigg| \d\pp(x).
\end{align*}
The first term converges to $0$ as ${m}\to \infty$ by assumption. Moreover, the second term vanishes as well, as ${m}\to \infty$, by the pointwise convergence ensured by \eqref{eq:unif} and the dominated convergence theorem.
\end{proof}

We present now some applications of  $(\eps,{k})$-subcovering probabilities in estimating integrals by means of empirical measures. In the ensuing results we presume that it is possible to estimate the measure $\pp$ on suitable sets $C_i$ for the sake of allocating suitable weights of quadrature formulas. 
Below we consider the space of bounded Lipschitz continuous functions $\Lip_b(\SS, \sfd)$
with norm $\|\varphi\|_{BL} = \|\varphi\|_\infty + \Lip(\varphi, \SS)$, where $\Lip(\varphi, \SS)$ is the Lipschitz constant of $\varphi$ (cf.~\eqref{eq:thelip}), and its dual distance
$$
\d_{BL}(\pp,\pp') =\sup_{\varphi \in \Lip_b(\SS,\sfd), \|\varphi\|_{BL}\leq 1} \bigg | \int_\SS \varphi(x) \de\pp(x) - \int_\SS \varphi(x) \de\pp'(x) \bigg |.
$$

\begin{lemma}\label{lem:matching1} Let $(\SS, \sfd, \pp)$ be a Polish metric-probability space.
Given a random $(\varepsilon,{k})$-subcovering $\CC=\{B(X_i,\varepsilon):i=1,\dots,{k}\}$ of $\SS$, we consider disjoint sets $C_i \subset B(X_i,\varepsilon)$, $i=1,\dots,{k}$, such that $\cup_{i=1}^{k} C_i=\cup_{i=1}^{k} B(X_i,\varepsilon)$ and we fix the empirical measure
\begin{equation}\label{eq:empir}
\pp^{{k}, \eps} = \frac{1}{\sum_{i=1}^{k} \pp(C_i)}
\sum_{i=1}^{k} \pp(C_i) \delta_{X_i}.
\end{equation}
Then, we have that
\begin{equation}\label{eq:BLdistest}
\mathbb E \left [ \d_{BL}(\pp,\pp^{{k}, \eps})\right ]  \leq \varepsilon p_{\eps,{k}} + 2 (1-p_{\eps,{k}}).
\end{equation}
\end{lemma}

\begin{proof}
For any $x \in \cup_{i=1}^{k} B(X_i,\varepsilon)$ we associate the unique $i=i(x)$ such that $x \in C_i$ and define $X^\star(x):=X_{i(x)}$. With this notation we may write
\begin{align} \label{eq:terms}
\begin{split}
\int_{\SS} \varphi(x) \de\pp(x) &= \int_{\cup_{i=1}^{k} B(X_i,\varepsilon)} \varphi(x) \de\pp(x) + \int_{\cap_{i=1}^{k} B^c(X_i,\varepsilon)} \varphi(x) \de\pp(x) \\
&= \int_{\cup_{i=1}^{k} B(X_i,\varepsilon)} \bigg (\varphi(x) -  \varphi(X^\star(x)) \bigg) \de\pp(x) + \sum_{i=1}^{k} \pp(C_i) \varphi(X_i) \\
& \quad + \int_{\cap_{i=1}^{k} B^c(X_i,\varepsilon)} \varphi(x) \de\pp(x). 
\end{split}
\end{align}
Since the sets $C_i$ are disjoint, we have that
$$
\sum_{i=1}^{k} \pp(C_i) = \pp(\cup_{i=1}^{k} C_i) = \pp(\cup_{i=1}^{k} B(X_i,\varepsilon)) =1-\pp(\cap_{i=1}^{k} B(X_i,\eps)^c).
$$
Using this property as well as the Lipschitz continuity and the boundedness of $\varphi$, we obtain that
\begin{multline*}
\bigg | \int_{\SS} \varphi(x) \de\pp(x)-\int_{\SS} \varphi(x) \de\pp^{{k},\eps}(x) \bigg | \\
\begin{aligned}
&\leq \eps\Lip(\varphi)  \pp(\cup_{i=1}^{k} B(X_i,\varepsilon))+\frac{(1- \sum_{i=1}^{k} \pp(C_i))}{\sum_{i=1}^{k} \pp(C_i)} \sum_{i=1}^{k} \pp(C_i) |\varphi(X_i)| +  \|\varphi\|_\infty \pp(\cap_{i=1}^{k} B(X_i,\varepsilon)^c) \\
&\leq \eps\Lip(\varphi)  \pp(\cup_{i=1}^{k} B(X_i,\varepsilon))+\frac{\pp(\cap_{i=1}^{k} B(X_i,\varepsilon)^c) }{\sum_{i=1}^{k} \pp(C_i)} \sum_{i=1}^{k} \pp(C_i) |\varphi(X_i)| +  \|\varphi\|_\infty \pp(\cap_{i=1}^{k} B(X_i,\varepsilon)^c) \\
&\leq \eps\Lip(\varphi)  \pp(\cup_{i=1}^{k} B(X_i,\varepsilon))+ 2\|\varphi\|_\infty \pp(\cap_{i=1}^{k} B(X_i,\varepsilon)^c).
\end{aligned}
\end{multline*}
Hence, we conclude that
$$
\sup_{\varphi \in \Lip(\SS),  \|\varphi\|_{BL}\leq 1} \bigg | \int_{\SS} \varphi(x) \de\pp(x) - \int_{\SS} \varphi(x) \de\pp^{{k},\eps}(x) \bigg | \leq \eps  \pp(\cup_{i=1}^{k} B(X_i,\varepsilon))+ 2 \pp(\cap_{i=1}^{k} B(X_i,\varepsilon)^c),
$$
and passing it to the expectation yields \eqref{eq:BLdistest}.
\end{proof}

The bound in \eqref{eq:BLdistest} guarantees that for every $\varepsilon>0$, there exists an integer ${k}_0$ such that, for all ${k}\geq {k}_0$, 
$$
\mathbb E [ \d_{BL}(\pp,\pp^{{k},\eps})] \leq 3 \eps.
$$
Moreover, provided that $\pp$ is compactly supported, Lemma~\ref{lem:expconv} provides that we can choose \[{k}_0 \geq  \lceil \log(\eps)/\log(1-\underline{\pp}_\eps) \rceil.\] Furthermore, we obtain the same result, up to a constant factor, with respect to the $p$-Wasserstein distance, in case $\pp$ is compactly supported.

\begin{corollary}\label{cor:simconv} Assume, in addition to the assumptions of Lemma \ref{lem:matching1}, that $\pp$ is compactly supported. Then 
\begin{equation}
    \mathbb E\bigg[ W_p(\pp,\pp^{{k}, \eps}) \bigg]  \le C_{\pp, \eps} \left (\varepsilon p_{\eps,{k}} + 2 (1-p_{\eps,{k}}) \right ),
\end{equation}
where $C_{\pp, \eps}$ is given by 
\[ C_{\pp, \eps} := 2(\diam(\supp \pp)+2\eps)^{\frac{p-1}{p}}(\diam(\supp \pp)+2\eps+1).\]
\end{corollary}

\begin{proof}
Let $K:=\{x \in \R^d : \text{dist}(x,\supp \pp) \le \eps\}$. Since $\pp$ is compactly supported, $K$ is compact as well with 
\begin{align} \label{eq:est1}
\diam(K)\le \diam(\supp \pp)+2\eps.
\end{align}
Observe that with probability $1$ both $\pp$ and $\pp^{{k},\eps}$ belong to $\prob(K)$.
Since $K$ is compact, all the $p$-Wasserstein distances are equivalent to the $1$-Wasserstein distance on $\mathcal{P}(K)$: indeed, it is easy to check that
\[ W_p(\mu, \nu) \le \diam(K)^{\frac{p-1}{p}} W_1(\mu,\nu) \quad \text{ for every } \mu, \nu \in \prob(K) \text{ and every } p \in [1,+\infty).\]
The Kantorovich-Rubinstein duality theorem (see, e.g.,~\cite[Rem.~6.5]{Villani:09}) states that
\begin{align} \label{eq:wpineq} W_1(\mu, \nu) = \sup \left \{ \int \varphi\, \de(\mu-\nu) : \varphi \in \rmC_b(K), \, \Lip(\varphi) \le 1 \right \} =:\d_{KR}(\mu, \nu)
\end{align}
for every $\mu, \nu \in \prob(K)$; moreover, it is not difficult to see that 
\begin{align} \label{eq:ineq3} \d_{KR}(\mu, \nu) \le 2(1+\diam(K)) \d_{BL}(\mu, \nu)
\end{align} 
for every $\mu, \nu \in \prob(K)$. Finally, combining the above inequalities~\eqref{eq:est1},~\eqref{eq:wpineq}, and~\eqref{eq:ineq3} with the inequality \eqref{eq:BLdistest} from Lemma~\ref{lem:matching1} yields the desired bound.
\end{proof}

The main issue with Lemma \ref{lem:matching1} and Corollary \ref{cor:simconv} is that the empirical measure used for estimating the integration does require the possibility of accessing and evaluating the entire measure $\pp$ on suitable subsets $C_i$ of random balls. Hence, in general, it won't be practicable to use such results for estimating $\d_{BL}$ or $W_p^p$, unless one considers a further density estimation on such random balls.
For this reason we leave it at this point and we consider a further useful result, which does not require to estimate the measure on sets. This simple result applies to a very special type of empirical functions, depending on the distance of the argument to an empirical realization of the measure $\pp$.

\begin{lemma}\label{lem:matching3}
Let $(\SS, \sfd, \pp)$ be a Polish metric-probability space and consider a function $L:\R_+ \to \R_+$ that is continuous, bounded, and satisfies 
$\lim_{r \downarrow 0} L(r) =0$. 
Then, for all $\delta>0$ there exists $\eps>0$ such that
\begin{equation}\label{eq:expL}
   \mathbb E\bigg[ \int_{\SS} L \left(\min_{i=1,\dots,k} \sfd(x,X_i)\right) \de\pp(x) \bigg ]\leq  \delta p_{\eps,{k}} + \|L\|_\infty (1- p_{\eps,{k}}),
\end{equation}
 where $X_1, \dots, X_{k}$ are i.i.d.~random variables on $\SS$ distributed according to $\pp$ and ${k} \in \N$.
\end{lemma}

\begin{proof}
For any $\delta>0$ there exists $\eps>0$ such that $L(r) \leq \delta$ for any $0\le r\leq \eps$. Let $X_1, \dots, X_{k}$ be i.i.d.~random variables on $\SS$ distributed according to $\pp$ and let us consider a random $(\varepsilon,{k})$-subcovering $\CC=\{B(X_i,\varepsilon):i=1,\dots,{k}\}$ of $\SS$. We estimate
\begin{align*}
\int_{\SS} L \left(\min_{i=1,\dots,{k}} \sfd(x,X_i)\right) \de\pp(x) &= \int_{\cup_{i=1}^{k} B(X_i,\varepsilon)} L \left(\min_{i=1,\dots,{k}} \sfd(x,X_i)\right) \de\pp(x)\\
& \quad + \int_{\cap_{i=1}^{k} B(X_i,\varepsilon)^c} L \left(\min_{i=1,\dots,{k}} \sfd(x,X_i)\right) \de\pp(x)\\
&\leq  \delta \pp(\cup_{i=1}^{k} B(X_i,\varepsilon)) + \|L\|_\infty \pp(\cap_{i=1}^{k} B^c(X_i,\varepsilon)).
\end{align*}
By taking the expectation we obtain the bound \eqref{eq:expL}.
\end{proof}
Inspired by the latter result, we come back now to the more concrete situation in which $(\SS, \sfd)=(\prob(K),W_p)$ for a compact subset $K \subset \R^d$ and we conclude this subsection with a random version of Proposition~\ref{prop:laprima}. Its proof actually follows implicitly a similar argument as the one of Lemma \ref{lem:matching3}. 

\begin{proposition}\label{prop:seconda} Let ${k} \in \N$ and $K \subset \R^d$ be a compact set. Consider $\mu_1, \dots, \mu_{k}$ i.i.d.~random variables on $\prob_p(K)=\prob(K)$ distributed according to $\pp \in \prob(\prob(K))$ and let $\vartheta \in \prob(K)$ be fixed. Define (the random) function $F_{k}(\vartheta, \cdot): \prob(K) \to \R$ as in Proposition \ref{prop:laprima} with $(\mu_i)_{i=1}^{k}$ in place of $(\sigma^h)_{h \in \N}$. Then
\begin{equation}\label{eq:precest}
 \E \left [ \int_{\prob(K)} |F_{k}(\vartheta,\mu)-W_p(\vartheta, \mu)|^q \,\de\pp(\mu) \right ] \le C(p,q,K)\left [ (1-p_{\eps,{k}})+p_{\eps,{k}}({k}^{-q}+\eps^{q})\right ],    
\end{equation}
where $q \in [1,+\infty)$, $\eps>0$, and $C(p,q,K)$ is a constant depending only on $p,q$ and $K$.
\end{proposition}

\begin{proof} The proof is based on the one of Proposition~\ref{prop:laprima} with some changes due to the present probabilistic setting.
We rewrite the expected value above in a convenient form: for any $\delta>0$ we have that
\begin{multline*}
\E \left [ \int_{\prob(K)} |F_{k}(\vartheta,\mu)-W_p(\vartheta, \mu)|^q \de \pp(\mu)\right ] \\
\begin{aligned}
&=\int_{(\prob(K))^{k}} \int_{\prob(K)} |F_{k}(\vartheta,\mu)-W_p(\vartheta, \mu)|^q \,\de\pp(\mu) \,\de(\otimes^{k} \pp) (\mu_1, \dots, \mu_{k})\\
&=\int_{(\prob(K))^{k}} \int_{\bigcap_{i=1}^{k} \rmB(\mu_i, \delta)^c} |F_{k}(\vartheta,\mu)-W_p(\vartheta, \mu)|^q \,\de\pp(\mu) \,\de(\otimes^{k} \pp) (\mu_1, \dots, \mu_{k}) \\
&\quad +\int_{(\prob(K))^{k}} \int_{\bigcup_{i=1}^{k} \rmB(\mu_i, \delta)} |F_{k}(\vartheta,\mu)-W_p(\vartheta, \mu)|^q \,\de\pp(\mu) \,\de(\otimes^{k} \pp) (\mu_1, \dots, \mu_{k}).
 \end{aligned}
 \end{multline*}
Recall that we can estimate the Wasserstein distance in terms of the diameter of $K$ and that $F_{k}(\vartheta,\mu)$ is bounded from above by $W_p(\vartheta,\mu)$. Hence, there exists a constant $C(q,p,K)$ depending only on $q,p$ and $K$ such that
\begin{align*}
|F_{k}(\vartheta,\mu)-W_p(\vartheta,\mu)|^q \leq C(q,p,K) \qquad \text{ for every } \mu \in \mathcal{P}(K).
\end{align*}
This leads to
 \begin{multline*}
\E \left [ \int_{\prob(K)} |F_{k}(\vartheta,\mu)-W_p(\vartheta, \mu)|^q \de \pp(\mu)\right ] \\
\begin{aligned}
     & \le C(q,p,K) \int_{(\prob(K))^{k}} \pp \left ( \bigcap_{i=1}^{k} \rmB(\mu_i, \delta)^c\right ) \,\de(\otimes^{k} \pp) (\mu_1, \dots, \mu_{k}) \\
     & \quad + \int_{(\prob(K))^{k}} \int_{\bigcup_{i=1}^{k} \rmB(\mu_i, \delta)} |F_{k}(\vartheta,\mu)-W_p(\vartheta, \mu)|^q \,\de\pp(\mu) \,\de(\otimes^{k} \pp) (\mu_1, \dots, \mu_{k})\\
     & \le  C(q,p,K)(1-p_{\delta, {k}}) \\
     & \quad + \int_{(\prob(K))^{k}} \int_{\bigcup_{i=1}^{k} \rmB(\mu_i, \delta)} |F_{k}(\vartheta,\mu)-W_p(\vartheta, \mu)|^q \,\de\pp(\mu) \,\de(\otimes^{k} \pp) (\mu_1, \dots, \mu_{k}).
\end{aligned}
\end{multline*}
We can assume without loss of generality that $\eps \in (0,1)$ and set $\delta:=\eps^p$ so that $\delta<\eps$; we now bound the second integral: for fixed $\mu_1, \dots, \mu_{k} \in \mathcal{P}(K)$, let $\mu \in \bigcup_{i=1}^{k} \rmB(\mu_i, \delta)$. We can thus find $i_0=i_0(\mu) \in \{1, \dots, {k}\}$ such that $W_p(\mu, \mu_{i_0})\le \eps^p $. Proceeding along the lines of the proof of Proposition \ref{prop:laprima} we obtain
\begin{align*}
    |F_{k}(\vartheta,\mu)-W_p(\vartheta, \mu)| & \le  \frac{1}{{k}} + p^{1/p}D_{p,R}^{1/p} W_p^{1/p}(\mu, \mu_{i_0})(1+\rsqm{\mu}+ \rsqm{\mu_{i_0}} )^{1/p} \\
& \quad + W_p(\mu_{i_0}, \mu) + W_p(\mu_{1/{k}}, \mu) + W_p(\hat{\vartheta}_{1/{k},R}, \vartheta),
\end{align*}
where we have used the notation introduced before the proof of Proposition \ref{prop:laprima} and $R>0$ is such that $K \subset \rmB(0,R-1)$.
Using  that $W_p(\mu_{1/{k}},\mu) \leq \rsqm{\kappa\mathcal{L}^d}/{k}$, see, e.g.,~\cite[Lemma 7.1.10]{AGS08}, and \eqref{eq:somerate} we further obtain that
\begin{align*}
 |F_{k}(\vartheta,\mu)-W_p(\vartheta, \mu)| & \le \frac{1}{{k}} + C(p,K,)\eps + \eps + \frac{C(p,K)}{{k}} \le C'(p,K)(\eps+1/{k}),
\end{align*}
where we have used that in \eqref{eq:somerate} $R$ depends only on $K$.
This estimate doesn't depend on $\mu, \mu_1, \dots, \mu_{k}$ so that we can integrate the inequality and conclude the proof.
\end{proof}

In contrast to Proposition \ref{prop:laprima}, the latter result shows that we can construct, from a finite number of samples of the distribution $\pp$, a cylinder function that approximates the Wasserstein distance in expectation. While the bound is not deterministic, we obtain a rate of convergence depending on the $\eps$-subcovering probabilities relative to $\pp$. The more such measure is locally concentrated, the tighter is the bound.


\section{Consistency of empirical risk minimization}\label{sec:cons}

In order to extend the approximation results of Section \ref{sec:31} to more general functions than the Wasserstein distance from a reference measure, we study in this section the construction of approximants by empirical risk minimization. We also introduce a regularizing term using the pre-Cheeger energy as additional term to cope with noisy data.

In Section \ref{sec:cons1}, we prove a completely deterministic result, which however lacks of a convergence rate and does not offer yet a proper analysis for data corrupted by noise.

In Section \ref{sec:cohen} we discuss a different approach (based on \cite{cohen}), which leads to an analogous kind of convergence in mean (i.e., non-deterministic) gaining an explicit order of convergence also for noisy data.

\subsection{Gamma convergence with approximating measures} \label{sec:cons1}
In this section, in order to keep things simpler, we choose the convenient exponent $p=2$ for the order of the Wasserstein space. We consider again a compact subset $K \subset \R^d$ and we thus work on the space $\W_2:=(\prob_2(K), W_2, \mm)= (\prob(K), W_2, \mm)$, where $\mm$ is a non-negative and finite Borel measure on $\prob(K)$. Recall that, since $K$ is compact, also $\prob(K)$ is compact and $W_2$ metrizes the weak convergence; see, e.g., \cite[Proposition 7.1.5]{AGS08}. We further fix a sequence $(\mm_n)_{n\in \N}$ of finite and positive Borel measures on $\prob(K)$ such that $\mm_n \to \mm$ (equivalently $W_2(\mm_n, \mm)\to 0$ as $n \to +\infty$, where, with a slight abuse of notation, we are denoting by $W_2$ also the Wasserstein distance on $\prob(\prob(K))$ with the ground metric being given by the $2$-Wasserstein distance on $\prob(K)$). We fix a Lipschitz continuous function $F:\prob(K) \to \R$ and denote by $L_F\ge 0$ its Lipschitz constant w.r.t.~the distance $W_2$. Notice that $F \in H^{1,2}(\W_2)$ and that it is bounded by a constant $C_F \ge 0$.  We approach the consistency problem by $\Gamma$-convergence methods \cite{DMgamma, Agamma}. Let us also mention that the study of metric Sobolev spaces with changing distances/measures has also been carried out in much more general situations, see, e.g.~\cite{pasqualetto, AES16}.

\medskip

The aim of this section is to study under which conditions a suitably modified and restricted version of the Sobolev norm on $\W_2=(\prob(K),W_2, \mm)$ can be approximated in the sense of $\Gamma$-convergence by the same object defined on $(\prob(K),W_2, \mm_n)$. We give here the relevant definitions and we refer to Section \ref{sec:introwss} for the notion of cylinder function and related definitions.

\begin{definition}\label{def:vn} Given a strictly decreasing and vanishing sequence $(\eps_n)_n \subset (0,+\infty)$ bounding from above the sequence $W_2(\mm_n, \mm)$, i.e., 
    \[
    W_2(\mm_n,\mm) \leq \eps_n \qquad \text{for all} \ n \in \mathbb{N},
    \]
we define the following sequence of subsets of the cylinder functions $\ccyl{\prob(K)}{\rmC_b^1(\R^d)}$: 
\[\mathcal{U}_n:= \left \{ G:= \psi \circ \lin \pphi : \psi \in \mathcal{P}_n, \, \pphi=(\phi_1, \dots, \phi_n), \, \|G\|_{CBL} \le \eta_n\right \}, \quad n \in \N,\]
where $(\phi_n)_n \subset \rmC_c^\infty(\R^d)$ is a countable and dense subset of $\rmC_c^\infty(\R^d)$, $\eta_n:= \eps_n^{-1/4}$, $\mathcal{P}_n$ denotes the set of polynomials on $\R^n$ of degree at most $n$, $\lin{\pphi}$ is as in \eqref{eq:lin}, and 
\[ \|G\|_{CBL} := \|G\|_\infty \vee \|\rmD G\|_\infty \vee \Lip(G, \prob(K)) \vee \Lip(\rmD G, \prob(K) \times K), \quad G \in \ccyl{\prob(K)}{\rmC_b^1(\R^d)},  \]
where $\Lip(\cdot, \cdot)$ is as in \eqref{eq:thelip} and $\rmD G$ is defined in \eqref{eq:thediff}.
\end{definition}

Notice that, since $\eta_n$ is increasing, we have that $\mathcal{U}_n \subset \mathcal{U}_{n+1}$ and $\bigcup_n \mathcal{U}_n$ is dense in $H^{1,2}(\W_2)$; for the latter result we refer to \cite[Proposition 4.19]{FSS22}.

\begin{definition}\label{def:functionals}
We fix $\lambda \ge 0$ and we further define the functionals
\begin{align*}
    \JJ_n(G) &:= \begin{cases} \|G-F\|^2_{L^2(\prob_2(K), \mm_n)} + \lambda \pCE_{2, \mm_n}(G), \quad &\text{ if } G \in \mathcal{U}_n, \\
    + \infty \quad &\text{ if } G \in H^{1,2}(\W_2) \setminus \mathcal{ U}_n,\end{cases}\\
    \JJ(G) &:= \begin{cases} \|G-F\|^2_{L^2(\prob_2(K), \mm)} + \lambda \pCE_{2, \mm}(G), \quad &\text{ if } G \in \ccyl{\prob(K)}{\rmC_b^1(\R^d)} , \\
    + \infty \quad &\text{ if } G \in H^{1,2}(\W_2) \setminus \ccyl{\prob(K)}{\rmC_b^1(\R^d)},\end{cases}\\
    \overline{\JJ}(G) &:= \|G-F\|^2_{L^2(\prob_2(K), \mm)} + \lambda \CE_{2, \mm}(G), \quad G \in H^{1,2}(\W_2),
\end{align*}
where $\pCE_{2,\mm}$ is the pre-Cheeger energy defined in \eqref{eq:prech} for the space $\W_2=(\prob(K), W_2, \mm)$, $\pCE_{2,\mm_n}$ is defined in the same way but with $\mm_n$ in place of $\mm$, and $\CE_{2,\mm}$ is the Cheeger energy from \eqref{eq:relaxintroche} for the space $\W_2$.
\end{definition}

 Notice that, thanks to Proposition \ref{prop:introwss}, the pre-Cheeger energy of a cylinder function has a simple expression:
\[ \pCE_{2, \mm}(G) = \int_{\prob(K)}\int_{K} |\rmD G(\mu, x)|^2 \de \mu(x) \de \mm(\mu), \quad \pCE_{2, \mm_n}(G) = \int_{\prob(K)}\int_K |\rmD G(\mu, x)|^2 \de \mu(x) \de \mm_n(\mu).\]

We start with the following preliminary result.

\begin{lemma}\label{le:stima}
    Let $\mathcal{U}_n$ be defined as in Definition \ref{def:vn}. Then, for any $G \in \mathcal{U}_n$, we have that
    \[ |\JJ(G)-\JJ_n(G)| \le 2\biggl ( (C_F + \|G\|_\infty)(\Lip(G, \prob(K))+ L_F ) + 2\lambda  \Lip(\rmD G, \prob(K) \times K) \|\rmD G\|_{\infty} \biggr ) W_2(\mm_n, \mm).\]
\end{lemma}

\begin{proof} The proof is only a matter of computations that we report in full for the reader's convenience.

\emph{Claim 1. The map $A(\mu):=|G(\mu)-F(\mu)|^2$, $\mu \in \prob(K)$, is $2( C_F + \|G\|_\infty)(\Lip(G, \prob(K))+ L_F  )$-Lipschitz continuous.}

For any $\mu, \mu' \in \prob(K)$ we have that
\begin{align*}
    |A(\mu)-A(\mu')| &= \left | |G(\mu)- F(\mu) |^2 - |G(\mu')-F(\mu')|^2\right |\\ 
     &\le \left ( |G(\mu)| + |G(\mu')| + |F  (\mu)| + | F (\mu')| \right ) \left ( | G(\mu) - G(\mu') | + |  F (\mu) - F (\mu') | \right)\\
    & \le (2 C_F   + 2 \|G\|_\infty ) (\Lip(G, \prob(K)) +  L_F ) W_2(\mu, \mu').
\end{align*}

\emph{Claim 2. The map $B(\mu):=\int |\rmD G(\mu, x)|^2 \, \de\mu(x)$, $\mu \in \prob(K)$, is $4 \Lip(\rmD G, \prob(K) \times K) \|\rmD G\|_{\infty}$-Lipschitz continuous.}

Let $\mu, \mu' \in \prob_2(K)$; then
\begin{align*}
    |B(\mu)-B(\mu')| &= \left | \int |\rmD G(\mu, x)|^2 \, \de\mu(x) - \int |\rmD G(\mu', y)|^2 \, \de\mu'(y) \right |\\
                    &\le \int \biggl ( \left ( |\rmD G(\mu, x)| + |\rmD G(\mu', y)| \right ) \left | \rmD G(\mu, x)-\rmD G(\mu', y) \right | \biggr ) \, \de\ggamma(x,y) \\
                    & \le 2 \Lip(\rmD G) \|\rmD G\|_{\infty} \int \left ( W_2(\mu, \mu') + |x-y| \right ) \, \de\ggamma(x,y) \\
                    & \le 4 \Lip(\rmD G) \|\rmD G\|_{\infty} W_2(\mu, \mu'),
\end{align*}
where $\ggamma \in \Gamma_o(\mu, \mu')$ (cf.~\eqref{eq:optplan}). In the last inequality we used that $W_1(\mu,\mu') \leq W_2(\mu,\mu')$, which follows immediately from Jensen's inequality.

\emph{Claim 3. Conclusion.}

Using the representation of the pre-Cheeger energy coming from \eqref{eq:pcerep}, we find that
\begin{align*}
   |\JJ(G)-\JJ_n(G)| &\le \left |\int |G(\mu)- F (\mu)|^2 \, \de\mm_n(\mu) - \int |G(\mu')- F(\mu')|^2 \, \de\mm(\mu') \right | \\
   & \quad + \lambda \left | \int \int |\rmD G(\mu,x) |^2 \, \de\mu(x) \, \de\mm_n(\mu) - \int \int |\rmD G(\mu',y) |^2 \, \de\mu'(y) \, \de\mm(\mu') \right | \\
   &= \left | \int A \, \de\mm_n - \int A \, \de\mm \right | + \lambda \left | \int B \, \de\mm_n - \int B \, \de\mm \right | \\
   & \le \Lip(A, \prob(K)) \int W_2(\mu, \mu') \de \ggamma_n(\mu, \mu') + \lambda \Lip(B, \prob(K)) \int W_2(\mu, \mu') \de \ggamma_n(\mu, \mu') \\
   &\le \left ( 2(C_F + \|G\|_\infty)(\Lip(G)+ L_F) + 4\lambda \Lip(\rmD G) \|\rmD G\|_{\infty} \right ) W_2(\mm_n, \mm),
\end{align*}
where $\ggamma_n \in \Gamma_o(\mm_n, \mm)$ (cf.~\eqref{eq:optplan}, with the distance $\varrho$ being the $2$-Wasserstein distance on $\prob(K)$).
\end{proof}

\begin{theorem} \label{thm:moscow} Let $\JJ_n$ and $\overline{\JJ}$ be defined as in Definition \ref{def:functionals}. Then $\JJ_n \to \overline{\JJ}$ as $n \to + \infty$ in the following sense:
\begin{enumerate}
    \item for every $(G_n)_{n} \subset H^{1,2}(\W_2)$ converging weakly in $L^2(\prob(K), \mm)$ to $G \in  H^{1,2}(\W_2)$, it holds
    \[ \liminf_n \JJ_n(G_n) \ge \overline{\JJ}(G);\]
   \item for every $G \in H^{1,2}(\W_2)$ there exists a sequence $(G_n)_{n} \subset H^{1,2}(\W_2)$ converging strongly to $G$ in $H^{1,2}(\W_2)$ such that 
   \[ \limsup_n \JJ_n(G_n) \le \overline{\JJ}(G).\]
\end{enumerate}
\end{theorem}
\begin{proof}
    We prove separately the $\liminf$ and the $\limsup$ inequalities. 

\medskip 

Regarding point (1), we can assume without loss of generality that $G_n \in \mathcal{U}_n$ for every $n \in \N$; by Lemma \ref{le:stima} and the definition of $\mathcal{U}_n$ in Definition \ref{def:vn}, we have
\begin{align*}
    \JJ_n(G_n) \ge \JJ(G_n) - 2\biggl ( 2\lambda \eta_n^2 + ( C_F + \eta_n)(\eta_n+ L_F )\biggr ) W_2(\mm_n, \mm).
\end{align*}
Passing to the $\liminf_n$ and using the definition of $\eta_n$ as well as the bound on $W_2(\mm_n,\mm)$, we get that 
\[ \liminf_n \JJ_n(G_n) \ge \liminf_n \JJ(G_n) \ge \overline{\JJ}(G), \]
where we used for the latter inequality that $\pCE_{2, \mm}(\cdot) \ge \CE_{2, \mm}(\cdot)$ and that the Cheeger energy and the $L^2$ norm are lower semicontinuous w.r.t.~the weak convergence in $L^2(\prob(K), \mm)$ (these facts easily follow from the expression in \eqref{eq:relaxintroche}).

\medskip

Concerning point (2), by the density of $\bigcup_n \mathcal{U}_n$ in  $H^{1,2}(\W_2)$, we can find functions $H_k$ and numbers $N_k \in \N$ such that 
\[H_k \to G \text{ in } H^{1,2}(\W_2) \quad \text{and} \quad  \JJ(H_k) \to \overline{\JJ}(G) \quad \text{as} \ k \to + \infty\]
and
\[H_k \in \mathcal{U}_{N_k} \quad \text{with} \quad N_k < N_{k+1} \quad \text{ for every} \ k \in \N. \]
We can thus define $G_n:= H_k$ if $N_{k} \le n < N_{k+1}$ and $G_n=\lin{\phi_1}$ if $n < N_1$; observe that it holds
\[ G_n \to G \text{ in } H^{1,2}(\W_2), \quad \JJ(G_n) \to \overline{\JJ}(G) \text{ as } n \to + \infty \text{ and } G_n \in \mathcal{U}_{n}, \text{ for every } n \in \N. \]
By Lemma \ref{le:stima}, we have
\[ \JJ_n(G_n) \le \JJ(G_n) + 2\biggl ( 2\lambda \eta_n^2 + ( C_F   + \eta_n)(\eta_n+L_F)\biggr ) W_2(\mm_n, \mm). \]
Passing to the $\limsup_n$ and using the definition of $\eta_n$ coming from Definition \ref{def:vn}, we get that
\[ \limsup_n \JJ_n(G_n) \le \overline{\JJ}(G);\]
this finishes the proof.
\end{proof}

\begin{corollary}\label{cor:gamma} Let $G^{\star}_n \in \mathcal{U}_n$ be the unique minimizer of $\JJ_n$ and let $\tilde{G}_n  \in \mathcal{U}_n$ be a quasi-minimizer in the sense that
\[ \JJ_n(\tilde{G}_n) \le \JJ_n(G^\star_n)+ \gamma_n\]
for a vanishing non-negative sequence $(\gamma_n)_n$. Then $\tilde{G}_n \to G^\star$ in $L^2(\prob(K), \mm)$ and $\JJ_n(\tilde{G}_n) \to \overline{\JJ}(G^\star)$ as $n\to+\infty$, where $G^\star$ is the unique minimizer of $\bar{\JJ}$. If $\lambda>0$, the convergence holds in {$H^{1,2}(\W_2)$.}
\end{corollary}

\begin{proof}
The strict convexity, lower semicontinuity and coercivity of the above functionals give the existence and uniqueness of their minimizers.

It is enough to show that from any subsequence of $(\tilde{G}_n)_{n \in \N}$ it is possible to extract a further subsequence such that the above convergences hold. Let us consider an unrelabelled subsequence of $(\tilde{G}_n)_{n \in \N}$. For $n \in \N$ sufficiently large we have by Lemma \ref{le:stima} that 
\begin{align*}
     \JJ(\tilde{G}_n) &\le \JJ_n(\tilde{G}_n)  + 2\biggl ( 2\lambda \eta_n^2 + (C_F  + \eta_n)(\eta_n+ L_F )\biggr ) W_2(\mm_n, \mm) \\
    & \le  \JJ_n(G^\star_n)+ \gamma_n + 1  \\
    &\le \JJ_n(\lin {\phi_1})  +2   \\
    &< C < + \infty.
\end{align*}  
In turn, since $\JJ$ is coercive, the sequence $(\tilde{G}_n)_n$  is uniformly bounded in $L^2(\prob(K), \mm)$ and thus it admits a (unrelabelled) weakly converging subsequence with limit denoted by $H \in L^2(\prob(K), \mm)$. 
For any given $G \in H^{1,2}(\W_2)$, we can find, thanks to the convergence deduced in Theorem~\ref{thm:moscow}, a sequence $(G_n)_n$ converging to $G$ in $H^{1,2}(\W_2)$ and such that $\JJ_n(G_n) \to \overline{\JJ}(G)$. We thus have
\[
\bar{\JJ}(G) =  \lim_n \JJ_n(G_n) \ge \liminf_n \JJ_n(G^\star_n) \ge \liminf_n \JJ_n(\tilde{G}_n) -\gamma_n \ge  \overline{\JJ}(H).
\]
Since $G\in H^{1,2}(\W_2)$ was chosen arbitrarily, this shows that $H=G^\star$. Moreover, choosing $G=G^\star$, the above inequalities are all equalities, showing that $ \JJ_n(\tilde{G}_n) \to \overline{\JJ}(G^\star)$. It only remains to prove that $\tilde{G}_n$ converges to $G^\star$ strongly in $L^2(\prob(K, \mm)$. First of all, in light of Lemma \ref{le:stima}, it is easy to see that $ \overline{\JJ}(\tilde{G}_n) \to \overline{\JJ}(G^\star)$. Since $H^{1,2}(\W_2)$ is Hilbertian (cf.~Theorem \ref{thm:mainold}), $\overline{\JJ}$ is quadratic and thus satisfies the parallelogram identity
\[ \frac{1}{4} \overline{\JJ}(F_1-F_0)= \frac{1}{2} \overline{\JJ}(F_0) + \frac{1}{2} \overline{\JJ}(F_1) - \overline{\JJ}\left ( \frac{1}{2}F_0+ \frac{1}{2}F_1 \right ) \text{ for every $F_0, F_1 \in H^{1,2}(\W_2)$. }\]
Choosing $F_0=G^\star$ and $F_1= \tilde{G}_n $, we get 
\[ \limsup_n \frac{1}{4}\overline{\JJ}\left (  \tilde{G}_n  - G^\star \right) = \frac{1}{2}\overline{\JJ}(G^\star) + \limsup_n \left ( \frac{1}{2}  \overline{\JJ}(\tilde{G}_n) - \overline{\JJ}\left ( \frac{1}{2}G^\star+ \frac{1}{2} \tilde{G}_n  \right ) \right ) \le 0,\]
where the inequality is due to the lower semicontinuity of $\overline{\JJ}$ w.r.t.~the weak $L^2(\prob(K), \mm)$ convergence. Since 
\[  \overline{\JJ}\left ( \tilde{G}_n - G^\star \right) = \|G^\star -  \tilde{G}_n \|_{L^2(\prob(K), \mm)}^2 + \lambda \CE_{2, \mm}(G^\star -  \tilde{G}_n), \]
we have indeed verified the convergence in $L^2(\mathcal{P}(K),\mm)$ (resp.~in $H^{1,2}(\W_2)$ if $\lambda >0$).
\end{proof}

\subsection{Regularized least squares approximation of bounded functions}\label{sec:cohen} 

This section serves as probabilistic counterpart of the previous Section \ref{sec:cons1}, doing, in a way, what Section \ref{sec:32} does compared to Section \ref{sec:31}. 

The final result of Section \ref{sec:cons1} states that we can approximate the minimizer of $\overline{\JJ}$ (basically the Sobolev norm in $H^{1,2}(\W_2)$ with quasi-minimizers of $\JJ_n$ (restrictions of the Sobolev norm in $H^{1,2}(\prob(K), W_2, \mm_n)$); however Corollary \ref{cor:gamma} does not provide us with any rate of such convergence. In this section we show that, at the cost of introducing some randomness and thus obtaining only estimates in expected value, we are able to approximate any given bounded function $F: \prob(K) \to \R$ with minimizers of (rescaled versions of) functionals $\JJ_N$, which have an analogous definition to the $\JJ_n$ of the previous section.

\medskip

Let us briefly introduce the setting of this section which is very similar to the one of Section \ref{sec:cons1}. We fix a compact $K \subset \R^d$ and a probability measure $\pp$ on $\prob(K)=\prob_2(K)$.
We fix two natural numbers $n,N \in \N_0$ and consider a \emph{finite dimensional subspace} of cylinder functions (cf.~\eqref{eq:acyl}) $\mathcal{V}_n \subset \ccyl{\prob(K)}{\rmC_b^1(K)} \subset H^{1,2}(\prob(K), W_2, \pp)$ such that $\dim(\mathcal{V}_n)=n$ and $N$ i.i.d.~points $\mu_1, \dots, \mu_N$ distributed on $\prob(K)$ according to $\pp$; this means that we are considering as in Section \ref{sec:32} an underlying probability space $(\Omega, \EE, \PP)$. Accordingly we define the (random) empirical measure $\pp_N:= \frac{1}{N}\sum_{j=1}^N \delta_{\mu_j}$ and we fix $\lambda \ge 0$.

\medskip

The following lemma is a simple consequence of well-known linear algebra results that we report here for the reader's convenience.
\begin{lemma}\label{le:db} There exists a finite family of functions $\mathcal{B}_n:=\{\ell_1, \dots, \ell_n\} \subset \mathcal{V}_n$ such that
\begin{enumerate}
    \item $\mathcal{B}_n$ is a linear basis of $\mathcal{V}_n$;
    \item $\int_{\prob(K)}\ell_i \ell_j \de \pp =0$ for every $i,j \in \{1, \dots, n\}$, $i \ne j$;
    \item $\int_{\prob(K)} |\ell_i|^2 \de \pp =1$ for every $i=1,\dots, n$;
    \item $\pCE_{2,\pp}(\ell_i, \ell_j)=0$ for every $i,j \in \{1, \dots, n\}$, $i \ne j$,
\end{enumerate} 
where $\pCE_{2,\pp}$ is defined as in \eqref{eq:prech}.
\end{lemma}

\begin{proof}
This is an immediate consequence of the following fact: if $V$ is a real vector space with dimension $n$ and $\la \cdot, \cdot \ra_1$ and $\la \cdot, \cdot \ra_2$ are two scalar products on $V$, then there exists a basis of $V$ which is orthonormal w.r.t.~$\la \cdot, \cdot \ra_1$ and orthogonal w.r.t.~$\la \cdot, \cdot \ra_2$. Indeed, let us denote by $V_i$ the space $V$ endowed with $\la \cdot, \cdot \ra_i$, $i=1,2$, and consider an orthonormal basis $(b_i)_{i=1}^n$ in $V_1$. Let $T:V \to V$ be defined as
\[ T(x):= \sum_{i=1}^n \la x,b_i \ra_2 b_i.\]
It is easy to verify that $\la T(x),y \ra_1 = \la x,y\ra_2$ for every $x,y \in V$ so that 
\[ \la T(x),y \ra_1 = \la x,y\ra_2 = \la x, T(y) \ra_1\]
which means that $T$ is selfadjoint as an operator on $V_1$. Moreover $\la T(x),x\ra_1 = |x|^2_2 >0$ for any $x \in V \setminus \{0\}$. This implies that $T$ is diagonalizable; i.e.,~there exists an orthonormal basis $e_1,\dots, e_n$ of $V_1$ such that $Te_i=\lambda_i e_i$ for some non zero $\lambda_i \in \R$. It is easy to check that $(e_i)_{i=1}^n$ is the sought basis.
\end{proof}

Recall that the pre-Cheeger energy of a cylinder function $G$ has a simple expression thanks to Proposition \ref{prop:introwss}:
\[ \pCE_{2, \pp}(G) = \int_{\prob(K)}\int_{K} |\rmD G(\mu, x)|^2 \de \mu(x) \de \pp(\mu)\]
and
\[\pCE_{2, \pp_N}(G) = \int_{\prob(K)}\int_K |\rmD G(\mu, x)|^2 \de \mu(x) \de \pp_N(\mu)= \frac{1}{N} \sum_{j=1}^N \int_K |\rmD G(\mu_j,x)|^2 \de \mu_j(x),\]
respectively.

Let us define some more objects we work with in this section.

\begin{definition}\label{def:someop} Given a function $F \in L^2(\prob(K), \pp_N)$, we define the functionals
\begin{align*}
\JJ_{N, F}^\lambda(G)&:= \|F-G\|^2_{L^2(\prob(K), \pp_N)} + \lambda \pCE_{2,\pp_N}(G), \quad &&G \in \ccyl{\prob(K)}{\rmC_b^1(K)},\\
\mathcal{N}_{N,F}(G) &:= \|F-G\|^2_{L^2(\prob(K), \pp_N)}, \quad &&G \in L^2(\prob(K), \pp_N).
\end{align*}
Both functionals admit a unique minimizer when restricted to $\mathcal{V}_n$ so that we can define
\begin{align*}
    S_{N}^{\lambda, n}(F):= \argmin_{G \in \mathcal{V}_n} \JJ_{N, F}^\lambda(G),\\
    P_N^n (F):= \argmin_{G \in \mathcal{V}_n} \mathcal{N}_{N,F}(G).
\end{align*}
\end{definition}

Our aim is to estimate the expected value of the $L^2(\prob(K), \pp)$-norm of the difference of $F$ and $S_{N}^{\lambda, n}(F)$, in case $F$ is bounded. For that purpose, we proceed by extending the approach outlined in \cite{cohen}.

In the next Lemma we make the form of $S_{N}^{\lambda, n}(F)$  as in Definition \ref{def:someop} more explicit. To do that we first need to define two matrices that are crucial in our approach.

\begin{definition}\label{def:matrices} Let $\{\ell_1, \dots, \ell_n\}$ be as in Lemma \ref{le:db}. We define the matrices $L_{N,n} \in \R^{N\times n}$ and $D_n \in \R^{n \times n}$ as 
\[ (L_{N,n})_{ji} := \frac{1}{\sqrt{N}}\ell_i(\mu_j), \quad (D_n)_{hi} = \frac{1}{N} \sum_{j=1}^N\int_{K} \rmD\ell_i(\mu_j,x) \cdot \rmD \ell_h(\mu_j,x) \de \mu_j(x). \]
\end{definition}

The expected value of the matrices $L_{N,n}^\top L_{N,n}$ and $D_n$ is easily computed:
\begin{equation}\label{eq:expv}
\mathbb{E} [L_{N,n}^\top L_{N,n}]= I_n \quad and \quad \mathbb{E}[D_n]= \Gamma_n,    
\end{equation}
where $I_n$ is the identity matrix of order $n$ and $\Gamma_n \in \R^{n \times n}$ is a diagonal matrix with entries
\[ (\Gamma_n)_{ii}= \pCE_{2,\pp}(\ell_i), \quad i=1, \dots, n.\]

\begin{lemma}\label{le:fewthings} Let $\{\ell_1, \dots, \ell_n\}$ be as in Lemma \ref{le:db} and let $F \in L^2(\prob(K), \pp_N)$. Then $G=S_{N}^{\lambda, n}(F)$ (cf.~Definition \ref{def:someop}) if and only if
\[ G= \sum_{i=1}^n \ell_i w^\star_i,\]
where $w^\star \in \R^n$ is the unique minimizer in $\R^n$ of the functional $\JJ_{N, z^{F}}^{\lambda,n}: \R^n \to \R$ defined as
\begin{equation}\label{eq:funcrn}
 \JJ_{N, z^{F}}^{\lambda,n}(w):= |L_{N,n}(w-z^F)|^2_N + \lambda w^\top D_n w, \quad w \in \R^n, 
\end{equation}
where $|\cdot|_N$ denotes the Euclidean norm in $\R^N$ and $z^{F} \in \R^n$ is such that
\[ P^n_N F = \sum_{i=1}^n \ell_i z_i^{F}.\]
Equivalently $w^\star$ solves
\begin{align} \label{eq:wstar}
(L_{N,n}^\top L_{N,n} + \lambda D_n ) w^\star= L_{N,n}^\top L_{N,n} z^{F}.
\end{align}
In particular, $S_{N}^{\lambda, n}$ is a linear operator.
\end{lemma}

\begin{proof}
For every $G \in \mathcal{V}_n$, we can rewrite
\begin{align*}
\JJ_{N, F}^\lambda(G) &= \frac{1}{N}\sum_{i=1}^N |G(\mu_i)-F(\mu_i)|^2 + \frac{\lambda}{N} \sum_{i=1}^N \|\rmD G[\mu_i]\|^2 \\
&= \frac{1}{N}\sum_{i=1}^N |G(\mu_i)-P^n_N F (\mu_i)|^2 + \frac{\lambda}{N} \sum_{i=1}^N \|\rmD  G  [\mu_i]\|^2 + C(n,N,F)
\end{align*}
where $C(n,N,F)$ is a constant depending only on $F$, $n$ and $N$, and $\|\rmD G[\cdot]\|$ is as in \eqref{eq:pcerep}. Furthermore, upon defining
\[
\tilde{\JJ}_{N, F}^\lambda(G):= \frac{1}{N}\sum_{i=1}^N |G(\mu_i)-P^n_N F(\mu_i)|^2 + \frac{\lambda}{N} \sum_{i=1}^N \|\rmD G[\mu_i]\|^2
\]
we obtain that
\begin{align*}
\JJ_{N, F}^\lambda(G) = \tilde{\JJ}_{N, F}^\lambda(G) + C(n,N,F).
\end{align*}

Since $\tilde{\JJ}_{N, F}^\lambda$ and $\JJ_{N, F}^\lambda$ differ only by a constant, they have the same minimizer. For a given $G \in \mathcal{V}_n$ there exists a unique $w^G \in \R^n$ such that 
\[ G= \sum_{i=1}^n \ell_i w_i^G\]
and it is immediate to check that $\tilde{\JJ}_{N, F}^\lambda (G) = \JJ_{N, z^{F}}^{\lambda,n}(w^G)$. This concludes the proof.
\end{proof}

\begin{remark} \label{rem:wstar}
We note that~\eqref{eq:wstar} can equivalently be stated as
\[
(L_{N,n}^\top L_{N,n}+\lambda D_n)w^\star=y^F,
\]
where $y^F \in \mathbb{R}^n$ with
\[y^F_i=\frac{1}{N} \sum_{j=1}^N F(\mu_j) \ell_i(\mu_j), \qquad 1 =1,\dotsc,n.\]
In particular, we may obtain the solution $w^\star$ by solving a linear equation.
\end{remark}

The next result is an auxiliary large deviation bound, which allows us to control the spectral behavior of the matrices introduced in Definition \ref{def:matrices}.

\begin{proposition} \label{prop:probbound} Let $L_{N,n}$ and $D_n$ be defined as in Definition \ref{def:matrices} and let $I_n$ and $\Gamma_n$ be as in \eqref{eq:expv}. Then 
\begin{align}\label{eq:lab1}
\PP ( \|L_{N,n}^\top L_{N,n} - I_n\| > \delta ) &\le 2n \exp \left \{ -N \frac{c_\delta}{ K_\lambda(n)}\right \} \quad \text{ for every } 0 < \delta <1
\end{align}
and, for every $0<\delta<1$,
\begin{align} \label{eq:lab2}
\mathbb{P} \left ( \|L_{N,n}^\top L_{N,n} + \lambda D_n -(I_n +\lambda \Gamma_n) \|> \delta \sigma_{\max,n}^{\lambda} \right ) &\le 2n \exp \left \{ -N \sigma_{\min,n}^{\lambda} \frac{c_\delta}{ K_\lambda(n)} \right \},
\end{align}
where 
\begin{align*}
 K_\lambda (n)&:= \sup_{\mu \in \prob(K)} \sum_{i=1}^n \left(|\ell_i(\mu)|^2 +  \lambda \int_K \left | \rmD \ell_i(\mu,x)\right |^2 \de \mu(x) \right), \\
\sigma_{\min,n}^{\lambda} &:= 1+\lambda \min_{i=1, \dots, n} \pCE_{2, \pp}(\ell_i), \quad \sigma_{\max,n}^{\lambda} := 1+\lambda \max_{i=1, \dots, n} \pCE_{2, \pp}(\ell_i), \\
c_{\delta}&:= (1+\delta)\log(1+\delta)-\delta,
\end{align*}
and $\|\cdot\|$ denotes the spectral norm.
\end{proposition}
\begin{proof} The first bound follows precisely as in the proof of \cite[Theorem 1]{cohen}. The second one analogously: $L_{N,n}^\top L_{N,n} +\lambda D_n $ is the sum of $N$ symmetric and positive semi-definite matrices $X_k$, $k=1, \dots, N$, which are i.i.d.~copies of the matrix $X$ with entries 
\[ X_{ih}:= X_{ih}^1+X_{ih}^2:= \frac{1}{N} \ell_i(\mu) \ell_h(\mu) +\frac{ \lambda}{N}\int_{K}  \rmD\ell_i(\mu,x) \cdot \rmD \ell_h(\mu,x) \de \mu(x), \quad i,h \in \{1, \dots, n\},
\]
with $\mu$ distributed according to $\pp$. 

Indeed, since both $X^1$ and $X^2$ are symmetric and positive semi-definite matrices, so is $X$ and thus its spectral norm can be bounded by its trace; i.e., we have that
\[
\norm{X} \leq \frac{1}{N} \sum_{i=1}^n \left( |\ell_i(\mu)|^2+\lambda \int_K \left|\rmD \ell_i(\mu, x)\right|^2 \de \mu(x) \right) \qquad \text{for a.e.} \ \mu \in \mathcal{P}(K).
\]
Invoking the definition of $ K_\lambda(n)$, this immediately yields the uniform bound
\[
\norm{X} \leq \frac{ K_\lambda(n)}{N}.
\]
In turn, for every $0<\delta<1$, the matrix Chernoff inequality, cf.~\cite[Theorem 1.1]{Tropp:2012}, yields that
\begin{align*}
    \mathbb{P} \left ( \lambda_{\min}( L_{N,n}^\top L_{N,n} + \lambda D_n )-  \sigma_{\min,n}^{\lambda}  \le - \delta \sigma_{\min,n}^{\lambda} \right ) &= \mathbb{P} \left ( \lambda_{\min}(L_{N,n}^\top L_{N,n} +\lambda D_n) \le (1-\delta) \sigma_{\min,n}^{\lambda} \right )\\
    & \le n \left ( \frac{\mathrm{e}^{-\delta}}{(1-\delta)^{1-\delta}} \right )^{\frac{N \sigma_{\min,n}^{\lambda}}{ K_\lambda(n)}}\\
    & \le n \left ( \frac{\mathrm{e}^{\delta}}{(1+\delta)^{1+\delta}} \right )^{\frac{N \sigma_{\min,n}^{\lambda}}{ K_\lambda(n)}}
\end{align*}
and 
\begin{align*}
    \mathbb{P} \left ( \lambda_{\max}( L_{N,n}^\top L_{N,n} +\lambda D_n )- \sigma_{\max,n}^{\lambda}  \ge  \delta  \sigma_{\max,n}^{\lambda}  \right ) & = \mathbb{P} \left ( \lambda_{\max}(L_{N,n}^\top L_{N,n} +\lambda D_n ) \ge (1+\delta) \sigma_{\max,n}^{\lambda} \right )\\
    & \le n \left ( \frac{\mathrm{e}^{\delta}}{(1+\delta)^{1+\delta}} \right )^{\frac{N \sigma_{\max,n}^{\lambda} }{ K_\lambda(n)}}\\
    & \le n \left ( \frac{\mathrm{e}^{\delta}}{(1+\delta)^{1+\delta}} \right )^{\frac{N \sigma_{\min,n}^{\lambda}}{ K_\lambda(n)}}
\end{align*}
so that 
\[ \mathbb{P} \left ( \|  L_{N,n}^\top L_{N,n}   +   \lambda D_n   -(I_n +\lambda \Gamma_n) \|> \delta   \sigma_{\max,n}^{\lambda}    \right ) \le 2n \exp \left \{ -N   \sigma_{\min,n}^{\lambda}    \frac{c_\delta}{ K_\lambda(n)} \right \}.\]
\end{proof}

\begin{remark}
    The quantities $K_\lambda(n)$, $\sigma^\lambda_{\min,n}$ and $\sigma^\lambda_{\max, n}$ depend of course on $\pp$, which is usually unknown, determining the chosen basis $\mathcal{B}_n$. Hence, there is not a universal method to estimate $K_\lambda(n)$, $\sigma^\lambda_{\min,n}$ and $\sigma^\lambda_{\max, n}$. Yet, as a simple example one could consider the case where $\pp = \mathsf{D}_\sharp \pp_0$ for a probability measure $\pp_0$ on $K$, where $\mathsf{D}: K \to \prob(K)$ is simply the map sending $x \mapsto \delta_x$. This provides an isomorphism between the Sobolev spaces $H^{1,2}(K, \sfd_e, \pp_0)$ and $H^{1,2}(\prob(K), W_2, \pp)$, where $\sfd_e$ is the Euclidean distance on $K$. In particular, whenever $f,g \in \rmC^\infty(K)$, we have
    \[ \langle \lin{f}, \lin{g} \rangle_{L^2(\prob(K), \pp)} = \langle f, g \rangle_{L^2(K, \pp_0)}, \quad \pCE_{2, \pp}(f,g) =\langle \nabla f, \nabla g \rangle_{L^2(K, \pp_0; \R^d)}. \]
    Thus providing a basis $\mathcal{B}_n$ as in Lemma \ref{le:db} amount to choosing functions $f_1, \dots, f_n \in \rmC^\infty(K)$ forming an orthonormal system in $L^2(K, \pp_0)$ and additionally satisfying 
    \[ \int_K |\nabla f_i(x)|^2 \de \pp_0(x)=1 \text{ for every } i=1, \dots, n.\]
    We can then set $\mathcal{B}_n:= \left \{\lin{f_1}, \dots, \lin{f_n} \right \}$. A bound for the quantity $K_\lambda(n)$ is easily found then observing that $|\lin{fi}(\mu)| \le \|f_i\|_\infty$ and $\int_K |\rmD \lin{f_i}(\mu, x)|^2 \de \mu(x) \le \|\nabla f_i\|_\infty^2$ for every $i=1, \dots, n$ and for every $\mu \in \prob(K)$.
    
\end{remark}

\nc

The condition~\eqref{eq:thecond} below ensures that the probabilities in Proposition~\ref{prop:probbound} are small: we think to fix the dimension $n$ of the subspace $\mathcal{V}_n$ and an order $r>0$. Then, up to choosing enough point evaluations $N$, we obtain that the above probability is smaller than $N^{-r}$. The required condition reads as
\begin{equation}\label{eq:thecond}
    \frac{N}{\log(N)} \ge \frac{(1+r) K_\lambda(n)}{  \sigma_{\min,n}^{\lambda}   c_{\frac{1}{2  \sigma_{\max,n}^{\lambda}   }}}, \quad N \ge 2,
\end{equation}
which further implies that
\[ \frac{N}{\log(N)} \ge \frac{(1+r) K_\lambda(n)}{c_{1/2}}.\]

Indeed, given~\eqref{eq:thecond} and choosing $\delta=1/2$ in \eqref{eq:lab1} as well as $\delta=(2\sigma_{\max,n}^{\lambda})^{-1}$ in \eqref{eq:lab2} leads to
\begin{multline} \label{eq:thecondcons}
\PP \left ( \|(  L_{N,n}^\top L_{N,n}   +  \lambda D_n  )-(I_n +\lambda \Gamma_n)\| > \frac{1}{2} \right ) + \PP \left ( \|  L_{N,n}^\top L_{N,n}   - I_n\| > \frac{1}{2} \right ) \\ \le 4\frac{n}{N} N^{-r}
\le 4 \frac{ K_\lambda(n)}{N} N^{-r} \le \frac{4c_{1/2}}{\log(N)(1+r)} N^{-r}  \le \frac{4c_{1/2}}{\log(2)} N^{-r} \le N^{-r},
\end{multline}
where we have used that $ K_\lambda(n) \ge n$ and that $N \ge 2$.

Notice moreover that, if $\|  L_{N,n}^\top L_{N,n}   - I_n\| \le \frac{1}{2}$, we get
\begin{equation}\label{eq:inverse1}
\|(  L_{N,n}^\top L_{N,n}  )^{-1} \| \le \sum_{k=0}^\infty \|  L_{N,n}^\top L_{N,n}  -I_n\|^k \le \frac{1}{1-\frac{1}{2}}=2.    
\end{equation}
Analogously, if $\|(  L_{N,n}^\top L_{N,n}   +  \lambda D_n  )-(I_n +\lambda \Gamma_n)\| \le \frac{1}{2}$, we get, setting $B:=(  L_{N,n}^\top L_{N,n}   +  \lambda D_n  )$ and $\Gamma_\lambda:=(I_n +\lambda \Gamma_n)$, that
\[   \sigma_{\min,n}^{\lambda}    \|B\Gamma_\lambda^{-1}-I_n\| \le \|B-\Gamma_\lambda\| \le \frac{1}{2}\]
so that
\[   \sigma_{\min,n}^{\lambda}    \|B^{-1}\| \le \|B^{-1} \Gamma_\lambda\| \le \frac{1}{1-\frac{1}{2  \sigma_{\min,n}^{\lambda}   }}.\]
Consequently, we have that
\begin{equation}\label{eq:inverse2}
\|(  L_{N,n}^\top L_{N,n}   +  \lambda D_n  )^{-1}\|=\|B^{-1}\| \le \frac{1}{\frac{1}{2} + \lambda \min_{i=1, \dots, n} \pCE_{2, \pp} (\ell_i)}.
\end{equation}

Now, we are ready to present the first result of this section, where the data are noiseless samples of the function $F$ we intend to recover. 

\begin{theorem}\label{teo:cohen2} Let $M>0$ and let $F: \prob(K) \to [-M,M]$ be a measurable function and let $F^{\star}$ be defined as 
\[ F^\star := -M \vee S_N^{\lambda,n}(F) \wedge M,\]
where $S_N^{\lambda,n}$ is as in Definition \ref{def:someop}; let $r>0$ be given and assume that $n,N \in \N$ satisfy \eqref{eq:thecond}. Then 
\begin{multline*}
    \mathbb{E} \left [ \|F-F^\star\|_{L^2(\prob(K), \pp)}^2\right ] \\
    \le 2 e(F,n) \left ( 1 + \frac{ c_{1/2}}{\log(N)(1+r)(1/2+\lambda   \mu_{\min,n}  )^2} \right ) + 4 \lambda \pCE_{2, \pp}(P_n F) +  2  M^2 N^{-r},
\end{multline*}
where $P_n F$ is the ${L^2(\prob(K), \pp)}$-orthogonal projection of $F$ onto $\mathcal{V}_n$  and 
\begin{align*}
e(F,n) &:= \|P_n F-F\|^2_{L^2(\prob(K), \pp)},\\
  \mu_{\min,n}  &:= \min_{i=1, \dots, n} \pCE_{2, \pp}(\ell_i). 
\end{align*}
\end{theorem}

\begin{proof}
Recall that we have fixed an underlying probability space $(\Omega, \EE, \PP)$; we consider a partition of $\Omega$ into the two sets
\[ \Omega_+:= \left \{ \omega \in \Omega : \|(  L_{N,n}^\top L_{N,n}   +  \lambda D_n  )-(I_n +\lambda \Gamma_n)\| \le \frac{1}{2} \wedge \|  L_{N,n}^\top L_{N,n}   - I_n\| \le \frac{1}{2} \right \}\]
 and its complement $\Omega_-:= \Omega\setminus \Omega_+$. Notice that \eqref{eq:thecondcons} gives $\PP(\Omega_-) \le N^{-r}$. We have then
\[ \mathbb{E} \left [ \|F-F^\star\|_{L^2(\prob(\Omega), \pp)}^2\right ] \le \int_{\Omega_+} \|F-F^\star\|_{L^2(\prob(\Omega), \pp)}^2 \de \PP + 2 M^2 N^{-r}.\]
We are left to estimate 

\[ \int_{\Omega_+} \|F-F^\star\|_{L^2(\prob(K), \pp)}^2 \de \PP \le \int_{\Omega_+} \|F-S_N^{\lambda,n}(F)\|_{L^2(\prob(K), \pp)}^2 \de \PP. \]

Since by Lemma \ref{le:fewthings} the operator $S_N^{\lambda,n}$ is linear, we have
\[ F-S_N^{\lambda,n}(F) = (I-S_{N}^{\lambda, n})H + P_n F -S_{N}^{\lambda, n} (P_n F), \]
where $H= F-P_nF$.

Observe that for $\xi \in \mathcal{V}_n$ and a draw in $\Omega_+$ it holds
\begin{align} \label{eq:Ineq1} 
\int_{\prob(K)} |\xi|^2 \de \pp \le 2 \int_{\prob(K)} |\xi|^2 \de \pp_N.
\end{align}
Indeed, there are coefficients $\theta_i \in \R$, $i=1, \dots, n$, such that $\xi=\sum_{i=1}^n \theta_i \ell_i$, and in turn
\[ \int_{\prob(K)} |\xi|^2 \de \pp = \sum_{i=1}^n |\theta_i|^2 = |\theta|_n^2\]
as well as
\[ \int_{\prob(K)} |\xi|^2 \de \pp_N = \frac{1}{N} \sum_{j=1}^N |\xi(\mu_j)|^2 = \frac{1}{N} \sum_{j=1}^N \sum_{i=1}^n \sum_{h=1}^n \theta_i \theta_h \ell_i(\mu_j)\ell_h(\mu_j) =   L_{N,n}^\top L_{N,n}   \theta \cdot \theta \ge \frac{1}{2} |\theta|_n^2,\]
where we have used~\eqref{eq:inverse1} and denoted by $|\cdot|_n$ the Euclidean norm in $\R^n$. By the very definition of $S_{N}^{\lambda, n}$ we further find that
\begin{align*}
 \|S_{N}^{\lambda, n}(G)- G\|^2_{L^2(\prob(K), \pp_N)} &\le \|S_{N}^{\lambda, n}(G)- G\|^2_{L^2(\prob(K), \pp_N)} + \lambda \pCE_{2,\pp_N}(S_{N}^{\lambda, n}(G)) \\
 &\le \|G- G\|^2_{L^2(\prob(K), \pp_N)} + \lambda \pCE_{2, \pp_N}(G) \\
 &= \lambda \pCE_{2, \pp_N}(G)
\end{align*}
for any $G \in \mathcal{V}_n$. Combining this inequality with~\eqref{eq:Ineq1} yields
\begin{align*}
\int_{\Omega_+} \|P_n F -S_{N}^{\lambda, n} (P_n F)\|_{L^2(\prob(K), \pp)}^2 \de \PP &\le 2 \int_{\Omega_+} \|P_n F -S_{N}^{\lambda, n} (P_n F)\|_{L^2(\prob(K), \pp_N)}^2   \de \PP\\
 & \le 2\lambda \int_{\Omega_+}\pCE_{2, \pp_N}(P_n F) \de \PP \\
& \le 2\lambda \int_{\Omega} \pCE_{2, \pp_N}(P_n F) \de \PP \\
& = 2\lambda \pCE_{2, \pp}(P_n F),
\end{align*}
where the last equality is a consequence of the fact that the measures $\mu_j$ are uniformly distributed according to $\pp$ and the definition of expected value.\\
It only remains to estimate
\[ \int_{\Omega_+} \| (I-S_{N}^{\lambda, n})H\|_{L^2(\prob(K), \pp)}^2 \de \PP.\]
This follows precisely as in \cite{cohen}: since $H$ is $L^2(\prob(K), \pp)$-orthogonal to $\mathcal{V}_n$ and $S_{N}^{\lambda, n} H$ belongs to $\mathcal{V}_n$, we have that 
\[ \| (I-S_{N}^{\lambda, n})H\|_{L^2(\prob(K), \pp)}^2 = \|H\|^2_{L^2(\prob(K), \pp)} + \|S_{N}^{\lambda, n} H\|^2_{L^2(\prob(K), \pp)} \]
and
\[ \|S_{N}^{\lambda, n} H\|^2_{L^2(\prob(K), \pp)} = \sum_{i=1}^n |a_i|^2\]
where $a=(a_i)_{i=1}^n \in \R^n$ is the solution of
\[ (  L_{N,n}^\top L_{N,n}   +   \lambda D_n  ) a = y^H \]
with $y^H \in \R^N$ defined as $y_i^H:= \int_{\prob(K)} H(\mu) \ell_i(\mu) \de \pp_N(\mu)$, $i=1, \dots, n$; cf.~Remark~\ref{rem:wstar}. Using \eqref{eq:inverse2}, we can obtain as in \cite{cohen} that 
\[ \int_{\Omega_+}\|S_{N}^{\lambda, n} H\|^2_{L^2(\prob(K), \pp)} \de \mathbb{P} \le \alpha_{\lambda,n}^2 \frac{ K_\lambda(n)}{N} \|H\|^2_{L^2(\prob(K), \pp)}, \]
where 
\begin{equation}\label{eq:alphaln}
\alpha_{\lambda,n}:=\frac{1}{\frac{1}{2} + \lambda \min_{i=1, \dots, n} \pCE_{2, \pp} (\ell_i)}.
\end{equation}

Putting everything together we get
\begin{multline*}
\mathbb{E} \left [ \|F-F^\star\|_{L^2(\prob(K), \pp)}^2\right ] \\
\begin{aligned}
&\le \int_{\Omega_+} \|F-F^\star\|_{L^2(\prob(K), \pp)}^2 \de \PP + 2 M^2 N^{-r}\\
&\le  2  \int_{\Omega_+} \| (I-S_{N}^{\lambda, n})H\|_{L^2(\prob(K), \pp)}^2 \de \PP + 4  \lambda \pCE_{2, \pp}(P_n F) + 2 M^2 N^{-r}\\
&\le  2  \|H\|^2_{L^2(\prob(K), \pp)} +  2\alpha_{\lambda,n}^2  \frac{ K_\lambda(n)}{N} \|H\|^2_{L^2(\prob(K), \pp)} + 4 \lambda \pCE_{2, \pp}(P_n F) + 2 M^2 N^{-r}. 
\end{aligned}
\end{multline*}
Finally, recalling the definition of $H$ and employing~\eqref{eq:thecond} yield that
\begin{align*}
\mathbb{E} \left [ \|F-F^\star\|_{L^2(\prob(K), \pp)}^2\right ] &\le \|P_n F-F\|^2_{L^2(\prob(K), \pp)} \left (  2  + \frac{  2\alpha_{\lambda,n}^2  c_{1/2}}{\log(N)(1+r)} \right ) \\
& \quad + 4  \lambda \pCE_{2, \pp}(P_n F) +  2M^2 N^{-r}.
\end{align*}
\end{proof}

We assume now to know the target function $F$ at the (random) points $\mu_1, \dots, \mu_N$ only up to some noise. In particular, we assume that we have at our disposal the noisy values
\begin{equation}\label{eq:feta}
  \widetilde{F}  (\mu_j):= F(\mu_j)+\eta_j, \quad j=1, \dots, N,
\end{equation}
where $(\eta_j)_{j=1}^N$ are i.i.d.~random variables with variance bounded by $\sigma^2>0$. Notice that the (random) function $  \widetilde{F}  $ belongs to $L^2(\prob(K), \pp_N)$.

\begin{theorem}\label{teo:cohen3}  Let $M>0$, let $F: \prob(K) \to [-M,M]$ be a measurable function and let $  \widetilde{F}^\star   $ be defined as 
\[   \widetilde{F}^\star    := -M \vee S_N^{\lambda,n}(  \widetilde{F}  ) \wedge M,\]
where $S_N^{\lambda,n}$ is as in Definition \ref{def:someop} and $  \widetilde{F}  $ is as in \eqref{eq:feta} based on the noisy data $(y_i)_{i=1}^N$ whose i.i.d.~errors $(\eta_i)_{i=1}^N$ have variance bounded by $\sigma^2>0$. Let $r>0$ be given and assume that $n,N \in \N$ satisfy \eqref{eq:thecond}. Then 
\begin{align*}
\mathbb{E} \left [ \|F-  \widetilde{F}^\star   \|_{L^2(\prob(K), \pp)}^2\right ] &\le 2 e(F,n) \left ( 1 + \frac{ c_{1/2}}{\log(N)(1+r)(1/2+\lambda   \mu_{\min,n}  )^2} \right ) + 8\lambda \pCE_{2, \pp}(P_n F) + \\
& \quad + 4\frac{\sigma^2}{(1+\lambda   \mu_{\min,n}  )^2}\frac{n}{N}+ 2  M^2 N^{-r},
\end{align*}
where $e(F,n)$ and $  \mu_{\min,n}  $ are as in Theorem \ref{teo:cohen2}.
\end{theorem}

\begin{proof}
The proof follows the one of Theorem \ref{teo:cohen2} and it is also inspired by \cite[Theorem 3]{cohen}.

We split again $\Omega$ into the sets $\Omega_+$ and $\Omega_-$. Performing the same calculations as before, we find that
\[ \mathbb{E} \left [ \|F-  \widetilde{F}^\star   \|_{L^2(\prob(\Omega), \pp)}^2\right ] \le \int_{\Omega_+} \|F-S_N^{\lambda,n}(  \widetilde{F}  )\|_{L^2(\prob(\Omega), \pp)}^2 \de \PP + 2 M^2 N^{-r}.\]

We can write
\[F - S_N^{\lambda,n}(  \widetilde{F}  ) = (I-S_{N}^{\lambda, n})(H + P_n F) - S_{N}^{\lambda, n}(\eta)\]
where $H= F-P_nF$, $P_n$ is the orthogonal $L^2(\prob(K), \pp)$ projection on $\mathcal{V}_n$ and $\eta \in L^2(\prob(K), \mm_N)$ is the (random) function taking values $\eta(\mu_j):=\eta_j$, $j=1, \dots, N$.

Following precisely the same steps as in the previous proof, we obtain the bound
\begin{align*}
\int_{\Omega_+} \|F-S_N^{\lambda,n}(  \widetilde{F}  )\|_{L^2(\prob(\Omega), \pp)}^2 \de \PP &\le 2 \|H\|^2_{L^2(\prob(K), \pp)}+8\lambda \pCE_{2, \pp}(P_n F)\\
&\quad + \int_{\Omega_+} \left ( 2\|S_N^{\lambda,n}(H)\|^2_{L^2(\prob(K), \pp)} + 4\|S_N^{\lambda,n}(\mathcal{\eta})\|^2_{L^2(\prob(K), \pp)} \right )\de \PP.    
\end{align*}
As before, we have that 
\[ \|S_N^{\lambda,n}(H)\|^2_{L^2(\prob(K), \pp)} = \sum_{i=1}^n |a_i|^2 \quad \text{and} \quad \|S_N^{\lambda,n}(\mathcal{\eta})\|^2_{L^2(\prob(K), \pp)} = \sum_{i=1}^n |v_i|^2\]
where the vectors $a,v\in \R^n$ are the solutions of
\[(  L_{N,n}^\top L_{N,n}   +   \lambda D_n   )a= y^H \quad \text{and} \quad (  L_{N,n}^\top L_{N,n}  +   \lambda D_n )v= y^\eta, \quad \]
with $y^H_i=N^{-1}\sum_{j=1}^N \ell_i(\mu_j) H(\mu_j)$ and $y^\eta_i=N^{-1} \sum_{j=1}^N \ell_i(\mu_j) \eta_j$ for $i=1,\dotsc,n$.
Repeating the computations of \cite[Theorem 3]{cohen}, using also \eqref{eq:inverse1} and \eqref{eq:inverse2}, then leads to
\[\int_{\Omega_+} \left ( 2\|S_N^{\lambda,n}(H)\|^2_{L^2(\prob(K), \pp)} + 4\|S_N^{\lambda,n}(\mathcal{\eta})\|^2_{L^2(\prob(K), \pp)} \right )\de \PP \le 2\alpha_{\lambda, n}^2\frac{K_\lambda(n)}{N} \|H\|^2_{L^2(\prob(K), \pp)} + 4\alpha_{\lambda,n}^2\sigma^2 \frac{n}{N}, \]
where $\alpha_{\lambda, n}$ is as in \eqref{eq:alphaln}.
Combining all the estimates and doing some simple manipulations finally lead to the claimed upper bound.
\end{proof}

\begin{remark} \label{rem:bestapprox}
In practical applications it is, however, not always simple to determine the best approximating finite dimensional subspace $\mathcal{V}_n \subset \ccyl{\prob(K)}{\rmC_b^1(K)}$. In particular, optimizing the best approximation error $e(F,n)=\|P_n F-F\|^2_{L^2(\prob(K), \pp)}$ over $\mathcal{V}_n$ remains a difficult problem. For this reason, a possible heuristic remedy by employment of neural networks is outlined in Section~\ref{sec:adversarialnetwork} below. Note however that cylinder functions are $L^2$-dense so that the projection error can be made arbitrarily small. 
Moreover, for the specific Wasserstein distance function $\mu \to F^{W_2}_\theta(\mu):=W_2(\theta,\mu)$, which is the relevant example in our paper, we prove in Proposition \ref{prop:seconda}  for $q=2$ that any subspace $\mathcal V_n$ containing the random cylinder function $F_k$ will fulfill - {\it with high probability} - a quantifiable rate of convergence as in \eqref{eq:precest} (the result is in expectation, but an obvious application of Markow inequality yields the result with high probability). Hence, for the Wasserstein distance function we do have a simple recipe to build such a space $\mathcal V_n$ at least with high probability (just pick one subspace $\mathcal V_n$ for which $F_k \in \mathcal V_n$) and we can quantify $e(F_\theta^{W_2},n)$ by the  observation:
\[ e(F_\theta^{W_2},n) := \|P_n F_\theta^{W_2}-F_\theta^{W_2}\|^2_{L^2(\prob(K), \pp)} = \min_{G \in \mathcal V_n} \|G-F_\theta^{W_2}\|^2_{L^2(\prob(K), \pp)} \leq \|F_k-F_\theta^{W_2}\|^2_{L^2(\prob(K), \pp)}. \] 
In order to obtain an estimate depending on $n$, one has to choose $k=k(n)$ depending on $n$. 
\nc 
\end{remark}

\begin{remark}\label{rem:infHilb}
The results in this section can most generally be cast within the context of infinitesimally Hilbertian metric measure spaces, denoted as $(\SS, \sfd, \mm)$. These spaces are defined as abstract metric spaces where the function space $H^{1,2}(\SS, \sfd, \mm)$ forms a Hilbert space or, equivalently, where the associated $2$-Cheeger energy exhibits quadratic behavior, as elucidated, for instance, in~\cite[Section 4.3]{GMR15}.

By formulating the Tikhonov approximation for bounded functions mapping from $\SS$ to the real numbers, and by replacing the pre-Cheeger energy $\pCE_{2,\pp_N}$ with the abstract Cheeger energy $\CE_{2,\pp_N}$ as presented in equation \eqref{eq:CheegerEnergy}, Theorem \ref{teo:cohen2} and Theorem \ref{teo:cohen3} can be straightforwardly extended to accommodate suitably chosen abstract sequence of finite-dimensional subspaces  $\mathcal{V}_n \subset \Lip_b(\SS, \sfd)$.

Consequently, the fundamental merit of the results in this paper can be understood in their ability to specify such an abstract framework \cite{FHSXX} to the interesting case of $\SS = \prob(K)$. In fact, in this particular scenario, the results are practically and computationally applicable due to the density of cylinder functions, their numerical approximability, and the explicit computability of the pre-Cheeger energy $\pCE_{2,\pp_N}$.
\end{remark}

\section{Solving Euler--Lagrange equations} \label{sec:eulerlagrange} 
In this section, we introduce the Euler--Lagrange equation related to the risk minimization problem from Section \ref{sec:cons1}, and study the corresponding saddle point problem in Section \ref{sec:spp}.\nc

{We consider once again the functional $\overline{\JJ}:\Hot \to \R$ 
from Section~\ref{sec:cons1}, but rather study its Euler--Lagrange equation than directly its minimization. We recall that
\begin{align} \label{eq:energyfunctional}
\overline{\JJ}(G)=\norm{G-\widetilde{F}}^2_{L^2_\mm}+ \lambda \CE_{2,\mm}(G), \qquad G \in \Hot,
\end{align}
where $\norm{\cdot}_{L^2_\mm}:=\norm{\cdot}_{\Lot}$ and $\widetilde{F}=F+\eta$ for some target function $F \in \Lot$ and some unknown additive noise $\eta \in \Lot$.}

\begin{theorem}[Euler--Lagrange equation]
The functional $\overline{\JJ}$ from~\eqref{eq:energyfunctional} has a unique minimizer, which is equivalently the unique solution of the linear equation
\begin{align} \label{eq:EL_explicit}
(G-\widetilde{F},H)_{L^2_\mm}+\lambda \CE_{2,\mm}(G,H)=0 \qquad \text{for all} \ H \in \Hot,
\end{align}
where $\CE_{2,\mm}(\cdot,\cdot)$ is introduced in Remark \ref{rem:precheeger}.
\end{theorem}

\begin{proof}
The proof follows from the direct method of calculus of variations. Indeed, we have that the functional $\overline{\JJ}$, cf.~\eqref{eq:energyfunctional}, is Fr\'{e}chet-differentiable, strictly convex, and weakly coercive, and thus has a unique minimizer; we refer to \cite[Theorem 25.E]{Zeidler:90}. Furthermore, by~\cite[Theorem 25.F]{Zeidler:90}, the unique minimizer of $\overline{\JJ}$ is equivalently the unique solution of the Euler--Lagrange equation
\begin{align} \label{eq:ELdual}
G \in \Hot: \qquad \dprod{\delta \overline{\JJ}(G),H}=0 \quad \text{for all} \ H \in \Hot,
\end{align}
where $\delta \overline{\JJ}$ denotes the Fr\'{e}chet-derivative of $\overline{\JJ}$ and $\dprod{\cdot,\cdot}$ signifies the duality pairing between $\Hot$ and its dual space. A straightforward calculation reveals that 
\begin{align} \label{eq:Gateaux}
 \frac{1}{2} \dprod{\delta \overline{\JJ}(G),H}=(G-\widetilde{F},H)_{L^2_\mm} +\lambda \CE_{2,\mm} (G,H), \qquad G,H \in \Hot,
\end{align}
and thus concludes the proof.
\end{proof}

In particular, rather than minimizing the functional~\eqref{eq:energyfunctional}, one may solve the linear equation~\eqref{eq:EL_explicit}.\\

We shall now make a link to Section~\ref{sec:cohen}. In particular, let $K$ be a compact subset of $\mathbb{R}^d$ and $\pp$ a probability measure on $\mathcal{P}(K)$. We further consider a finite dimensional subspace of cylinder functions $\mathcal{V}_n \subset \ccyl{\prob(K)}{\rmC_b^1(K)} \subset H^{1,2}(\prob(K), W_2, \pp)$ and $N$ i.i.d.~points $\mu_1, \dots, \mu_N$ distributed on $\prob(K)$ according to $\pp$. As in the previous section, we define the (random) empirical measure $\pp_N:= \frac{1}{N}\sum_{i=1}^N \delta_{\mu_i}$. Then, we are interested in the minimization problem
\begin{align} \label{eq:finempproblem}
\argmin_{G \in \mathcal{V}_n} \JJ_{N,\widetilde{F}}^\lambda(G),
\end{align}
where, for $G \in \mathcal{V}_n$,
\begin{align} \label{eq:finempfunctional}
\JJ_{N,\widetilde{F}}^\lambda(G):=\norm{G-\widetilde{F}}^2_{L^2_{\pp_N}}+ \lambda \pCE_{2,\pp_N}(G);
\end{align}
cp.~Definition~\ref{def:someop}.
Analogously as before, the ensuing result holds.

\begin{theorem}
The optimization problem~\eqref{eq:finempproblem} has a unique solution, which is equivalently the unique solution of the finite dimensional linear equation
\begin{align} \label{eq:discreteEL}
(G-\widetilde{F},H)_{L^2_{\pp_N}}+\lambda \pCE_{2,\pp_N}(G,H)=0 \qquad \text{for all} \ H \in \mathcal{V}_n,
\end{align}
where, for $G,H \in \Cot$,
\begin{align} \label{eq:preip}
    \pCE_{2,\mm}(G,H):= \int_{\mathcal{P}_2(\mathbb{R}^d)} \int_{\mathbb{R}^d} \rmD G(\mu,x) \cdot \rmD H(\mu,x) \de\mu(x) \, \de\mm(\mu);
\end{align}
here, $\rmD G(\mu,x)$ and $\rmD H(\mu,x)$ are defined as in~\eqref{eq:thediff}.
\end{theorem}
Indeed, the unique minimizer is given by $S_{N}^{\lambda, n}(\widetilde{F})$ stemming from Definition~\ref{def:someop}, and thus Theorem~\ref{teo:cohen3} directly applies, which allows to explain the behavior of the solution as $N\to \infty$.

\subsection{Corresponding saddle point problem }\label{sec:spp}

 Upon introducing the operator norm
 \begin{align} \label{eq:operatornorm}
 \norm{\delta \overline{\JJ}(G)}_{\mathrm{op}}:=\sup_{\substack{H \in \Hot \\ H \neq 0}} \frac{ \frac{1}{2} \dprod{\delta \overline{\JJ}(G),H}}{|H|_{H^{1,2}_\mm}},
 \end{align}
 the equation~\eqref{eq:EL_explicit} amounts to finding an element $G \in \Hot$ such that
 \begin{align} \label{eq:operatoreq}
 \norm{\delta \overline{\JJ}(G)}_{\mathrm{op}}=0.
 \end{align}
  Indeed, in light of~\eqref{eq:Gateaux} we have that 
  \[
    \sup_{\substack{H \in \Hot \\ H \neq 0}} \frac{(G-\widetilde{F},H)_{L^2_\mm} +\lambda \CE_{2,\mm}(G,H)}{|H|_{H^{1,2}_\mm}} \geq 0 \qquad \text{for all} \ G \in \Hot,
  \]
  with equality if and only if $G$ is the unique minimizer of the functional $\overline{\JJ}$, or equivalently the unique solution of the Euler--Lagrange equation~\eqref{eq:EL_explicit}.  For that reason, and in view of a computational framework that is outlined in Section~\ref{sec:adversarialnetwork}, we are interested in the saddle point problem
 \begin{align} \label{eq:saddelpointcont}
\inf_{G \in \Hot} \sup_{\substack{H \in \Hot \\ H \neq 0}} \frac{(G-\widetilde{F},H)_{L^2_\mm} +\lambda \CE_{2,\mm}(G,H)}{|H|_{H^{1,2}_\mm}} =0.
 \end{align}
 
Even for two explicitly given cylinder functions $G,H \in \Cot$ we won't be able to determine the Cheeger inner product $\CE_{2,\mm}(G,H)$ and the norm $|H|_{H^{1,2}_\mm}$, respectively. However, thanks to our Proposition~\ref{prop:equiv} below, we may replace those objects by the explicit  pre-Cheeger inner product $\pCE_{2,\mm}(G,H)$ from~\eqref{eq:preip}
and the norm $|H|_{pH^{1,2}_\mm}$ from~\eqref{eq:prenorm},
respectively.

\begin{proposition} \label{prop:equiv}
The saddle-point problem~\eqref{eq:saddelpointcont} attains the same value as\nc
\begin{align} \label{eq:saddlepoint2}
\inf_{G \in \Cot} \sup_{\substack{H \in \Cot \\ H \neq 0}} \frac{(G-\widetilde{F},H)_{L^2_\mm} +\lambda \pCE_{2,\mm}(G,H)}{|H|_{pH^{1,2}_\mm}}.
\end{align}
In particular, both of them are zero.
\end{proposition}

\begin{proof}
For the ease of notation, let us define $\CC:=\Cot, \HH:=\Hot$ and $\CC_0:= \CC \setminus \{0\}, \HH_0:= \HH \setminus \{0\}$. Moreover, for $(G,H) \in \HH \times \HH_0$, define
\[
\mathcal{F}(G,H):=\frac{(G-\widetilde{F},H)_{L^2_\mm}+\lambda \CE_{2,\mm}(G,H)}{|H|_{H^{1,2}_\mm}}
\]
and, for $(G,H) \in \CC\times \CC_0$,
\[
\widetilde{\mathcal{F}}(G,H):=\frac{(G-\widetilde{F},H)_{L^2_\mm}+\lambda \pCE_{2,\mm}(G,H)}{|H|_{pH^{1,2}_\mm}}.
\]
First of all we show that
\[  \inf_{G \in \HH} \sup_{H \in \HH_0} \mathcal{F}(G,H) \le \inf_{G \in \CC} \sup_{H \in \CC_0} \widetilde{\mathcal{F}}(G,H). \]
For any $G \in \CC$ and $\bar{H} \in \HH_0$ we note that 
\begin{align*}
\mathcal{F}(G,\bar{H}) = \frac{(G-\widetilde{F}, \bar{H})_{L_\mm^2} + \lambda \int \rmD_\mm G \cdot \rmD_\mm \bar{H}}{|\bar{H}|_{H^{1,2}_\mm}} = \frac{(G-\widetilde{F}, \bar{H})_{L_\mm^2} + \lambda \int \rmD G \cdot \rmD_\mm \bar{H}}{|\bar{H}|_{H^{1,2}_\mm}},
\end{align*}
where $\rmD_\mm \bar{H}$ and $\rmD_\mm G$ are defined as in Remark~\ref{rem:precheeger} and the latter equality is due to \eqref{eq:useeq}. We further know that there exists a sequence $(H_n)_n \subset \CC_0$ such that
\[ H_n \to \bar{H} \quad \text{and} \quad \rmD H_n \to \rmD_\mm \bar{H}.\]
In turn, together with the above calculation, we obtain that
\[
\mathcal{F}(G,\bar{H})=\lim_{n \to \infty} \widetilde{\mathcal{F}}(G,H_n) \leq \sup_{H \in \CC_0} \widetilde{\mathcal{F}}(G,H).
\]
Since $\bar{H} \in \HH_0$ was arbitrary, we can pass to the supremum in $\HH_0$ obtaining that
\[ \sup_{H \in \HH_0}\mathcal{F}(G,H)\le  \sup_{H \in \CC_0} \tilde{\mathcal{F}}(G,H) \]
and consequently
\[  \inf_{G \in \HH} \sup_{H \in \HH_0}\mathcal{F}(G,H) \le \inf_{G \in \CC} \sup_{H \in \HH_0}\mathcal{F}(G,H) \le \inf_{G \in \CC} \sup_{H \in \CC_0} \widetilde{\mathcal{F}}(G,H). \]

We now show the converse inequality: for any $G \in \HH$ we can find an element $G_\eps \in \CC$ such that
\[ |G-G_\eps|_{L_\mm^2} < \eps \quad \text{and} \quad |\rmD_\mm G - \rmD G_\eps|_{L_{\bmm}^2} < \eps.\]
Therefore we have for any $H \in \CC_0$ that
\begin{multline*}
    \biggl |\frac{(G-\widetilde{F}, H)_{L_\mm^2} + \lambda \int \rmD_\mm G \cdot \rmD H }{|H|_{H^{1,2}_\mm}} - \frac{(G_\eps-\widetilde{F}, H)_{L_\mm^2} + \lambda \int \rmD G_\eps \cdot \rmD H }{|H|_{H^{1,2}_\mm}} \biggr | \\
    \le \eps \frac{ |H|_{L_\mm^2} + \lambda |\rmD H|_{L_{\bmm}^2}}{|H|_{H^{1,2}_\mm}} \le \eps (\lambda^2 +1)^{1/2}.
\end{multline*}
Consequently, and in light of \eqref{eq:useeq}, we find that
\begin{align*}
\sup_{H \in \CC_0} \mathcal{F}(G,H) &\ge \sup_{H \in \CC_0} \frac{(G_\eps-\widetilde{F}, H)_{L_{\mm}^2} + \lambda \int \rmD G_\eps \cdot \rmD H }{|H|_{H^{1,2}_\mm}} - \eps(\lambda^2+1)^{1/2}    \\
& \ge \sup_{H \in \CC_0} \frac{(G_\eps-\widetilde{F}, H)_{L_\mm^2} + \lambda \int \rmD G_\eps \cdot \rmD H }{|H|_{pH^{1,2}_\mm}} - \eps(\lambda^2+1)^{1/2}\\
& = \sup_{H \in \CC_0} \tilde{\mathcal{F}}(G_\eps, H) - \eps(\lambda^2+1)^{1/2}\\
& \ge \inf_{G \in \CC} \sup_{H \in \CC_0} \tilde{\mathcal{F}}(G,H) - \eps(\lambda^2+1)^{1/2}. 
\end{align*}
Passing to the limit as $\eps \downarrow 0$ we thus get
\[ \sup_{H \in \HH_0} \mathcal{F}(G,H) \ge \sup_{H \in \CC_0} \mathcal{F}(G,H) \ge \inf_{G \in \CC} \sup_{H \in \CC_0} \widetilde{\mathcal{F}}(G,H) \]
for any $G \in \HH$. Passing to the infimum w.r.t.~$G \in \HH$ leads to the conclusion.
\end{proof}

\begin{remark} \label{rem:empmeasure}
As previously stated, in practice we neither have $\mm$ nor $\widetilde{F}$ at our disposal, but only an empirical approximation
\[
\frac{1}{N} \sum_{j=1}^N \delta_{\mu_j} =: \pp_N
\]
and the corresponding (noisy) function values
\[
y_j=\widetilde{F}(\mu_j), \qquad j \in \{1,2,\dotsc,N\};
\]
in particular, here we assume that $\mm=\pp$ is a probability measure. Furthermore, for the sake of actual numerical implementation, we shall again consider a finite dimensional subspace of cylinder functions $\mathcal{V}_n \subset \ccyl{\prob(K)}{\rmC_b^1(K)}$. Then, the corresponding finite dimensional empirical saddle point problem reads as
\begin{align} \label{eq:empiricalsaddlepoint}
\inf_{G \in \mathcal{V}_n } \sup_{\substack{H \in \mathcal{V}_n \\ H \neq 0}} \frac{(G-\widetilde{F},H)_{L^2_{\pp_N}} +\lambda \pCE_{2,\pp_N}(G,H)}{|H|_{pH^{1,2}_{\pp_N}}},
\end{align}
where
\begin{align*}
(G-\widetilde{F},H)_{L^2_{\pp_N}}&:=\frac{1}{N}\sum_{j=1}^N (G(\mu_j)-\widetilde{F}(\mu_j)) H(\mu_j)=\frac{1}{N}\sum_{j=1}^N (G(\mu_j)-y_j) H(\mu_j),\\
\pCE_{2,\pp_N}(G,H)&:=\frac{1}{N} \sum_{j=1}^N \int_{\mathbb{R}^d} \rmD G(x,\mu_j) \cdot \rmD H(x,\mu_j) \, \d \mu_j(x), 
\end{align*}
and
\begin{align} \label{eq:discreteH12}
|H|^2_{pH^{1,2}_{\pp_N}}:=\frac{1}{N} \sum_{j=1}^N \Big(|H(\mu_j)|^2+\int_{\mathbb{R}^d} |\rmD H(x,\mu_j)|^2 \, \de\mu_j(x) \Big).
\end{align}
We emphasize again that in practice it is not simple to find the best approximating subspace $\mathcal{V}_n$, see Remark~\ref{rem:bestapprox}. The heuristic use of trained neural networks is again the computational remedy to  leverage the saddle point problem~\eqref{eq:empiricalsaddlepoint} as shown in Section~\ref{sec:adversarialnetwork} below. 
\end{remark}

\section{Numerical experiments}\label{sec:num}

In this section, we run some numerical tests to experimentally investigate the efficacy of some of our analytical results; in all our experiments below, the goal is to approximate the Wasserstein distance function $\FWT$ from~\eqref{eq:WDF2intro}. For that purpose, we consider two distinct datasets,  namely MNIST and CIFAR-10 as well-known databases of images. Back in the year 2000, image retrieval from a database was one of the very first applications of the so-called Earth-mover distance, which is the $1$-Wasserstein distance \cite{Rubner2000TheEM}. In those applications, the images were first encoded as frequency histograms of their gray-levels (hence, discrete probability densities). In this section, we artificially treat images themselves as discrete probability densities, by renormalization. The scope of our experiments is not necessarily to offer a better or faster method for image retrieval or classification, but rather to have  meaningful databases to test our theoretical results, namely the learning of the Wasserstein distance from data. In particular, our results show that our method (see, Exp.~\ref{exp:trainable}) does allow to compute at the same or better accuracy the Wasserstein distances over the test set (e.g., extracted from MNIST or CIFAR-10) almost {4-5 times faster}, { \underline{including}  training time}, than the employment of the sole (traditional) optimal transport algorithms, see Remark \ref{rem61}. This is certainly the main take home message of this section and, perhaps, of the entire paper.\nc
\\

\textbf{MNIST.} The MNIST dataset is a large collection of handwritten digits. Thereby, the training set $\ttrain$ consists of $60 \, 000$ elements, $\{\mu_k\}_{k=1}^{60 \, 000}$, whereas the test set $\ttest$ is composed of $10 \, 000$ datapoints, $\{\nu_k\}_{k=1}^{10 \, 000}$. Each image in the MNIST dataset is given by $28 \times 28$ pixels in grayscale. For our experiments, we have normalized each image such that they correspond to probability measures. Moreover, our reference measure $\vartheta$ is the one corresponding to the barycenter of all the images of the digit "0" in the training set $\ttrain$; this image has been computed with the Python open source library \emph{POT}, cf.~\cite{flamary2021pot}. In the following, we use the notions of an image and its corresponding probability measure interchangeably. We have further used the \emph{ot.emd} solver from the POT library to compute the 2-Wasserstein distance of each image in the training and test sets to the reference image. For the training set, we have also computed the corresponding Kantorovich potentials. \\

\textbf{CIFAR-10.} The CIFAR-10 dataset, cf.~\cite{cifar10}, is a collection of $32 \times 32$ colour images in 10 classes, whereby the training set $\ttrain$ consists of $50 \, 000$ elements and the test set $\ttest$ is composed of $10 \, 000$ elements. As for the MNIST dataset, we have normalized each image so that we can work with probability measures. The reference measure $\vartheta$ was randomly picked from the training set, and happened to be the image of a deer. As before, we computed the 2-Wasserstein distances as well as the Kantorovich potentials with the \emph{ot.emd} solver from the Python open source library POT.  

\begin{remark}\label{rem61}
We emphasize that the advantage of the approach presented herein is the superior evaluation time. To underline this property, we compare the time needed to determine the distance of every element in the test set to the reference measure in the case of the CIFAR-10 dataset by our approximating function from Experiment~\ref{exp:trainable} below with two built-in Python functions from the \emph{POT} library. Indeed, we consider the functions: 
\begin{enumerate}[1.]
    \item \emph{ot.emd}, which computes the 2-Wasserstein distance with the algorithm from \cite{BPPH:2011}, and
    \item \emph{ot.bregman.sinkhorn2}, which employs the Sinkhorn-Knopp matrix scaling algorithm from~\cite{Cuturi:2013} for the entropy regularized optimal transport problem; here, we set the regularization parameter to $0.1$ and the stop threshold to $0.001$.
\end{enumerate}
In Table~\ref{table:comptime}, the average computational time over ten runs of the three considered numerical schemes is shown; we note that the time is scaled in such a way that the computational time of the approach introduced herein is one unit. As we can clearly observe, the method presented in this work is by orders of magnitude faster than the two built-in functions from the \emph{POT} library. Rather surprisingly, the \emph{ot.bregman.sinkhorn2} solver took even longer than the \emph{ot.emd} solver and lead to a mean relative error of approximately $0.1956$; this is considerably larger than the mean relative error obtained by our deep learning scheme. Indeed, this may seem to be an unfair comparison, since our approach relies on a training on a suitable training set, whereas the other two solvers can be applied to compute numerically the Wasserstein distance between any two measures without  preceding training. We need to reiterate that our novel computational scheme is mainly suited for large datasets. In particular, for the experiment under consideration we made the following observation: The {\it total computational time} of training of the neural network from Experiment~\ref{exp:trainable} over 100 epochs using the full training set {\it and} subsequently evaluating the Wasserstein distance function on the test set is {\it significantly smaller} (the factor being approximately 4.6) than the one of \emph{ot.bregman.sinkhorn2} as specified above, and leads to a smaller error, see Figure~\ref{fig:trainable}. This observation can be generalized as follows: \\

 \begin{tcolorbox}[width=\linewidth, sharp corners=all, colback=white!95!black,title=General pipeline and main result]
Given a large set of distributions for which the pairwise Wasserstein distance should be computed. Then, instead of applying a standard solver, such as the Sinkhorn algorithm, for all the pairs, the following approach may lead to a significant improvement of the computational time: Split the data into training and test sets. Subsequently, compute the Wasserstein distances for the elements in the training set (with a standard solver), which then serves as the basis for the training of the neural network. Last, the trained neural network is employed to approximate the Wasserstein distance for the elements in the test set. \\
Moreover, for any two new distributions (obtained, e.g., through a measurement) belonging to the same data, the Wasserstein distance can be evaluated almost for free.
\end{tcolorbox}
\vspace{1em}

We emphasise that the latter property could be of tremendous importance in applications, in which a function, e.g., the Wasserstein distance, has to be computed in a split second based on previous experience. As a relevant example we mention autonomous driving and refer, e.g., to~\cite{Liuetal:2020}.\nc

\end{remark}

\begin{table}[h] 
\centering
\begin{tabular}{| C{3.5cm} | C{3.5cm} | C{3.3cm} | C{3.7cm} |}
\hline
& DNN from Exp.~\ref{exp:trainable} & ot.emd & ot.bregman.sinkhorn2 \\ \hline
Computational time & 1 & $\approx 1.8767 \cdot 10^4$ & $\approx 8.4439\cdot 10^4$\\
\hline
\end{tabular}
\caption{Computational time to determine the distance of every element in the test set to the reference element in case of the CIFAR-10 dataset.}
\label{table:comptime}
\end{table}

We note that in the context of the MNIST and CIFAR-10 datasets, the underlying base sets are discrete. For that reason, we elaborated analogous results to the ones of Section~\ref{sec:conofpot} for this considerably simpler case, which however are directly applicable in many applications dealing with a finite dataset; we refer to Appendix~\ref{sec:4} for details.
A comparison of a deep learning pipeline as in Experiment~\ref{exp:trainable} with traditional optimal transport algorithms, namely {\it ot.emd} or {\it ot.bregman.sinkhorn2}, may be considered unsuited as the latter are not pre-trainable. For this reason, in the following section we develop a trainable baseline for comparisons based on Proposition \ref{prop:seconda} and the results in Appendix~\ref{sec:4}, which provide theoretical guarantees of convergence with generalization rates.

\subsection{Baseline -- Approximation by cylinder functions} \label{sec:baseline}
\nc

 Based on the MNIST and CIFAR-10 datasets, we shall experimentally investigate how well we can approximate the function $\FWT$ by $\GW{}$, cf.~\eqref{eq:FGDef}, based on the set of pre-computed Kantorovich potentials. Since we have theoretical guarantees for this approach, cf.~Section~\ref{sec:conofpot} and~Appendix~\ref{sec:4}, and are not aware of any similar analysis in the literature, we will call this approach our baseline.

\experiment \label{exp:testind}
For each 
\begin{align*}
j \in \mathcal{I}_{\mathrm{pot}}^M:=\{10,50,100,250,500,1 \, 000,2 \, 500,5 \, 000,10 \ 000,20 \, 000,40 \, 000,60 \, 000\}
\end{align*}
and 
\begin{align*}
j \in \mathcal{I}_{\mathrm{pot}}^C:=\{10,50,100,250,500,1 \, 000,2 \, 500,5 \, 000,10 \ 000,20 \, 000,30 \, 000,40 \, 000, 50 \, 000\}
\end{align*}
for the MNIST and CIFAR-10 datasets, respectively, we randomly picked subsets $I_j \subseteq \{1,2,\dotsc, 60 \, 000\}$ and $I_j \subseteq \{1,2,\dotsc, 50 \, 000\}$, respectively, with $|I_j|=j$, satisfying the hierarchical property $I_j \subseteq I_{j'}$ for $j \leq j'$; here, $|I_j|$ denotes the cardinality of the set $I_j$. Subsequently, we computed the relative error $\nicefrac{\left(\FWT(\nu)-\GW{I_j}(\nu)\right)}{\FWT(\nu)}>0$ for one hundred elements $\nu \in \ttest$ drawn from the test set at random, as well as the mean of those individual errors. In Figure~\ref{fig:testind} we can observe how the individual errors, as well as their mean, decay with an increasing number $j$ of considered potentials. However, in the case of the MNIST dataset (Figure~\ref{fig:testind} left), even when we consider all the $60 \, 000$ potentials related to the entire training set, the mean relative error is still above $0.1$. The situation is indeed slightly better for the CIFAR-10 dataset (Figure~\ref{fig:testind} right), but we still have a mean relative error of approximately $0.02$ when all the potentials arising from the entire training set are considered.

\begin{figure}[h]
	\includegraphics[width=0.49\textwidth]{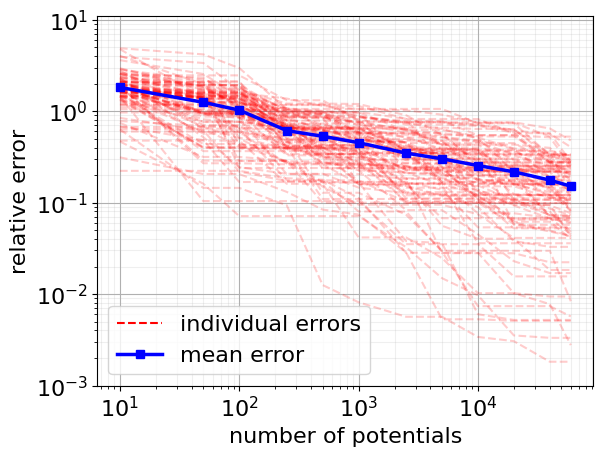} \hfill
 \includegraphics[width=0.49\textwidth]{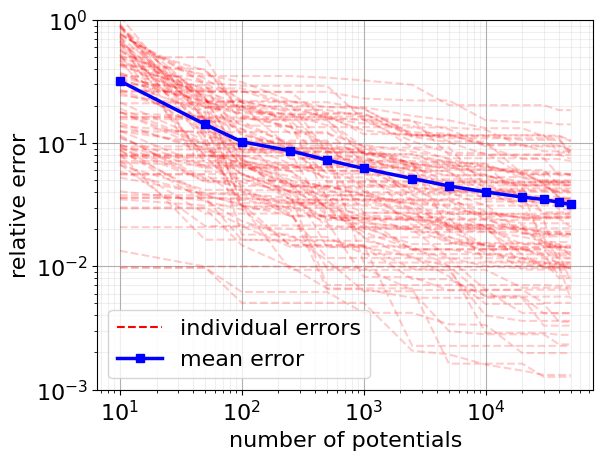}
	\caption{Experiment~\ref{exp:testind}: Plot of the relative approximation error $\nicefrac{\left(\FWT(\nu)-\GW{I_j}(\nu)\right)}{\FWT(\nu)}$ for one hundred randomly drawn elements $\nu \in \ttest$ and their mean against the number of potentials $j$. Left: MNIST. Right: CIFAR-10.}
	\label{fig:testind}
\end{figure}

\experiment \label{exp:testmeantot}

We run a similar experiment as the one from before. However, for each $j \in \mathcal{I}_{\mathrm{pot}}^M$ and $j \in \mathcal{I}_{\mathrm{pot}}^C$, respectively, we only consider the mean of the relative errors $\nicefrac{\left(\FWT(\nu)-\GW{I_j}(\nu)\right)}{\FWT(\nu)}>0$ over all the elements $\nu \in \ttest$. This procedure is repeated twenty times: for each loop the sets $I_j$ are drawn at random once again, but still satisfying the hierarchical property. Subsequently, we plot each of those twenty mean errors, as well as their mean, against the number of pre-computed potentials $j$ ; see Figure~\ref{fig:testmeantot} (left for MNIST and right for CIFAR-10). For a small number of potentials, the mean error slightly differs for the twenty draws. However, for an increasing number of potentials, they more and more coincide, and finally are equal when all the potentials of the training set are considered. As a comparison, we also depict the mean error obtained when only the potentials corresponding to the digit ``1" in the training set for MNIST and the images of ships for CIFAR-10, respectively, are taken into account. In the context of the MNSIT dataset this indeed yields a much larger mean error compared to a comparable number of potentials drawn at random. In the setting of the CIFAR-10 dataset, we only have a relatively small improvement by drawing the potentials over all the possible classes, which is rather surprising. 

\begin{figure}[h]
	\includegraphics[width=0.49\textwidth]{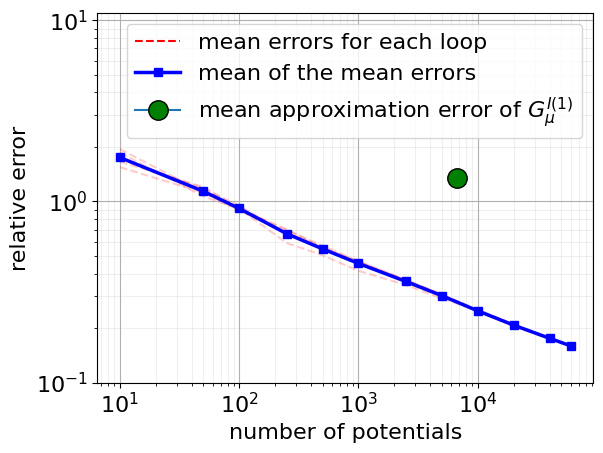} \hfill
 \includegraphics[width=0.49\textwidth]{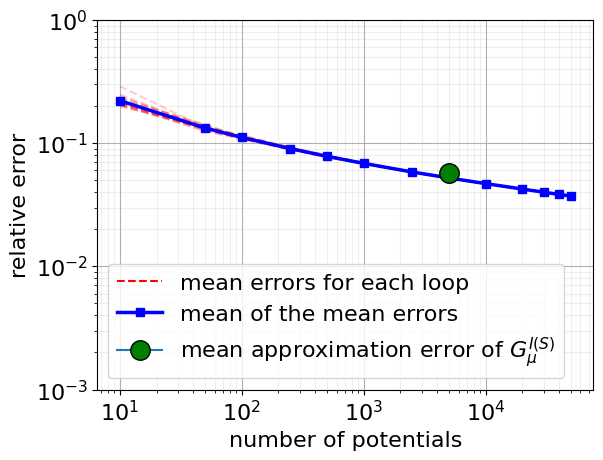}
	\caption{Experiment~\ref{exp:testmeantot}: Plot of the mean relative errors $\frac{1}{|\ttest|} \sum_{\nu \in \ttest}\nicefrac{\left(\FWT(\nu)-\GW{I_j}(\nu)\right)}{\FWT(\nu)}$ for twenty draws of the index sets $I_j$ and their mean against the number of potentials $j$. Moreover, we further portray the mean error $\frac{1}{|\ttest|} \sum_{\nu \in \ttest}\nicefrac{\left(\FWT(\nu)-\GW{I(1)}(\nu)\right)}{\FWT(\nu)}$ (left) and $\frac{1}{|\ttest|} \sum_{\nu \in \ttest}\nicefrac{\left(\FWT(\nu)-\GW{I(S)}(\nu)\right)}{\FWT(\nu)}$ (right), where $I(1)$ and $I(S)$ are the sets of indices of the images of the digit ``1" and ships, respectively, in the training set of MNIST and CIFAR-10, respectively. Left: MNIST. Right: CIFAR-10.}
	\label{fig:testmeantot}
\end{figure}

\experiment \label{exp:trainind}

We repeat Experiment~\ref{exp:testind}, but instead of picking one hundred elements from the test set $\ttest$, they are drawn at random from the training set $\ttrain$; i.e., we let $\{\mu_\ell\}_{\ell \in I_{\mathrm{RT}}} \subset \ttrain$, where $I_{\mathrm{RT}} \subset \{1,2,\dotsc,60 \, 000\}$ for the MNIST dataset and $I_{\mathrm{RT}} \subset \{1,2,\dotsc, 50 \, 000\}$ for the CIFAR-10 dataset, respectively, with $|I_{\mathrm{RT}}|=100$. Apart from that, and the modified sets \footnotesize
\[
\mathcal{I}_{\mathrm{pot}}^{M'}:=\{10,50,100,250,500,1 \, 000,2\,500,5\,000,10\,000,20\,000,40\,000,50\,000,55\,000,57\,000,59\,000,60\,000\},
\]
\normalsize and \footnotesize
\[
\mathcal{I}_{\mathrm{pot}}^{C'}:=\{10,50,100,250,500,1 \, 000,2 \, 500,5 \ 000,10 \, 000,20 \, 000,30 \ 000,40 \, 000,45 \, 000,47 \ 000,49 \, 000,50 \, 000\},
\]
\normalsize respectively, the experimental setup is exactly the same as in Experiment~\ref{exp:testind}. We emphasize that if $\ell \in I_j$, then $\FWT(\mu_\ell)-\GW{I_j}(\mu_\ell)=0$, and, in turn, $\frac{1}{|I_{\mathrm{RT}}|}\sum_{\ell \in I_{\mathrm{RT}}} \nicefrac{\left(\FWT(\mu_\ell)-\GW{I_{60 \, 000}}(\mu_\ell)\right)}{\FWT(\mu_\ell)}=0$ for the MNIST dataset and $\frac{1}{|I_{\mathrm{RT}}|}\sum_{\ell \in I_{\mathrm{RT}}} \nicefrac{\left(\FWT(\mu_\ell)-\GW{I_{50 \, 000}}(\mu_\ell)\right)}{\FWT(\mu_\ell)}=0$ for the CIFAR-10 dataset, respectively. As we can witness from Figure~\ref{fig:trainind} (left for MNIST and right for CIFAR-10), the individual errors $\nicefrac{\left(\FWT(\mu_\ell)-\GW{I_j}(\mu_\ell)\right)}{\FWT(\mu_\ell)}$ behave qualitatively very similar as for draws from the test set, cp.~Figure~\ref{fig:testind}, as long as $\ell \not \in I_j$. In particular, the smaller mean error (for a large number of potentials) in the given setting arises due to the elements with zero error contribution. 

\begin{figure}[h]
	\includegraphics[width=0.49\textwidth]{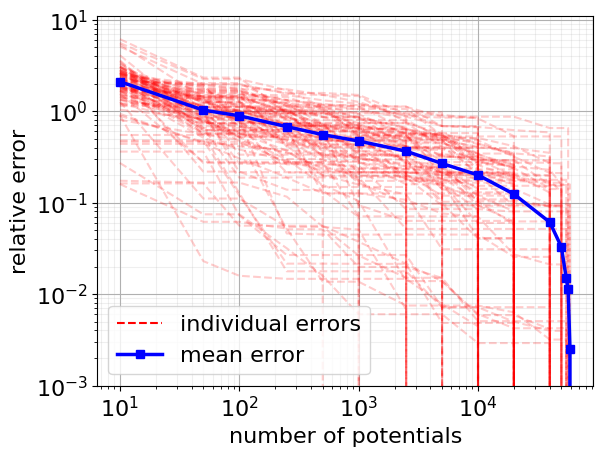} \hfill
 \includegraphics[width=0.49\textwidth]{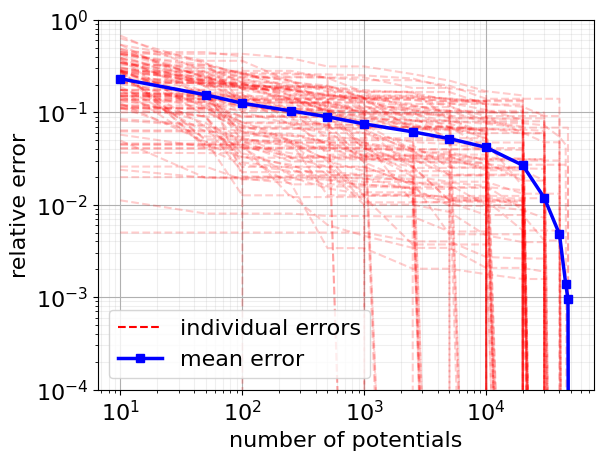}
	\caption{Experiment~\ref{exp:trainind}: Plot of the relative approximation error $\nicefrac{\left(\FWT(\nu)-\GW{I_j}(\nu)\right)}{\FWT(\nu)}$ for one hundred randomly drawn elements $\nu \in \ttrain$ and their mean against the number of potentials $j$. Left: MNIST. Right: CIFAR-10.}
	\label{fig:trainind}
\end{figure}

Finally, we also repeat the Experiment~\ref{exp:testmeantot}, however we replace the test set with $10 \, 000$ random elements in the training set $\ttrain$. The corresponding error plot is illustrated in Figure~\ref{fig:trainmeantot}.

\begin{figure}[h]
	\includegraphics[width=0.49\textwidth]{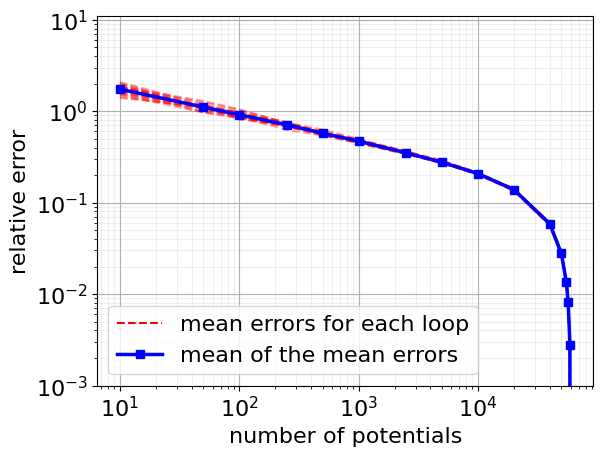} \hfill
 \includegraphics[width=0.49\textwidth]{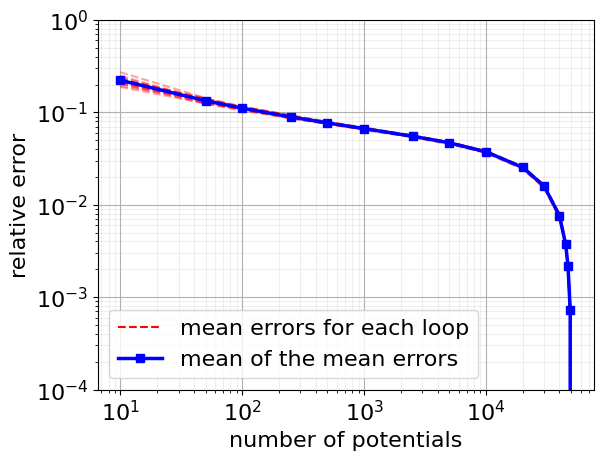}
	\caption{Experiment~\ref{exp:trainind} (second part): Based on $10 \, 000$ random elements $\{\mu_\ell\}_{\ell \in I_{\mathrm{RT}}} \subset \ttrain$, with $|I_{\mathrm{RT}}|=10 \, 000$, we plot the mean relative errors $\frac{1}{|I_{\mathrm{RT}}|} \sum_{\ell \in I_{\mathrm{RT}}}\nicefrac{\left(\FWT(\mu_\ell)-\GW{I_j}(\mu_\ell)\right)}{\FWT(\mu_\ell)}$ for twenty draws of the index sets $I_j$ and their mean against the number of potentials $j$ . Left: MNIST. Right: CIFAR-10.}
	\label{fig:trainmeantot}
\end{figure}

\subsection{Trainable baseline -- Approximation of suitable Kantorovich potentials by a deep neural network} \label{sec:deepneuralnetwork}
\nc
In this subsection, we first verify that the function $\GW{}$, cf.~\eqref{eq:FGDef}, can be realized by a deep neural network, and thereafter highlight that the approximation accuracy from our previous tests can be further enhanced by a suitable training of the network. We first note that, for any pair of continuous functions $\varphi,\psi \in \rmC (\ds)$, the mapping 
\[
\mu \mapsto \int_{\ds} \varphi \, \mathrm{d} \mu+ \int_{\ds} \psi \, \mathrm{d} \vartheta, \qquad \mu \in \mathcal{M}(\ds),
\]
is affine, where $\mathcal{M}(\ds)$ denotes the set of finite, signed measures on $\ds$. In turn, for any finite subset $I:=\{k_1,k_2,\dotsc,k_j\} \subset \mathbb{N}$, the mapping $\GV{I}:\mathcal{M}(\ds) \to \mathbb{R}^{j}$,
\begin{align} \label{eq:linmapfirsthidden}
    \GV{I}(\mu):=\left(\int_{\ds} \varphi_{k_1} \, \mathrm{d} \mu+ \int_{\ds} \psi_{k_1} \, \mathrm{d} \vartheta, \dotsc, \int_{\ds} \varphi_{k_j} \, \mathrm{d} \mu+ \int_{\ds} \psi_{k_j} \, \mathrm{d} \vartheta\right)^{\intercal},
\end{align}
is affine as well. Moreover, since the extension of the homeomorphism $J$ from~\eqref{eq:Jop}, $\tilde{J}:\mathcal{M}(\ds) \to \mathbb{R}^{d_{\ds}}$, is linear, we further have that $\LR{I}:=\GV{I} \circ \tilde{J}^{-1}:\mathbb{R}^{d_{\ds}} \to \mathbb{R}^{j}$ is affine, and thus can be represented jointly by a matrix $\ALL{I} \in \mathbb{R}^{j \times d_{\ds}}$ and a vector $\mathbf{b}_{\vartheta}^{I} \in \mathbb{R}^{j}$. In particular, with the specific numbering of $\ds$ that defines the mapping $\tilde{J}$, we have that
\begin{align} \label{eq:WMBV}
    \ALL{I}:=\begin{pmatrix}
        \varphi_{k_1}(\xx_1) & \varphi_{k_1}(\xx_2) & \cdots & \varphi_{k_1}(\xx_{d_{\ds}}) \\
        \varphi_{k_2}(\xx_1) & \varphi_{k_2}(\xx_2) &  & \vdots \\
        \vdots & & \ddots & \vdots \\
        \varphi_{k_j}(\xx_1) & \cdots & \cdots & \varphi_{k_j}(\xx_{d_{\ds}})
    \end{pmatrix}
    \quad
    and
    \quad
    \bmu{I}:=\left(\int_{\ds} \psi_{k_1} \, \mathrm{d} \vartheta, \dotsc, \int_{\ds} \psi_{k_j} \, \mathrm{d} \vartheta \right)^{\intercal},
\end{align}
respectively. We may consider $\ALL{I}$ and $\bmu{I}$ as the weight matrix and bias vector, respectively, of a shallow neural network with no hidden layers, which represents the mapping $\LR{I}$. We further note that $\GW{I}(\mu)=\max\left(\GV{I}(\mu)\right)=\max\left(\LR{I}(J(\mu))\right)$, $\mu \in \mathcal{P}(\ds)$, where the maximum is taken over the components of the vector $\GV{I}(\mu)=\LR{I}(J(\mu))$; moreover, the composition of functions, which are represented by deep neural networks, is the realization of the concatenation of the corresponding neural networks; we refer, e.g., to~\cite[Prop.~2.14]{JR:2020}. Hence, if we can represent the maxima function
\begin{align} \label{eq:maxfunc}
(x_1,\dotsc,x_n)^{\intercal} \mapsto \max\{x_1,\dotsc,x_n\}
\end{align}
by a deep neural network, then the same holds true for $\GW{I}$. It is indeed known in the literature that the maxima function~\eqref{eq:maxfunc} can be realized by a deep ReLU neural network, see, for instance, the works \cite{BJK:2022,GJS:2022,JR:2020}. 

\subsubsection{Deep neural network representation of the maxima function} \label{sec:maxima}

In the present work, we  only recall a (simplified) construction of the maxima function on $\mathbb{R}^{2^k}$, for some $k \in \mathbb{N}$. For that purpose, we first remind that the ReLU function $\mathfrak{r}:\mathbb{R} \to \mathbb{R}_{ \geq 0 }$ is defined by
\[
\mathfrak{r}(x):=\max\{0,x\}.
\]
In the following, against common convention, we number the weight matrices starting from the output layer towards the input layer; the reason for that is clarified below. Then, the realization of a neural network $\mathcal{A}=(A_l,A_{l-1},\dotsc,A_1)$ (without biases) with activation function $\mathfrak{r}$ is given by
\[
\NN_{\mathcal{A}}(x):=A_1(\mathfrak{r}(A_2(\mathfrak{r}\dotsc (A_{l-1}\mathfrak{r}(A_l x))\dotsc))),
\]
where $x$ is the input of the deep neural network and the activation function $\mathfrak{r}$ is applied componentwise. In the following, we shall construct the neural network representation of the maxima function~\eqref{eq:maxfunc} on $\mathbb{R}^n$, with $n=2^k$ for some $k \in \mathbb{N}$. Starting with $k=1$, the network that realizes~\eqref{eq:maxfunc} is given by $\mathcal{A}_1:=(A_2,A_1)$, where
\begin{align*}
A_2:=\begin{pmatrix}
1 & -1 \\
0 & 1 \\
0 & -1 \\
\end{pmatrix}
\qquad \text{and} \qquad
A_1:=\begin{pmatrix}
1 & 1 & -1
\end{pmatrix}.
\end{align*}
It is straightforward to check that, for any given $x=(x_1,x_2)^{\intercal}\in \mathbb{R}^2$, 
\begin{align*}
\NN_{\mathcal{A}_1}(x)=A_1\left(\mathfrak{r}(A_2 x)\right)=\max\{x_1-x_2,0\}+\max\{x_2,0\}-\max\{-x_2,0\}=\max\{x_1,x_2\}.
\end{align*}
Now let $B_1:=A_2$ and $C_1:=A_1$. Then, for given $\ell \in \mathbb{N}$, we define the block matrices
\begin{align*}
B_{\ell+1}:= \begin{pmatrix}
B_{\ell} & 0 \\
0 & B_{\ell}
\end{pmatrix} = \underbrace{\begin{pmatrix}
A_2 & 0 & \cdots  & 0 \\
0 & A_2 & \ddots & \vdots \\
\vdots & \ddots & \ddots  & 0 \\
0 & \cdots & 0 & A_2
\end{pmatrix}}_{2^{\ell} \ \text{times}} \in \mathbb{R}^{3 \cdot 2^\ell \times 2^{\ell+1}}
\end{align*}
and
\begin{align*}
C_{\ell+1} = \begin{pmatrix}
C_{\ell} & 0 \\
0 & C_{\ell}
\end{pmatrix} = \underbrace{\begin{pmatrix}
A_1 & 0 & \cdots  & 0 \\
0 & A_1 & \ddots & \vdots \\
\vdots & \ddots & \ddots  & 0 \\
0 & \cdots & 0 & A_1
\end{pmatrix}}_{2^{\ell} \ \text{times}} \in \mathbb{R}^{2^\ell \times 3 \cdot 2^{\ell}}.
\end{align*}
Moreover, still for a natural number $\ell \geq 1$, consider the product
\begin{align}
D_{\ell}:=B_{\ell}C_{\ell+1} \in \mathbb{R}^{3 \cdot 2^{\ell-1} \times 3 \cdot 2^{\ell}}.
\end{align} 
We claim that, for given $k \in \mathbb{N}$, the ReLU neural network $\NN_{\mathcal{A}_k}$ with weight matrices $\mathcal{A}_k:=(B_k,D_{k-1},\dotsc,D_1,A_1)$ and without biases represents the maxima function on $\mathbb{R}^{2^k}$, which is shown by induction over $k$. We already know that the claim holds true for $k=1$, i.e., for $\NN_{\mathcal{A}_1}$ with $\mathcal{A}_1=(A_2,A_1)$.
So assume that $\NN_{\mathcal{A}_k}$ represents the maxima function for input dimension $n=2^k$, and show that this implies the claim for $k+1$. Let $x=(x_1,\dotsc,x_{2^{k+1}})^{\intercal} \in \mathbb{R}^{2^{k+1}}$ be arbitrarily chosen. Then, a simple calculation reveals that
\begin{align*}
\NN_{\mathcal{A}_{k+1}}(x)&=\NN_{\mathcal{A}_k}\left(\left(\max\{x_1,x_2\},\max\{x_3,x_4\},\dotsc,\max\{x_{2^{k+1}-1},x_{2^{k+1}}\}\right)^{\intercal}\right),
\end{align*}
and thus, by the inductive assumption,
\begin{align*}
\NN_{\mathcal{A}_{k+1}}(x)&=\max\left\{\max\{x_1,x_2\},\max\{x_3,x_4\},\dotsc,\max\{x_{2^{k+1}-1},x_{2^{k+1}}\}\right\} \\
&=\max\left\{x_1,x_2,\dotsc,x_{2^{k+1}}\right\},
\end{align*}
which finishes the inductive argument.\\

We note that the neural network $\NN_{\mathcal{A}_k}$ is rather sparse and cone shaped, consists of $k=\log_{2}(n)$ hidden layers, and the $i$th hidden layer has width $3 \cdot 2^{k-i}=\nicefrac{3n}{2^i}$ (see Figure~\ref{fig:NNmax} for a depiction of the network architecture). The results concerning the architecture of the neural network representing the maxima function for general input dimension $n \in \mathbb{N}$ are very similar; we refer the interested reader to~\cite[Prop.~5.4]{GJS:2022}.

\begin{figure} 
\def\layersep{2cm}
\begin{tikzpicture}[shorten >=1pt,->,draw=black!50, node distance=\layersep]
    \tikzstyle{every pin edge}=[<-,shorten <=1pt]
    \tikzstyle{neuron}=[circle,fill=black!25,minimum size=8pt,inner sep=0pt]
    \tikzstyle{input neuron}=[neuron, fill=green!50];
    \tikzstyle{output neuron}=[neuron, fill=red!50];
    \tikzstyle{hidden neuron}=[neuron, fill=blue!50];
    \tikzstyle{annot} = [text width=4em, text centered, black]

    \foreach \name / \y in {1,...,16}
        \node[input neuron] (I-\name) at (0,-0.5*\y cm) {};

    \foreach \name / \y in {1,...,24}
        \path[yshift=2cm]
            node[hidden neuron] (H-\name) at (\layersep,-0.5*\y cm) {};

    \foreach \name / \y in {1,...,12}
        \path[yshift=-1cm]
            node[hidden neuron] (H3-\name) at (2*\layersep,-0.5*\y cm) {};

    \foreach \name / \y in {1,...,6}
        \path[yshift=-2.5cm]
            node[hidden neuron] (H5-\name) at (3*\layersep,-0.5*\y cm) {};

    \foreach \name / \y in {1,...,3}
        \path[yshift=-3.25cm]
            node[hidden neuron] (H7-\name) at (4*\layersep,-0.5*\y cm) {};
                     
    \node[output neuron] (O) at (5*\layersep,-4.25 cm) {};

    \foreach \source in {1,2}
        \foreach \dest in {1,2,3}
            \path (I-\source) edge (H-\dest);
    
        \foreach \source in {3,4}
        \foreach \dest in {4,5,6}
            \path (I-\source) edge (H-\dest);
            
        \foreach \source in {5,6}
        \foreach \dest in {7,8,9}
            \path (I-\source) edge (H-\dest);
            
        \foreach \source in {7,8}
        \foreach \dest in {10,11,12}
            \path (I-\source) edge (H-\dest);

    \foreach \source in {9,10}
        \foreach \dest in {13,14,15}
            \path (I-\source) edge (H-\dest);
    
        \foreach \source in {11,12}
        \foreach \dest in {16,17,18}
            \path (I-\source) edge (H-\dest);
            
        \foreach \source in {13,14}
        \foreach \dest in {19,20,21}
            \path (I-\source) edge (H-\dest);
            
        \foreach \source in {15,16}
        \foreach \dest in {22,23,24}
            \path (I-\source) edge (H-\dest);
            
    \foreach \source in {1,2,3,4,5,6}
        \foreach \dest in {1,2,3}
            \path (H-\source) edge (H3-\dest);
    
    \foreach \source in {7,8,9,10,11,12}
        \foreach \dest in {4,5,6}
            \path (H-\source) edge (H3-\dest);
            
    \foreach \source in {13,14,15,16,17,18}
        \foreach \dest in {7,8,9}
            \path (H-\source) edge (H3-\dest);
    
    \foreach \source in {19,20,21,22,23,24}
        \foreach \dest in {10,11,12}
            \path (H-\source) edge (H3-\dest);

    \foreach \source in {1,2,3,4,5,6}
        \foreach \dest in {1,2,3}
            \path (H3-\source) edge (H5-\dest);
    
    \foreach \source in {7,8,9,10,11,12}
        \foreach \dest in {4,5,6}
            \path (H3-\source) edge (H5-\dest);
            
    \foreach \source in {1,2,3,4,5,6}
        \foreach \dest in {1,2,3}
            \path (H5-\source) edge (H7-\dest);

    \foreach \source in {1,...,3}
        \path (H7-\source) edge (O);

    \node[annot,above of=H-1, node distance=1cm] (hl) {1st hidden layer};
    \node[annot,left of=hl] {Input layer};
    \node[annot,right of=hl] (h2) {2nd hidden layer};
    \node[annot,right of=h2] (h3) {3rd hidden layer};
    \node[annot,right of=h3] (h4) {4th hidden layer};
    \node[annot,right of=h4] {Output layer};
\end{tikzpicture}
\caption{The architecture $\mathcal{A}_4$, i.e., of the neural network representing the maxima function in $\mathbb{R}^{16}$.}
\label{fig:NNmax}
\end{figure}

\subsubsection{Experiments} \label{sec:dnnexp}
We now run some numerical tests to examine the approximation ability of the function $\FWT$ by our deep neural network architecture introduced above. However, instead of simply computing the weight matrix and bias vector of the first hidden layer, cf.~\eqref{eq:WMBV}, with the help of the Python library POT, as (in principle) in the experiments from subsection~\ref{sec:baseline}, the emphasis in the following tests lies on the training of the  deep neural network and the resulting error decay over the epochs. As our loss function we consider the mean absolute error
\begin{align} \label{eq:loss}
\frac{1}{|\mathcal{B}|} \sum_{\mu \in \mathcal{B}} \left|\FWT(\mu)-\mathcal{NN}(\mu)\right|,
\end{align}
where $\mathcal{NN}(\cdot)$ is the realization of the neural network and 
$\mathcal{B} \subset \ttrain$, but still depict the mean relative error
\begin{align} \label{eq:relativeerror}
\frac{1}{|\mathcal{T}|} \sum_{\mu \in \mathcal{T}} \frac{\left|\FWT(\mu)-\mathcal{NN}(\mu)\right|}{\FWT(\mu)},
\end{align}
for $\mathcal{T} \in \{\ttrain,\ttest\}$, in our figures. For the construction of the neural network we employ the Python open source library PyTorch. Moreover, the network is trained with the ADAM optimizer, see~\cite{Kingma2014AdamAM} for its original introduction, in combination with mini batches. 

\experiment \label{exp:fixedMax}
In our first experiment, we use a random initialization of the weight matrix and bias vector of the first hidden layer, cf.~\eqref{eq:linmapfirsthidden} and~\eqref{eq:WMBV}; the remaining part of the neural network architecture is constructed in such a way that it represents the maxima function, see Section~\ref{sec:maxima}, with fixed weights and biases. Here, we consider $2^{12}$ neurons in the first hidden layer for both the MNIST and CIFAR-10 datasets, resulting in $3 \, 211 \, 264$ and $12 \, 587 \, 008$ trainable parameters, respectively. As we can see in Figure~\ref{fig:fixedMax} (left), for the MNIST dataset we have a step decay of the mean relative error for both the training set as well as the test set at the beginning of the training. Yet, after this short initial phase of substantial drop in the loss, we can only observe a slight improvement thereafter. However, at this point, the error is already well below $0.1594\dotsc$, which is the mean relative error obtained in Experiment~\ref{exp:testmeantot} when all the potentials corresponding to the training set $\ttrain$ are considered. This is in contrast to the corresponding experiment for the CIFAR-10 dataset, see Figure~\ref{fig:fixedMax} (right). Here, the error only improves little over the three hundred training epochs, and even the final mean relative error is remarkably larger than the error in Experiment~\ref{exp:testmeantot} for $50 \, 000$ pre-computed potentials arising from the training set, cf.~Figure~\ref{fig:testmeantot} (right). 

\begin{figure}[h]
	\includegraphics[width=0.49\textwidth]{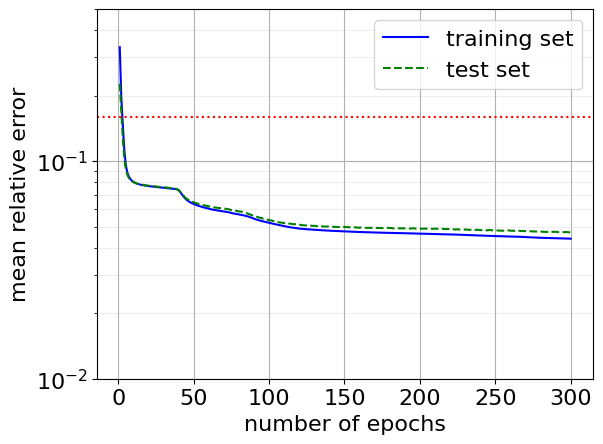} \hfill
 \includegraphics[width=0.49\textwidth]{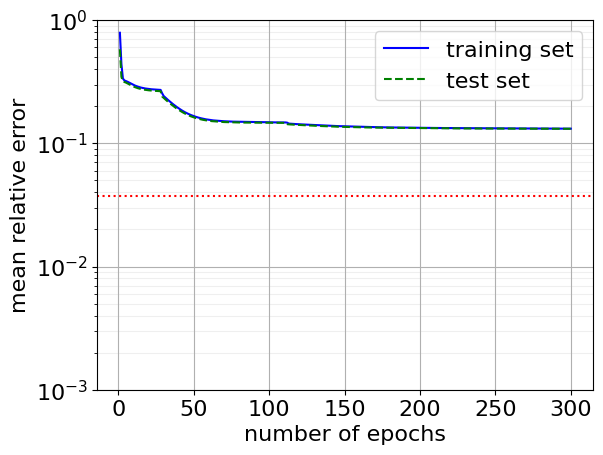}
	\caption{Experiment~\ref{exp:fixedMax}: Approximation of the function $\FWT$ by a deep neural network with $2^{12}$ hidden neurons in the first hidden layer, corresponding to the trainable part of the deep neural network. The red dotted line indicates the final error when all the potentials from the training set were considered in Experiment~\ref{exp:testmeantot}; i.e., this is the final error from our baseline, see Figure~\ref{fig:testmeantot}. Left: MNIST. Right: CIFAR-10.}
\label{fig:fixedMax}
\end{figure}

\experiment \label{exp:fixedprecomputed}
As can be observed in the previous experiment, the deep neural network approach performs worse than the one based on pre-computed potentials in the case of the CIFAR-10 dataset. In this experiment, we shall investigate if a combination of those two ideas may further enhance the approximation accuracy. In particular, we initialize the weight matrix and bias vector, corresponding to the mapping from the input layer to the first hidden layer, with $2^{12}$ pre-computed potentials; cf.~\eqref{eq:linmapfirsthidden}, \eqref{eq:WMBV}. The remaining part of the architecture represents the maxima function, and, as before, has fixed weights and biases. Subsequently, we train the deep neural network over one hundred epochs, using ADAM combined with mini batches. As we can observe from Figure~\ref{fig:fixedPrecomputed}, the mean error for both the training set as well as the test set drastically drops at the initial phase of the training. Thereafter, for both the MNIST and CIFAR-10 datasets, the training loss still steadily decays, however, with a reduced rate. Albeit, for CIFAR-10, the test loss stagnates after the initial phase, see Figure~\ref{fig:fixedPrecomputed} (right). 

\begin{figure}[h]
\includegraphics[width=0.49\textwidth]{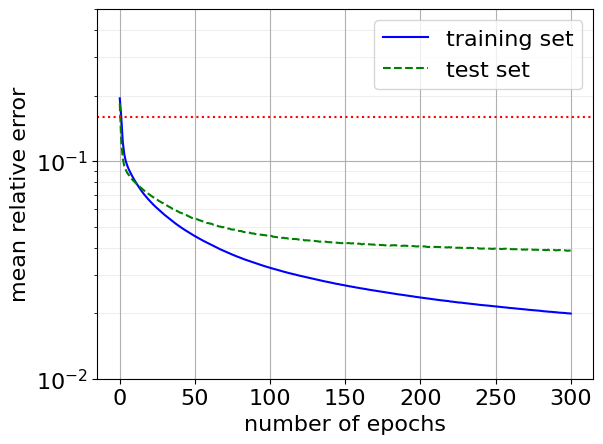} 
\hfill
\includegraphics[width=0.49\textwidth]
{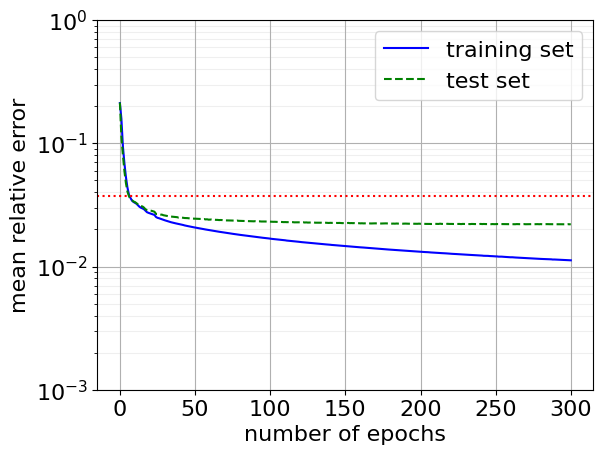}
\caption{Experiment~\ref{exp:fixedprecomputed}. Approximation of the function $\FWT$ by a deep neural network with $2^{12}$ hidden neurons in the first hidden layer, with the weight matrix and bias vector corresponding to~\eqref{eq:linmapfirsthidden} and \eqref{eq:WMBV} initialized by pre-computed Kantorovich potentials. The red dotted line indicates the final error when all the potentials from the training set were considered in Experiment~\ref{exp:testmeantot}; we recall once more that this is our baseline\nc. Left: MNIST. Right: CIFAR-10.} 
	\label{fig:fixedPrecomputed}
\end{figure}

\subsection{Approximation by fully trainable neural networks}
\nc

\experiment \label{exp:trainable}
We use the same neural network design as in the previous experiments; i.e., we consider a weight matrix and bias vector corresponding to~\eqref{eq:linmapfirsthidden}, \eqref{eq:WMBV} between the input and first hidden layers (without ReLU activation function), and the remainder of the architecture represents the maxima function as outlined in Section~\ref{sec:maxima}. However, in contrast to the experiments before, we allow all the neural network parameters to be trainable, including the ones from the part corresponding to the representation of the maxima function. Moreover, in one run the entries of the weight matrix and bias vector corresponding to~\eqref{eq:WMBV} are randomly initialized according to the default setting of the PyTorch package (see Figure~\ref{fig:trainable} (left)), and the other time we use pre-computed Kantorovich potentials as in Experiment~\ref{exp:fixedprecomputed}; here, we only consider the CIFAR-10 dataset, and the first hidden layer consists of $2^{10}$ neurons. Figure~\ref{fig:trainable} indicates that this might further enhance the accuracy compared to the setting of Experiment~\ref{exp:fixedprecomputed}, mostly for the generalization; we emphasize that in the present numerical test the number of trainable parameters is $6 \, 294 \, 010$ compared to $12 \, 587 \, 008$ in the previous experiment. Surprisingly, the performance is even slightly superior for the random initialization of the weight matrix and the bias vector of the first hidden layer compared to the initialization with pre-computed potentials. \

\begin{figure}[h]
	\includegraphics[width=0.49\textwidth]{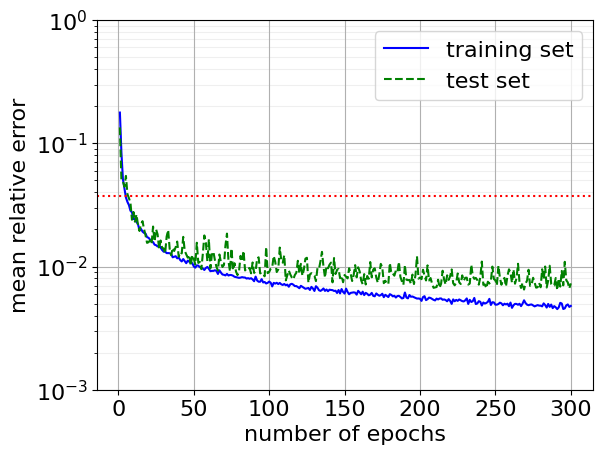} \hfill
 \includegraphics[width=0.49\textwidth]{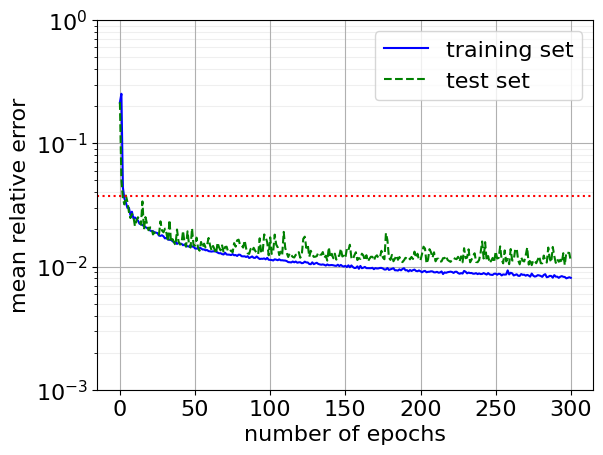}
	\caption{Experiment~\ref{exp:trainable}: Approximation of the function $\FWT$ by a fully trainable deep neural network with $2^{10}$ hidden neurons in the first hidden layer, composed with the neural network that initially, i.e., before the training, represents the maxima function. The red dotted line indicates the final error when all the potentials from the training set were considered in Experiment~\ref{exp:testmeantot}. Left: Random initialization of the weight matrix and bias vector of the first hidden layer. Right: Initialization based on pre-computed potentials.}
	\label{fig:trainable}
\end{figure}

\experiment \label{exp:regularised}

We reconsider the fully trainable neural network from Experiment~\ref{exp:trainable} with random initialization of the matrix and bias vector of the first hidden layer. However, we regularise the loss function~\eqref{eq:loss} with the discretised $H^{1,2}$-norm from~\eqref{eq:discreteH12}; i.e., the loss function takes the form
\begin{align} \label{eq:regularisedloss}
\frac{1}{|\mathcal{B}|} \left(\sum_{\mu \in \mathcal{B}} \left|\FWT(\mu)-\mathcal{NN}(\mu)\right|^{2} + \lambda \ \bigg(|\mathcal{NN}(\mu)|^2+\int_{\mathbb{R}^d} |\mathrm{D} \mathcal{NN} (x,\mu)|^2 \, \de\mu(x) \bigg)\right). 
\end{align}
We recall that the neural network resembles a cylinder function so that $\mathrm{D} \mathcal{NN}$ is well-defined, cf.~\eqref{eq:thediff}, and can be approximated by means of a finite difference scheme for the space derivatives. In Figure~\ref{fig:regularised} we plot the evolution of the mean relative error~\eqref{eq:relativeerror}, for $\ttrain$ and $\ttest$, respectively, over the training epochs. For both choices of the regularization parameter, $\lambda=0.1$ (left) and $\lambda=0.001$ (right), respectively, the mean relative error behaves similarly as in the case of the unregularized loss function, cf.~Figure~\ref{fig:trainable} (left). 

\begin{figure}[h]
\includegraphics[width=0.49\textwidth]{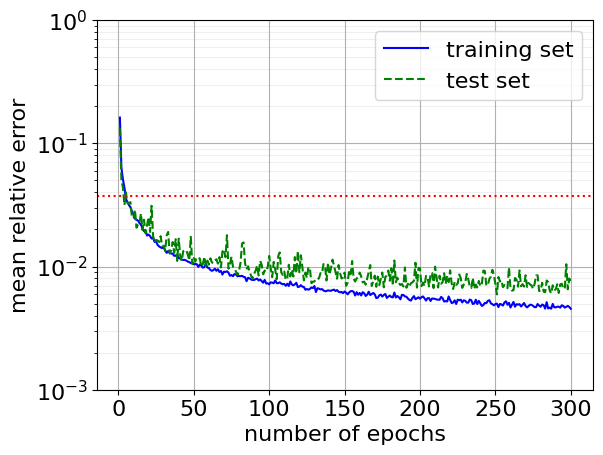} \hfill
 \includegraphics[width=0.49\textwidth]{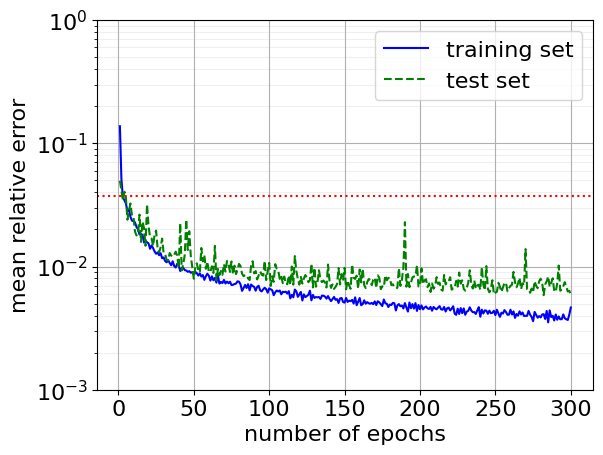}
\caption{Experiment~\ref{exp:regularised}: Repetition of the Experiment~\ref{exp:trainable}, however with an additional regularization term in the loss function, see~\eqref{eq:regularisedloss}. Left: $\lambda=0.1$. Right: $\lambda=0.001$.}
	\label{fig:regularised}
\end{figure}

\experiment \label{exp:CNN}

As is well known in the literature, feed forward neural networks as considered in our previous experiments are not best suited for the CIFAR-10 dataset. A commonly employed better choice are convolutional neural networks (CNNs), for which, however, we do not have any analytical results. Nonetheless, we shall compare the performance of a CNN with a similar number of trainable 
parameters ($6 \, 896 \, 321$) against the one from Experiment~\ref{exp:trainable}. The CNN architecture was borrowed from \href{https://www.kaggle.com}{Kaggle.com}, which is a renowned AI \& ML community, and adapted for our purposes. In particular, the original architecture stems from the reference~\cite{kaggle} and the modified CNN can be found on our GitHub repository \url{https://github.com/heipas/Computing-on-WassersteinSpace}. In Figure~\ref{fig:trainableCNN}, we compare the performances of the CNN and the FFNN from Experiment~\ref{exp:trainable}, which is quite similar. In particular, our FFNN even performs slightly better than the CNN in terms of test error. We obtained similar results for other CNN architectures provided by trustful resources on the internet, e.g., \href{https://pytorch.org/}{PyTorch.org},  as well as designed by ChatGPT through a suitable prompt. In particular, even though CNNs are, in general, better suited for datasets of images such as CIFAR-10, our analytically based FFNN matches and slightly outperforms their performance.

\begin{figure}[h]
	
 \includegraphics[width=0.49\textwidth]{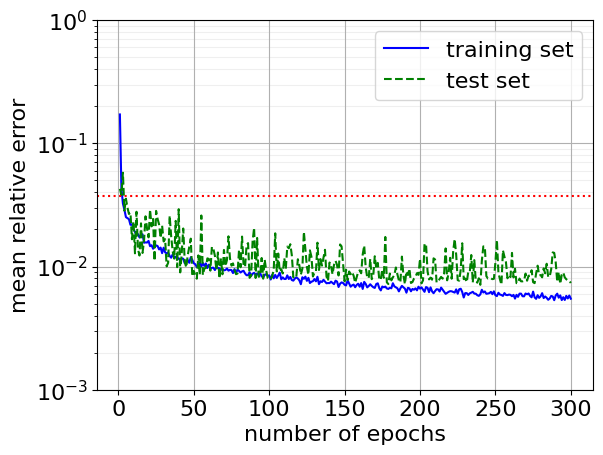}
 \includegraphics[width=0.49\textwidth]{CIFAR_DNN_Trainable_r.png} \hfill
	\caption{Experiment~\ref{exp:CNN}: Comparison of a fully trainable CNN with $6 \, 896 \, 321$ parameters (left) and the FFNN from Experiment~\ref{exp:trainable} with $6 \, 294 \, 010$ parameters (right). The red dotted line indicates once more the error obtained from our baseline, cf.~Experiment~\ref{exp:testmeantot}.} 
	\label{fig:trainableCNN}
\end{figure} 

\nc

\subsection{Adversarial deep neural network approach for the solution of the saddle-point problem}\label{sec:adversarialnetwork}
In the following, we show how the saddle-point problem \eqref{eq:empiricalsaddlepoint} can be leveraged to approximate the minimizer of~\eqref{eq:energyfunctional} in applications. As in Remark~\ref{rem:empmeasure}, let 
\[
\pp_N=\frac{1}{N}\sum_{i=1}^N \delta_{\mu_i}
\]
be our random empirical measure. As already pointed out in that remark, it is not always easy to find a best approximating finite dimensional subspace $\mathcal{V}_n$ of cylinder functions. However, we know that cylinder functions can be modelled by deep neural networks; cf.~Section~\ref{sec:deepneuralnetwork}. So let $F_{\theta}$ be a cylinder function that is realized by a deep neural network, where $\theta$ denotes the specific network parameters stemming from an admissible set $\Theta$ of trainable parameters; this serves as our solution network. In particular, the infimum in~\eqref{eq:empiricalsaddlepoint} is taken over
\[
\{ F_\theta: \theta \in \Theta \} \subset \Cot.
\]
Similarly we consider an adversarial network $H_\xi$ that realizes a cylinder function, and which depends on the trainable parameters $\xi \in \Xi$, for the supremum in~\eqref{eq:empiricalsaddlepoint}. This gives rise to the saddle point problem 
\begin{align} \label{eq:saddleParam}
 \inf_{\theta \in \Theta} \sup_{\xi \in \Xi} \frac{(F_\theta-\widetilde{F},H_\xi)_{L^2_{\pp_N}} +\lambda  \pCE_{2,\pp_N}  (F_\theta,H_\xi)}{|H_\xi|_{pH^{1,2}_{\pp_N}}}.
\end{align}
\nc

We shall mention that a methodologically equivalent approach was introduced in~\cite{Zangetal:20} for the numerical solution of the weak formulation of high-dimensional partial differential equations, which often arise as the Euler--Lagrange equation of an energy minimization problem. On the one hand, the method developed in this section is somewhat reminiscent of \emph{generative adversarial networks} as introduced in~\cite{GAN:2014}. On the other hand, it is certainly more closely related to \emph{physics-informed neural networks} \cite{RAISSI2019686} than generative models.

\begin{remark}
(1) Even if we consider the same architecture for the solution and adversarial networks, respectively, we may not find a parameter $\theta \in \Theta$ in the admissible set such that
\[
(F_\theta-\widetilde{F},H_\xi)_{L^2_{\pp_N}} +\lambda \pCE_{2,\pp_N} (F_\theta,H_\xi)=0 \qquad \text{for all} \ \xi \in \Xi;
\]
cp.~\eqref{eq:discreteEL}. This is due to the fact that 
\[
\{F_\theta: \theta \in \Theta \} \subset \Cot
\]
is not a linear subspace. \\
(2) A common advantage of the adversarial approach is that it leads, in general, to more robustness; see, e.g.,~\cite{HouzeMeng:2023}.
\end{remark}

We now discuss the specific algorithm, cf.~Algorithm~\ref{alg:Adversarial}, which is very closely related to the one from~\cite{Zangetal:20}, to obtain a suitable solution approximation. For given neural network architectures for $F_\theta$ and $H_\xi$, we initialize the weights $\theta \in \Theta$ and $\xi \in \Xi$. Subsequently, we repeat the following procedure until a stopping criterion, which could be a prescribed number of loops or an error tolerance for the loss function $\LO_{\xi}(\theta)$, cf.~\eqref{eq:LossTheta} below, is satisfied:
\begin{enumerate}[(1)]
\item for a fixed parameter $\theta \in \Theta$, we define the adversarial loss function 
\begin{align} \label{eq:Lossxi}
\LO_\theta(\xi):=-\frac{(F_\theta-\widetilde{F},H_\xi)_{L^2_{\pp_N}} +\lambda \pCE_{2,\pp_N} (F_\theta,H_\xi)}{|H_\xi|_{pH^{1,2}_{\pp_N}}}
\end{align}
with input variable $\xi \in \Xi$. Then we train, for instance by applying gradient descent, the neural network $H_\xi$ over a prescribed number of epochs $N_\Xi \geq 1$ to minimize the loss function $\LO_\theta(\xi)$;
\item for given $\xi \in \Xi$, we define the solution loss function
\begin{align} \label{eq:LossTheta}
\LO_\xi(\theta):=\frac{(F_\theta-\widetilde{F},H_\xi)_{L^2_{\pp_N}} +\lambda \pCE_{2,\pp_N} (F_\theta,H_\xi)}{|H_\xi|_{pH^{1,2}_{\pp_N}}},
\end{align}
which depends on $\theta \in \Theta$. Similar as in (1), we employ a training procedure for the neural network $F_\theta$ to optimize the loss function~\eqref{eq:LossTheta}.
\end{enumerate}

\begin{remark}
We note that both loss functions $\LO_\theta$ and $\LO_\xi$, cf.~\eqref{eq:LossTheta} and~\eqref{eq:Lossxi}, respectively, depend on a set of given data $\{(\mu_j,y_j)\}_{j=1}^{N}$, which is our training set, through the empirical measure $\pp_N$ and the corresponding noisy function values. In practice, we rather use mini batches thereof for the training, and apply optimization procedures that are superior to gradient descent, such as stochastic gradient descent or ADAM (see~\cite{Kingma2014AdamAM}).  
\end{remark}

\begin{algorithm}[t]
\caption{\small Adversarial deep neural network algorithm for the solution of the Euler--Lagrange equation}
\label{alg:Adversarial}
\begin{flushleft} 
\textbf{Input:} Training data $\{(\mu_j,y_j)\}_{j=1}^{N}$ and number of epochs $N_\Xi,N_\Theta \geq 1$.\\
\textbf{Initialize:} Neural networks $F_\theta$ and $H_\xi$.
\end{flushleft}
\begin{algorithmic}[1]
\Repeat
\For {$k=1, \dotsc, N_\Xi$}
\State One training step for $H_\xi$ to minimize the loss function $\LO_{\theta}(\xi)$ from~\eqref{eq:Lossxi}.
\EndFor
\For {$k=1, \dotsc, N_\Theta$}
\State One training step for $F_\theta$ to minimize the loss function $\LO_{\xi}(\theta)$ from~\eqref{eq:LossTheta}.
\EndFor
\Until stopping criterion is satisfied
\end{algorithmic}
\begin{flushleft}
\textbf{Output:} Trained solution network $F_\theta$.
\end{flushleft}
\end{algorithm}

\subsubsection{Experiments} We now test our Algorithm~\ref{alg:Adversarial} in the context of the MNIST dataset. We once more use the network architecture from Experiment~\ref{exp:trainable}, adapted to the smaller input dimension of the MNIST dataset compared to CIFAR-10, for both the solution network as well as the adversarial network. 

\experiment \label{exp:ELH1} We compare the performance of Algorithm~\ref{alg:Adversarial}, with the regularization parameter set to $\lambda=0.001$, for two distinct choices of the pairs of parameters $(N_\Xi,N_\Theta)$. In Figure~\ref{fig:ELH1} (left) we set $N_\Xi=1$ and $N_\Theta=2$, whereas in Figure~\ref{fig:ELH1} (right) we choose $N_\Xi=2$ and $N_\Theta=1$. We can observe in both figures that the mean relative error of both the training as well as the test set do nicely decay, with a comparable overall performance as in Experiment~\ref{exp:fixedprecomputed}. Even though we obtain, in general, a better approximation accuracy with the approach outlined in Section~\ref{sec:deepneuralnetwork}, it is noteworthy that the mean relative error of the training and test sets almost coincide in this experiment, in contrast to the ones in Section~\ref{sec:deepneuralnetwork}, meaning that the generalization error almost vanishes.

\begin{figure}[h]
\includegraphics[width=0.49\textwidth]{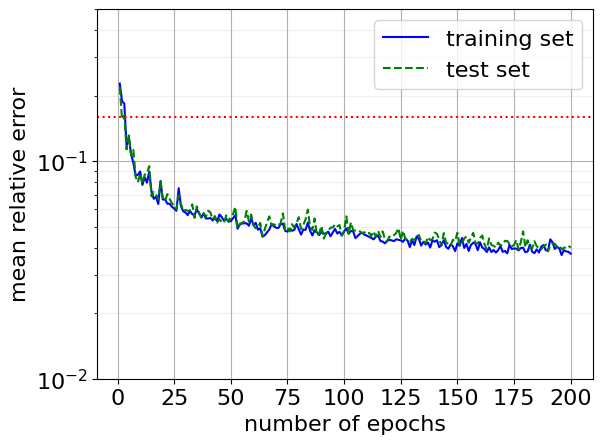} \hfill
 \includegraphics[width=0.49\textwidth]{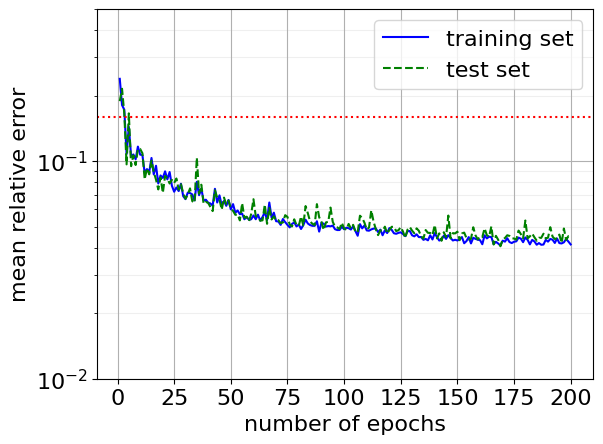}
\caption{Experiment~\ref{exp:ELH1}: Decay of the mean relative error for the MNIST dataset based on Algorithm~\ref{alg:Adversarial}. Left: $N_\Xi=1$ and $N_\Theta=2$. Right: $N_\Xi=2$ and $N_\Theta=1$.}
	\label{fig:ELH1}
\end{figure}

\experiment \label{exp:ELL2} We again consider the Euler--Lagrange approach from Section~\ref{sec:eulerlagrange}, however we work on $\Lot$ rather than $\Hot$. In turn, the discrete saddle point problem, cf.~\eqref{eq:saddleParam}, reads as
\[
 \inf_{\theta \in \Theta} \sup_{\xi \in \Xi} \frac{(F_\theta-\widetilde{F},H_\xi)_{L^2_{\pp_N}}}{|H_\xi|_{L^2_{\pp_N}}},
\]
where 
\[
|H|^2_{L^2_{\pp_N}}:=(H,H)_{L^2_{\pp_N}}=\frac{1}{N} \sum_{j=1}^N |H(\mu_j)|^2, \qquad H \in \Lot.
\]
Except for that, we use the same framework as in Experiment~\ref{exp:ELH1}. In Figure~\ref{fig:ELL2} we can once more observe the decay of the mean relative error, with a similar overall behaviour of the error over the epochs as before. Albeit, the error seems to fluctuate more than in Experiment~\ref{exp:ELH1}; i.e., considering the $\Hot$ framework instead of working on $\Lot$ might indeed lead to a smoother approximation of the target function.

\begin{figure}[h]
\includegraphics[width=0.49\textwidth]{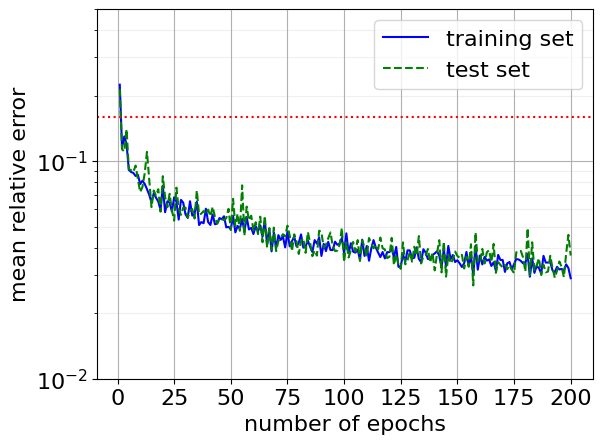} \hfill
 \includegraphics[width=0.49\textwidth]{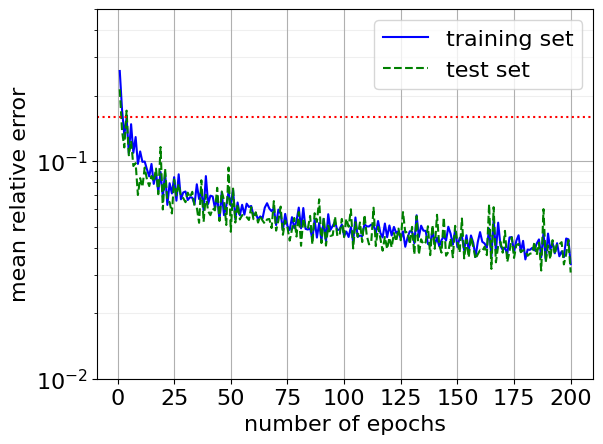}
\caption{Experiment~\ref{exp:ELL2}: Decay of the mean relative error for the MNIST dataset based on Algorithm~\ref{alg:Adversarial}. Left: $N_\Xi=1$ and $N_\Theta=2$. Right: $N_\Xi=2$ and $N_\Theta=1$.}
	\label{fig:ELL2}
\end{figure}

\appendix 
\section{Extended review of Wasserstein distance computation} \label{sec:app1}
\nc

In this appendix, we provide an overview of computational schemes to compute the Wasserstein distance between two given measures, thereby distinguishing three different cases: (a) Both measures are continuous, (b) both measures are discrete, i.e., a sum of finitely many weighted Dirac measures, and (c) one is continuous and the other discrete, separately. Regrettably, in our overview we may have inadvertently omitted some recent works in this area; we  apologize to those whose contributions may not have been included in this discussion.\\

We start our overview by presenting some of the existing procedures for the numerical solution of the Wasserstein distance in the context of two discrete measures; this is certainly the best studied and most elaborated among the three cases (a)--(c) outlined above. First of all, we note that in the discrete setting the Kantorovich problem is a classical, finite dimensional, linear program; i.e., the problem amounts to the optimization of an objective function that is linear and whose constraints are linear as well. Indeed, the Kantorovich problem is a minimum cost network flow problem. Consequently, in the given setting, the Wasserstein distance can be computed by common algorithmic tools from linear programming. A very well written overview of applicable algorithms is  given in the book of Peyr\'{e} and Cuturi, \cite{PeyreCuturi:2019}, or the thesis~\cite{Schieber:2019}. To mention only a few of them, we point to the Hungarian method introduced by Kuhn, \cite{Kuhn:1955}, or the Auction method, which was originally proposed by Bertsekas, \cite{Bertsekas:1981}, further improved in~\cite{BE:1988}, and applied to the transportation problem in~\cite{BC:1989}; we further refer to~\cite{BPPH:2011}. Even though those algorithms allow to compute an approximation with an arbitrary accuracy, a common disadvantage of those algorithms stemming from the linear programming is that they are computationally very expensive; see, e.g.,~\cite{GoldbergTarjan:1989, Tarjan:1997, PeleWerman:2009}. In particular, they have cubical complexity (with respect to the number of atoms of the discrete measures). The computational cost can be drastically improved by solving a flow problem on a suitable graph, see~\cite{LingOkada:2007} and the more recent work~\cite{BGV:2020}, which, however, only applies to the 1-Wasserstein distance, but not to a general $p$-Wasserstein distance. 

A prominent approach to relieve the computational cost is to add an entropic regularization penalty term to the optimal transport problem. We refer to the book of Peyr\'{e} and Cuturi, \cite{PeyreCuturi:2019}, for a gentle treatment of this topic. We stress that the solution of the regularized problem converges to the optimal transport plan as the regularization parameter goes to zero. The most famous approach to solve the regularized problem is given by the Sinkhorn algorithm, which applies a simple iteration procedure, see, e.g.,~\cite[Sec.~4.2]{PeyreCuturi:2019}. We note that these iterations are based on matrix-vector products, and thus are suitable for GPU computations. Furthermore, this scheme was formerly known as the iterative proportional fitting procedure (IPFP), see, e.g.,~\cite{DemingStephan:1940}. However, as the convergence proof is credited to R.~Sinkhorn, \cite{Sinkhorn:1964}, the iteration method is nowadays known as the Sinkhorn algorithm. It was shown in~\cite{FranklinLorenz:1989} that the convergence of the Sinkhorn algorithm is indeed linear. The Sinkhorn algorithm has regained more attention thanks to the work~\cite{Cuturi:2013}, which highlighted that this scheme has several favourable computational properties. Since then, the algorithm has witnessed many modifications and extensions. For instance, the numerical instability with respect to a small regularization parameter was alleviated by Schmitzer, ~\cite{Schmitzer:2019}. In~\cite{BCCNP:2015}, the authors further improved Sinkhorn's algorithm by exploiting that the regularized problem corresponds to the Kullback--Leibler Bregman divergence projection, which allows to apply a Bregman--Dykstra iteration scheme. We further remark that for the approximation of the solution of the \emph{unregularized} problem one may apply a so-called proximal point algorithm for the Kullback--Leibler divergence, cf.~\cite[Rem.~4.9]{PeyreCuturi:2019}. Finally, for a generalized Sinkhorn algorithm that can be applied to larger class of convex optimization problems we refer to~\cite[Sec.~4.6]{PeyreCuturi:2019}.

A more recent approach to compute the 2-Wasserstein distance is given by the linear optimal transport (LOT) framework, which was originally introduced in~\cite{WSBOR:2013} in the context of image datasets; see also the follow-up works~\cite{BKR:2014,KTOR:2016}. This approach employs a meaningful linearization of the Wasserstein metric measure space by a projection onto the tangent space at a given reference measure. Then, to compute the LOT distance between two measures, which in particular is an Euclidean distance on the tangent space, one first needs to compute the optimal transport plan of each of those measures to the fixed reference measure. We also refer to the closely related work~\cite{SeguyCuturi:2015}. Morever, in~\cite{CFD:2018}, an Euclidean embedding, for which the Euclidean distance again approximates the Wasserstein distance, is learnt with a neural network. For further applications of Wasserstein embeddings we refer, for instance, to~\cite{KNRH:2021,KST:2020}; certainly, there are (many) more works that apply the Wasserstein embedding methodology. 

At last in the context of the computation of the Wasserstein distance between two discrete measures, we shall point to the Wasserstein  generative adversarial networks (WGANs) as introduced in~\cite{WGAN:2017}. In that framework, the Wasserstein distance is estimated by computing the Kantorovich dual for a discriminator network. More precisely, the set of admissible potentials in the dual formulation is restricted to the realization of a given neural network, which is the discriminator in the setting of WGAN. Subsequently, the neural network is trained to maximize the expected value in the dual formulation; i.e., the discriminator is trained to approximate the Kantorovich potential. This approach has indeed been very successful. We further refer to the improved WGAN~\cite{WGAN2:2017}, and to \cite{HXD:2018}, where two-step method for the computation of the Wasserstein distance in WGANs is introduced. In~\cite{WGAN3:2019} a comparison of different schemes studied in the WGAN literature is presented, and \cite{WGAN4:2021} provides theoretical insights of WGANs. \\

Next, we address the semi-discrete setting, which, however, is only discussed superficially; we point the interested reader to~\cite{PeyreCuturi:2019, santambrogio} and the references given therein for more details. A common approach to solve the semi-discrete optimization problem numerically is to consider the dual problem and split the integral over the underlying domain into so-called Laguerre cells. In the special case $p=2$, i.e., the 2-Wasserstein distance, the Laguerre cells are known as \emph{power cells}. In that case, the cells are polyhedral, and can be computed efficiently using computational geometry algorithms, see, e.g.,~\cite{Aurenhammer:1987}. Then, in order to solve the optimization problem, one may apply, combined with a scheme from computational geometry, gradient algorithms, which include Netwon methods; we refer, for instance, to the works~\cite{Merigot:2011, Levy:2015, KMT:2016, LevySchwindt:2018, KMT:2019, BK:2021}. \\

Lastly, we consider the \emph{continuous} case, specifically the computation of the Wasserstein distance between (probability) distributions. For that purpose, we mainly summarize the schemes presented in \cite[Sec.~6]{santambrogio} in a condensed manner, and recall some of the reference stated therein. The most prominent method in the continuous case is based on the Benamou--Brenier formulation, in which the optimal transport problem is transformed into a convex optimization problem with linear constraints; see, e.g., the monographs~\cite{santambrogio,ABS:2021}, or the original report~\cite{BenamouBrenier:2000}. In order to solve this convex optimization problem numerically, the authors of~\cite{BenamouBrenier:2000} employed an augmented Lagrangian scheme combined with Uzawa's gradient method. It can be verified that this procedure converges. For instance, this was shown in the works of Benamou, Brenier, and Guittet, \cite{BBG:2004}, and Guittet, \cite{Guittet:2003}; we refer to~\cite[Sec.~6.1]{santambrogio} and the references therein for further details and other application areas of this approach. Further numerical schemes in the context of continuous measures include the algorithm due to Angenent, Haker, and Tannenbaum, cf.~\cite{AHT:2003} and~\cite{HZTA:2004}, as well as computational procedures for the Monge-Ampère problem. This equation was solved in~\cite{LP:2005} by an iterative nonlinear solver based on Newton's method, whereas the authors of~\cite{CWVB:2009} employed a gradient descent solution to the Monge-Ampère problem. We shall further mention the manuscript~\cite{CRLVP:2020}, in which the authors proved a sample complexity bound and studied the performance of the Sinkhorn divergence estimator, which was introduced in~\cite{RGC:2017}.

Finally, we want to point to the stochastic optimization approach for large-scale optimal transport problems as introduced in~\cite{GCPB:2016}, which can be applied to the continuous, discrete, as well as the semi-discrete cases. Especially in the continuous setting, the authors of~\cite{GCPB:2016} proposed a novel method that exploits reproducing kernel Hilbert spaces, and proved the convergence under suitable assumptions.

\section{Wasserstein distance approximation for a discrete base set}\label{sec:4}
\nc 

We consider a Wasserstein Sobolev space over a \emph{finite} metric space $\mathfrak{D}$.
 This allows for some simplifications of the results from Section~\ref{sec:conofpot}, and, in turn, for easier implementation in certain applications; here, we have foremost the grayscaled and coloured images from the MNIST and CIFAR-10 datasets in mind. As usual, we signify the set of all probability measures on $\ds$ by 
\begin{align*}
\mathcal{P}(\mathfrak{D}):=\left\{\mu=\sum_{\xx \in \ds} i_{\xx} \delta_\xx: \sum_{\xx \in \ds} i_{\xx}=1 \ \text{and} \ i_{\xx} \in [0,1] \ \forall \xx \in \ds \right\},
\end{align*}
where $\delta_{\xx}$ denotes the Dirac measure at $\xx \in \mathfrak{D}$.  

Note that any ordering on the set $\ds$ defines, in the obvious way, a bijection 
\begin{align} \label{eq:Jop}
J:\ds \to \mathbb{S}_{\mathbb{R}^{d_\ds}}:=\left\{\xx \in \mathbb{R}_{ \geq 0 }^{d_{\ds}}: \norm{x}_1=1\right\},
\end{align}
where $d_\ds:=|\ds|$ denotes the cardinality of $\ds$ and $\mathbb{R}_{ \geq 0 }:=[0,\infty)$. We show that $J$ is indeed a homeormophism.

\begin{lemma} \label{lem:homeo}
Given a sequence $(\mu_n)_n \subset \mathcal{P}(\ds)$ and an element $\mu \in \mathcal{P}(\ds)$, then the following two statements are equivalent:
\begin{itemize}
    \item[(i)] $W_p(\mu_n,\mu) \to 0$ as $n \to \infty$;
    \item[(ii)] $\left|J(\mu_n)-J(\mu)\right| \to 0$ as $n \to \infty$, where $| \cdot |$ denotes any $p$-norm on $\mathbb{R}^{d_{\ds}}$; for instance, we may consider the Euclidean distance $|\cdot|=\norm{\cdot}_{2}$. 
\end{itemize} 
In particular, the operator $J$ from~\eqref{eq:Jop} is a homeomorphism.
\end{lemma}

\begin{proof}
Our proof relies on the fact that $W_p(\mu_n,\mu)$ vanishes as $n \to \infty$ if and only if $\mu_n$ converges weakly to $\mu$, which means that
\begin{align*}
\int_{\ds} \varphi \, \mathrm{d} \mu_n \to \int_{\ds} \varphi \, \mathrm{d} \mu \qquad \text{for all} \ \varphi \in \rmC (\ds),
\end{align*}
where $\rm \rmC  (\ds)$ denotes the set of (continuous) functions on $\ds$; this result can be found, for instance, in~\cite[Thm.~6.9]{Villani:09}. First, let us assume that $\mu_n=\sum_{\xx \in \ds} i_{n,\xx} \delta_{\xx}$ converges to $\mu=\sum_{\xx \in \ds} i_{\xx} \delta_{\xx}$ with respect to the topology induced by the Wasserstein distance $W_p$, or, equivalently, with respect to the weak topology. For any $\xx^\star \in \ds$ consider the function $\phi_{\xx^\star}:\ds \to \mathbb{R}$ given by
\[
\phi_{\xx^\star}(\xx)=\begin{cases}
0 & \xx \neq \xx^\star,\\
1 & \xx = \xx^\star.
\end{cases}
\]
Then, thanks to the weak convergence, we have that
\[
i_{n,\xx^\star} = \int_\ds \phi_{\xx^\star} \, \mathrm{d} \mu_n \to \int_{\ds} \phi_{\xx^\star} \, \mathrm{d} \mu=i_{\xx^\star} \qquad \text{as} \ n \to \infty.
\]
Since this holds for any $\xx^\star \in \ds$, we indeed have that $|J(\mu_n)-J(\mu)| \to 0$ as $n \to \infty$. 

Next we assume that (ii) holds true and want to show that this implies (i). So take any $\phi \in \rm \rmC  (\ds)$, and note that $\phi$ is uniformly bounded since $\ds$ is finite. Consequently, we find that
\[
\left|\int_{\ds} \phi \, \mathrm{d} \mu_n-\int_{\ds} \phi \, \mathrm{d} \mu\right| \leq \norm{\phi}_{\infty} \sum_{\xx \in \ds} \left|i_{n,\xx}-i_{\xx}\right| \leq \norm{\phi}_{\infty} \sqrt{d_{\ds}} \norm{J(\mu_n)-J(\mu)}_2,
\]
where we employed the Cauchy--Schwarz inequality in the last step. Since all $p$-norms on $\mathbb{R}^{d_{\ds}}$ are equivalent, the right-hand side above vanishes as $n \to \infty$. This proves the weak convergence, and, in turn, the convergence with respect to the $W_p$-metric.
\end{proof}

The Heine Borel theorem implies that $\mathbb{S}_{\mathbb{R}^{d_\ds}} \subseteq \mathbb{R}^{d_{\ds}}$ is compact, and thus, since $J:\mathcal{P}(\ds) \to \mathbb{S}_{\mathbb{R}^{d_\ds}}$ is a homeormorphism by Lemma~\ref{lem:homeo}, the same holds true for $\mathcal{P}(\ds)$.

\begin{proposition} \label{prop:compact}
The space $\mathcal{P}(\ds)$ endowed with the topology induced by the Wasserstein distance $W_p$ is compact.
\end{proposition}

\begin{remark} As recalled in Section \ref{sec:wasserstein},  the conclusion of Proposition~\ref{prop:compact} is a well-known result and holds true for $\mathcal{P}(K)$, where $K$ is any compact subset of a metric space. This result follows, for instance, from Prokhorov's theorem. For the sake of completeness, we still wanted to include the proof for the discrete case.
\end{remark}

In the following, let $\vartheta \in \mathcal{P}(\ds)$ be a reference measure and $\{\mu_{k}\}_{k=1}^\infty \subset \mathcal{P}(\ds)$ a countable, dense subset of $\mathcal{P}(\ds)$, which exists since $\mathcal{P}(\ds)$ is a Polish space. We further denote by $(\varphi_k,\psi_k)$ the pair of the Kantorovich potentials for $\vartheta$ and $\mu_k$; i.e., we have that, for any $k \in \mathbb{N}$,
\begin{align*}
W_p^p(\vartheta,\mu_k)=\int_{\ds} \varphi_k(\xx) \, \mathrm{d} \mu_k(\xx) + \int_{\ds} \psi_k(\xx) \, \mathrm{d} \vartheta(\xx).
\end{align*}
It further holds that $\psi_k$ is the $c$-transform, $\varphi_k^{c}$, of $\varphi_k$ and vice versa, where
\[
\varphi_k^{c}(\xx):=\inf_{\yy \in \ds} d(\xx,\yy)^p-\varphi_k(\yy), \qquad \xx \in \ds,
\]
with $d:\ds \times \ds \to [0,\infty)$ signifying the metric on $\ds$; we refer to~\cite[Thm.~1.39]{santambrogio}. In what follows, for any finite subset $I \subset \mathbb{N}$, let $\FW,\GW{I}:\mathcal{P}(\ds) \to \mathbb{R}$ be defined as
\begin{align} \label{eq:FGDef}
\FW(\mu):=W_p^p(\vartheta,\mu) \qquad \text{and} \qquad \GW{I}(\mu):=\max_{k \in I} \int_\ds \varphi_k \, \mathrm{d} \mu + \int_\ds \psi_k \, \mathrm{d} \vartheta, 
\end{align}
respectively. The goal is to show that $\GW{I}$ approximates in a certain sense the Wasserstein distance (function) $\FW$. 

We first note that $\FW:\mathcal{P}(\ds) \to [0,\infty)$ is continuous: indeed, if $\nu \to \mu$ in $W_p$, then 
\[
\lim_{\nu \to \mu} \FW(\nu) =\lim_{\nu \to \mu} W_p^p(\nu,\vartheta) =W_p^p(\mu,\vartheta)=\FW(\mu).
\] Together with the compactness of the set $\mathcal{P}(\ds)$, cf.~Proposition~\ref{prop:compact}, this immediately implies the following auxiliary result.

\begin{lemma} \label{lem:uniformlycontinuous}
The function $\FW:\mathcal{P}(\ds) \to [0,\infty)$ is uniformly continuous.
\end{lemma}

By following along the lines of the proof of~\cite[Thm.~3.1]{S22} one may find that, for any $k \in \mathbb{N}$ and $\yy',\yy'' \in \ds$,
\begin{align*}
|\varphi_k(\yy')-\varphi_k(\yy'')| \leq p \max_{\xx \in \ds} |d(\xx,\yy')-d(\xx,\yy'')| |d(\xx,\yy')^{p-1}-d(\xx,\yy'')^{p-1}| \leq C d(\yy',\yy''),
\end{align*}
where the constant $C$ depends on $p$ and $\ds$, but is independent of $k$, $\yy'$ and $\yy''$; for the last inequality we employed the (reverse) triangle inequality and further used that $\ds$ is finite, which, in turn, implies that $d$ is uniformly bounded on $\ds \times \ds$. Then, by similar arguments as in Remark~\ref{rem:above} and Remark~\ref{rem:below}, we obtain the following bound.

\begin{lemma} \label{lem:Gbound}
There exists a positive constant $C_{p,\ds}$, only depending on $p$, the set $\ds$, and the base metric $d$, such that, for any finite subset $I \subset \mathbb{N}$, we have
\begin{align} \label{eq:GopBound}
\left|\GW{I}(\mu)-\GW{I}(\nu)\right| \leq C_{p,\ds} W_p(\mu,\nu) \qquad \text{for all} \ \mu,\nu \in \mathcal{P}(\ds).
\end{align}
\end{lemma}

We have gathered all the ingredients to prove the main theorem of the present section.

\begin{theorem}\label{thm:eps}
For all $\varepsilon>0$ there exists a finite set $I_{\eps } \subset \mathbb{N}$ such that
\[
\sup_{\mu \in \mathcal{P}(\ds)} \left|\FW(\mu)-\GW{I_{\eps }}(\mu)\right| \leq {\eps },
\]
where $\FW$ and $\GW{I_{\eps }}$ are defined as in~\eqref{eq:FGDef}.
\end{theorem}

\begin{proof}
Since $\FW:\mathcal{P}(\ds) \to [0,\infty)$ is uniformly continuous, cf.~Lemma~\ref{lem:uniformlycontinuous}, for all ${\eps }>0$ there exists a $\delta_{{\eps }}>0$ such that $\left|\FW(\mu)-\FW(\nu)\right|\leq \eps/2 $ whenever $W_p(\mu,\nu)\leq \delta_{{\eps }}$. Set $\delta:=\min\left\{\delta_{{\eps }}, \frac{{\eps }}{2 C_{p,\ds} }\right\}$. Then, since $\mathcal{P}(\ds)$ is compact by Proposition~\ref{prop:compact}, there exists a finite subset $I_{\eps } \subset \mathbb{N}$ such that
\begin{align} \label{eq:covering}
\mathcal{P}(\ds)=\bigcup_{k \in I_{\eps }} \rmB(\mu_k,\delta),
\end{align}
where $\rmB(\mu,\delta)=\left\{\nu \in \mathcal{P}(\ds):W_p(\mu,\nu) < \delta\right\}$ is the open ball of radius $\delta$ centered at $\mu$.  Let $\mu \in \mathcal{P}(\ds)$ be arbitrary, and choose $k \in I_{\eps }$ such that $W_p(\mu,\mu_k)<\delta$, cf.~\eqref{eq:covering}. Then, we have that 
\begin{align*}
\left|\FW(\mu)-\GW{I_{\eps }}(\mu)\right| & \leq \left|\FW(\mu)-\FW(\mu_k)\right|+\left|\FW(\mu_k)-\GW{I_{\eps }}(\mu_k)\right|+\left|\GW{I_{\eps }}(\mu_k)-\GW{I_{\eps }}(\mu)\right| \\
& \leq \frac{{\eps }}{2}+\left|\FW(\mu_k)-\GW{I_{\eps }}(\mu_k)\right|+\frac{{\eps }}{2},
\end{align*}
where the upper bound for the first summand follows from the uniform continuity, and the estimate for the third term is due to~\eqref{eq:GopBound} and the choice of $\delta$. Finally, we note that
\begin{align*}
\FW(\mu_k)=W^p_p(\mu_k,\vartheta)=\int_{\ds} \varphi_k \, \mathrm{d} \mu_k + \int_{\ds} \psi_k \, \mathrm{d} \vartheta=\GW{I_{\eps }}(\mu_k),
\end{align*}
where the latter equality holds since $k \in I_{\eps }$ and the supremum of the dual problem is attained at the Kantorovich duals $(\varphi_k,\psi_k)$. As $\mu \in \mathcal{P}(\ds)$ was chosen arbitrarily, the claim is proved.
\end{proof}

 As an immediate consequence of the above theorem we get the ensuing result.
 
\begin{corollary}
We have that 
\[
\sup_{\mu \in \mathcal{P}(\ds)} \left|\FW(\mu)-\GW{1:j}(\mu)\right| \to 0 \quad \text{as} \ j \to \infty,
\]
where $1:j:=\{1,2,\dotsc,j\}$ for $j \in \mathbb{N}$.
\end{corollary} 

\bibliographystyle{plain}
\bibliography{biblio}

\begin{thebibliography}{100}

\bibitem{ABS:2021}
Luigi Ambrosio, Elia Bru\'{e}, and Daniele Semola.
\newblock {\em Lectures on {O}ptimal {T}ransport}, volume 130 of {\em Unitext}.
\newblock Springer, Cham, 2021.

\bibitem{AES16}
Luigi Ambrosio, Matthias Erbar, and Giuseppe Savar{\'e}.
\newblock Optimal transport, {C}heeger energies and contractivity of dynamic
  transport distances in extended spaces.
\newblock {\em Nonlinear Anal.}, 137:77--134, 2016.

\bibitem{AGS08}
Luigi Ambrosio, Nicola Gigli, and Giuseppe Savar{\'e}.
\newblock {\em Gradient flows in metric spaces and in the space of probability
  measures}.
\newblock Lectures in Mathematics ETH Z\"urich. Birkh\"auser Verlag, Basel,
  second edition, 2008.

\bibitem{AGS14I}
Luigi Ambrosio, Nicola Gigli, and Giuseppe Savar{\'e}.
\newblock Calculus and heat flow in metric measure spaces and applications to
  spaces with {R}icci bounds from below.
\newblock {\em Invent. Math.}, 195(2):289--391, 2014.

\bibitem{AMBROSIO2021108968}
Luigi Ambrosio, Shouhei Honda, Jacobus~W. Portegies, and David Tewodrose.
\newblock Embedding of {RCD}({K},{N}) spaces in ${L}^2$ via eigenfunctions.
\newblock {\em Journal of Functional Analysis}, 280(10):108968, 2021.

\bibitem{AHT:2003}
Sigurd Angenent, Steven Haker, and Allen Tannenbaum.
\newblock Minimizing flows for the {M}onge-{K}antorovich problem.
\newblock {\em SIAM J. Math. Anal.}, 35(1):61--97, 2003.

\bibitem{WGAN:2017}
Martin Arjovsky, Soumith Chintala, and L{\'e}on Bottou.
\newblock {W}asserstein generative adversarial networks.
\newblock In Doina Precup and Yee~Whye Teh, editors, {\em Proceedings of the
  34th International Conference on Machine Learning}, volume~70 of {\em
  Proceedings of Machine Learning Research}, pages 214--223. PMLR, 06--11 Aug
  2017.

\bibitem{Agamma}
H.~Attouch.
\newblock {\em Variational convergence for functions and operators}.
\newblock Applicable Mathematics Series. Pitman (Advanced Publishing Program),
  Boston, MA, 1984.

\bibitem{Aurenhammer:1987}
F.~Aurenhammer.
\newblock Power diagrams: properties, algorithms and applications.
\newblock {\em SIAM J. Comput.}, 16(1):78--96, 1987.

\bibitem{BK:2021}
Mohit Bansil and Jun Kitagawa.
\newblock A {N}ewton {A}lgorithm for {S}emidiscrete {O}ptimal {T}ransport with
  {S}torage {F}ees.
\newblock {\em SIAM Journal on Optimization}, 31(4):2586--2613, 2021.

\bibitem{BGV:2020}
Federico Bassetti, Stefano Gualandi, and Marco Veneroni.
\newblock On the computation of {K}antorovich-{W}asserstein distances between
  two-dimensional histograms by uncapacitated minimum cost flows.
\newblock {\em SIAM J. Optim.}, 30(3):2441--2469, 2020.

\bibitem{BKR:2014}
Saurav Basu, Soheil Kolouri, and Gustavo~K. Rohde.
\newblock Detecting and visualizing cell phenotype differences from microscopy
  images using transport-based morphometry.
\newblock {\em Proceedings of the National Academy of Sciences},
  111(9):3448--3453, 2014.

\bibitem{BJK:2022}
Christian Beck, Arnulf Jentzen, and Benno Kuckuck.
\newblock Full error analysis for the training of deep neural networks.
\newblock {\em Infin. Dimens. Anal. Quantum Probab. Relat. Top.}, 25(2):Paper
  No. 2150020, 76, 2022.

\bibitem{BenamouBrenier:2000}
Jean-David Benamou and Yann Brenier.
\newblock A computational fluid mechanics solution to the {M}onge-{K}antorovich
  mass transfer problem.
\newblock {\em Numer. Math.}, 84(3):375--393, 2000.

\bibitem{BBG:2004}
Jean-David Benamou, Yann Brenier, and Kevin Guittet.
\newblock Numerical analysis of a multi-phasic mass transport problem.
\newblock In {\em Recent advances in the theory and applications of mass
  transport}, volume 353 of {\em Contemp. Math.}, pages 1--17. Amer. Math.
  Soc., Providence, RI, 2004.

\bibitem{BCCNP:2015}
Jean-David Benamou, Guillaume Carlier, Marco Cuturi, Luca Nenna, and Gabriel
  Peyr\'{e}.
\newblock Iterative {B}regman {P}rojections for {R}egularized {T}ransportation
  {P}roblems.
\newblock {\em SIAM Journal on Scientific Computing}, 37(2):A1111--A1138, 2015.

\bibitem{Bertsekas:1981}
Dimitri~P. Bertsekas.
\newblock A new algorithm for the assignment problem.
\newblock {\em Math. Programming}, 21(2):152--171, 1981.

\bibitem{BC:1989}
{Dimitri P.} Bertsekas and {David A.} Castanon.
\newblock The auction algorithm for the transportation problem.
\newblock {\em Annals of Operations Research}, 20(1):67--96, December 1989.

\bibitem{BE:1988}
Dimitri~P. Bertsekas and Jonathan Eckstein.
\newblock Dual coordinate step methods for linear network flow problems.
\newblock {\em Math. Programming}, 42(2, (Ser. B)):203--243, 1988.

\bibitem{Bjorn-Bjorn11}
Anders Bj{\"o}rn and Jana Bj{\"o}rn.
\newblock {\em Nonlinear potential theory on metric spaces}, volume~17 of {\em
  EMS Tracts in Mathematics}.
\newblock European Mathematical Society (EMS), Z\"urich, 2011.

\bibitem{BGKP:2019}
Helmut B\"{o}lcskei, Philipp Grohs, Gitta Kutyniok, and Philipp Petersen.
\newblock Optimal {A}pproximation with {S}parsely {C}onnected {D}eep {N}eural
  {N}etworks.
\newblock {\em SIAM Journal on Mathematics of Data Science}, 1(1):8--45, 2019.

\bibitem{BPPH:2011}
Nicolas Bonneel, Michiel Panne, Sylvain Paris, and Wolfgang Heidrich.
\newblock Displacement {I}nterpolation {U}sing {L}agrangian {M}ass {T}ransport.
\newblock {\em ACM Trans. Graph.}, 30:158, 12 2011.

\bibitem{HouzeMeng:2023}
Houze Cao and Meng Xue.
\newblock Adversarial training for better robustness.
\newblock In S{\'e}rgio~Ivan Lopes, Paula Fraga-Lamas, Tiago~M.
  Fern{\'a}ndes-Cam{\'a}res, Babu~R. Dawadi, Danda~B. Rawat, and Subarna
  Shakya, editors, {\em Smart Technologies for Sustainable and Resilient
  Ecosystems}, pages 75--84, Cham, 2023. Springer Nature Switzerland.

\bibitem{CWVB:2009}
Rick Chartrand, Brendt Wohlberg, Kevin Vixie, and Erik Bollt.
\newblock A gradient descent solution to the {M}onge-{K}antorovich problem.
\newblock {\em Applied Mathematical Sciences}, 3:1071--1080, 01 2009.

\bibitem{Cheeger99}
Jeff Cheeger.
\newblock Differentiability of {L}ipschitz functions on metric measure spaces.
\newblock {\em Geom. Funct. Anal.}, 9(3):428--517, 1999.

\bibitem{CRLVP:2020}
Lenaic Chizat, Pierre Roussillon, Flavien L{\'e}ger, Fran{\c{c}}ois-Xavier
  Vialard, and Gabriel Peyr{\'e}.
\newblock Faster wasserstein distance estimation with the sinkhorn divergence.
\newblock {\em Advances in Neural Information Processing Systems},
  33:2257--2269, 2020.

\bibitem{cohen}
Albert Cohen, Mark~A. Davenport, and Dany Leviatan.
\newblock Correction to: {O}n the stability and accuracy of least squares
  approximations [ {MR}3105946].
\newblock {\em Found. Comput. Math.}, 19(1):239, 2019.

\bibitem{CFD:2018}
Nicolas Courty, Rémi Flamary, and Mélanie Ducoffe.
\newblock Learning {W}asserstein {E}mbeddings.
\newblock In {\em International Conference on Learning Representations}, 2018.

\bibitem{Cuturi:2013}
Marco Cuturi.
\newblock Sinkhorn {D}istances: {L}ightspeed {C}omputation of {O}ptimal
  {T}ransport.
\newblock In C.J. Burges, L.~Bottou, M.~Welling, Z.~Ghahramani, and K.Q.
  Weinberger, editors, {\em Advances in Neural Information Processing Systems},
  volume~26. Curran Associates, Inc., 2013.

\bibitem{Cybenko1989ApproximationBS}
George~V. Cybenko.
\newblock Approximation by superpositions of a sigmoidal function.
\newblock {\em Mathematics of Control, Signals and Systems}, 2:303--314, 1989.

\bibitem{DMgamma}
Gianni Dal~Maso.
\newblock {\em An introduction to {$\Gamma$}-convergence}, volume~8 of {\em
  Progress in Nonlinear Differential Equations and their Applications}.
\newblock Birkh\"{a}user Boston, Inc., Boston, MA, 1993.

\bibitem{DelloSchiavo20}
Lorenzo Dello~Schiavo.
\newblock A {R}ademacher-type theorem on {$L^2$}-{W}asserstein spaces over
  closed {R}iemannian manifolds.
\newblock {\em J. Funct. Anal.}, 278(6):108397, 57, 2020.

\bibitem{DemingStephan:1940}
W.~Edwards Deming and Frederick~F. Stephan.
\newblock On a least squares adjustment of a sampled frequency table when the
  expected marginal totals are known.
\newblock {\em Ann. Math. Statistics}, 11:427--444, 1940.

\bibitem{devore_hanin_petrova_2021}
Ronald DeVore, Boris Hanin, and Guergana Petrova.
\newblock Neural network approximation.
\newblock {\em Acta Numerica}, 30:327–444, 2021.

\bibitem{Elbrchter2019DeepNN}
Dennis Elbr{\"a}chter, Dmytro Perekrestenko, Philipp Grohs, and Helmut
  B{\"o}lcskei.
\newblock Deep {N}eural {N}etwork {A}pproximation {T}heory.
\newblock {\em IEEE Transactions on Information Theory}, 67:2581--2623, 2019.

\bibitem{flamary2021pot}
R{\'e}mi Flamary, Nicolas Courty, Alexandre Gramfort, Mokhtar~Z. Alaya,
  Aur{\'e}lie Boisbunon, Stanislas Chambon, Laetitia Chapel, Adrien Corenflos,
  Kilian Fatras, Nemo Fournier, L{\'e}o Gautheron, Nathalie~T.H. Gayraud,
  Hicham Janati, Alain Rakotomamonjy, Ievgen Redko, Antoine Rolet, Antony
  Schutz, Vivien Seguy, Danica~J. Sutherland, Romain Tavenard, Alexander Tong,
  and Titouan Vayer.
\newblock {POT}: {P}ython {O}ptimal {T}ransport.
\newblock {\em Journal of Machine Learning Research}, 22(78):1--8, 2021.

\bibitem{FHSXX}
Massimo Fornasier, Pascal Heid, and Giacomo~E. Sodini.
\newblock Stability and accuracy of {T}ikhonov regularization on
  infinitesimally {H}ilbertian metric measure spaces.
\newblock in preparation.

\bibitem{FSS22}
Massimo Fornasier, Giuseppe Savar\'{e}, and Giacomo~Enrico Sodini.
\newblock Density of subalgebras of {L}ipschitz functions in metric {S}obolev
  spaces and applications to {W}asserstein {S}obolev spaces.
\newblock {\em J. Funct. Anal.}, 285(11):Paper No. 110153, 76, 2023.

\bibitem{FranklinLorenz:1989}
Joel Franklin and Jens Lorenz.
\newblock On the scaling of multidimensional matrices.
\newblock {\em Linear Algebra Appl.}, 114/115:717--735, 1989.

\bibitem{GCPB:2016}
Aude Genevay, Marco Cuturi, Gabriel Peyr\'{e}, and Francis Bach.
\newblock Stochastic {O}ptimization for {L}arge-{S}cale {O}ptimal {T}ransport.
\newblock In {\em Proceedings of the 30th International Conference on Neural
  Information Processing Systems}, NIPS'16, page 3440–3448, Red Hook, NY,
  USA, 2016. Curran Associates Inc.

\bibitem{Gigli15-new}
Nicola Gigli.
\newblock On the differential structure of metric measure spaces and
  applications.
\newblock {\em Mem. Amer. Math. Soc.}, 236(1113):vi+91, 2015.

\bibitem{GMR15}
Nicola Gigli, Andrea Mondino, and Tapio Rajala.
\newblock Euclidean spaces as weak tangents of infinitesimally {H}ilbertian
  metric measure spaces with {R}icci curvature bounded below.
\newblock {\em J. Reine Angew. Math.}, 705:233--244, 2015.

\bibitem{pasquabook}
Nicola Gigli and Enrico Pasqualetto.
\newblock {\em Lectures on nonsmooth differential geometry}, volume~2 of {\em
  SISSA Springer Series}.
\newblock Springer, Cham, [2020] \copyright 2020.

\bibitem{GoldbergTarjan:1989}
Andrew~V. Goldberg and Robert~E. Tarjan.
\newblock Finding minimum-cost circulations by canceling negative cycles.
\newblock {\em J. Assoc. Comput. Mach.}, 36(4):873--886, 1989.

\bibitem{GAN:2014}
Ian Goodfellow, Jean Pouget-Abadie, Mehdi Mirza, Bing Xu, David Warde-Farley,
  Sherjil Ozair, Aaron Courville, and Yoshua Bengio.
\newblock Generative {A}dversarial {N}ets.
\newblock In Z.~Ghahramani, M.~Welling, C.~Cortes, N.~Lawrence, and K.Q.
  Weinberger, editors, {\em Advances in Neural Information Processing Systems},
  volume~27. Curran Associates, Inc., 2014.

\bibitem{GJS:2022}
Philipp Grohs, Arnulf Jentzen, and Diyora Salimova.
\newblock Deep neural network approximations for solutions of {PDE}s based on
  {M}onte {C}arlo algorithms.
\newblock {\em Partial Differ. Equ. Appl.}, 3(4):Paper No. 45, 41, 2022.

\bibitem{Guittet:2003}
K.~Guittet.
\newblock On the time-continuous mass transport problem and its approximation
  by augmented {L}agrangian techniques.
\newblock {\em SIAM J. Numer. Anal.}, 41(1):382--399, 2003.

\bibitem{WGAN2:2017}
Ishaan Gulrajani, Faruk Ahmed, Martin Arjovsky, Vincent Dumoulin, and Aaron~C
  Courville.
\newblock Improved {T}raining of {W}asserstein {GAN}s.
\newblock In I.~Guyon, U.~Von Luxburg, S.~Bengio, H.~Wallach, R.~Fergus,
  S.~Vishwanathan, and R.~Garnett, editors, {\em Advances in Neural Information
  Processing Systems}, volume~30. Curran Associates, Inc., 2017.

\bibitem{HZTA:2004}
Steven Haker, Lei Zhu, Allen Tannenbaum, and Sigurd Angenent.
\newblock Optimal {M}ass {T}ransport for {R}egistration and {W}arping.
\newblock {\em Int. J. Comput. Vision}, 60(3):225–240, dec 2004.

\bibitem{MR4362469}
Lukas Herrmann, Joost A.~A. Opschoor, and Christoph Schwab.
\newblock Constructive deep {R}e{LU} neural network approximation.
\newblock {\em J. Sci. Comput.}, 90(2):Paper No. 75, 37, 2022.

\bibitem{HORNIK1991251}
Kurt Hornik.
\newblock Approximation capabilities of multilayer feedforward networks.
\newblock {\em Neural Networks}, 4(2):251--257, 1991.

\bibitem{HORNIK1989359}
Kurt Hornik, Maxwell Stinchcombe, and Halbert White.
\newblock Multilayer feedforward networks are universal approximators.
\newblock {\em Neural Networks}, 2(5):359--366, 1989.

\bibitem{Zhangkai23}
Zhangkai Huang.
\newblock Isometric immersions of {RCD(K, N)} spaces via heat kernels.
\newblock {\em Calculus of Variations and Partial Differential Equations},
  62(4):121, 2023.

\bibitem{kaggle}
Shadab Hussain.
\newblock Cifar 10- cnn using pytorch.
\newblock Online:
  \url{https://www.kaggle.com/code/shadabhussain/cifar-10-cnn-using-pytorch},
  2021.
\newblock Accessed: 2024-08-19.

\bibitem{JR:2020}
Arnulf Jentzen and Adrian Riekert.
\newblock Strong overall error analysis for the training of artificial neural
  networks via random initializations.
\newblock {\em arXiv:2012.08443}, 2020.

\bibitem{KST:2020}
Keisuke Kawano, Satoshi Koide, and Takuro Kutsuna.
\newblock Learning {W}asserstein {I}sometric {E}mbedding for {P}oint {C}louds.
\newblock In {\em 2020 International Conference on 3D Vision (3DV)}, pages
  473--482, 2020.

\bibitem{KidgerLyons2020}
Patrick Kidger and Terry Lyons.
\newblock {Universal {A}pproximation with {D}eep {N}arrow {N}etworks}.
\newblock In Jacob Abernethy and Shivani Agarwal, editors, {\em Proceedings of
  Thirty Third Conference on Learning Theory}, volume 125 of {\em Proceedings
  of Machine Learning Research}, pages 2306--2327. PMLR, 09--12 Jul 2020.

\bibitem{Kingma2014AdamAM}
Diederik~P. Kingma and Jimmy Ba.
\newblock Adam: {A} {M}ethod for {S}tochastic {O}ptimization.
\newblock {\em CoRR}, abs/1412.6980, 2014.

\bibitem{KMT:2016}
Jun Kitagawa, Quentin Mérigot, and Boris Thibert.
\newblock A {N}ewton algorithm for semi-discrete optimal transport.
\newblock {\em Journal of the European Mathematical Society}, 21, 03 2016.

\bibitem{KMT:2019}
Jun Kitagawa, Quentin Mérigot, and Boris Thibert.
\newblock Convergence of a {N}ewton algorithm for semi-discrete optimal
  transport.
\newblock {\em J. Eur. Math. Soc.}, 21(9):2603--2651, 2019.

\bibitem{MR0111809}
A.~N. Kolmogorov.
\newblock On the representation of continuous functions of many variables by
  superposition of continuous functions of one variable and addition.
\newblock {\em Dokl. Akad. Nauk SSSR}, 114:953--956, 1957.

\bibitem{KNRH:2021}
Soheil Kolouri, Navid Naderializadeh, Gustavo~K. Rohde, and Heiko Hoffmann.
\newblock Wasserstein {E}mbedding for {G}raph {L}earning.
\newblock In {\em International Conference on Learning Representations}, 2021.

\bibitem{Kolouri:2017}
Soheil Kolouri, Se~Rim Park, Matthew Thorpe, Dejan Slepcev, and Gustavo~K.
  Rohde.
\newblock Optimal {M}ass {T}ransport: {S}ignal processing and machine-learning
  applications.
\newblock {\em IEEE Signal Processing Magazine}, 34(4):43--59, 2017.

\bibitem{KTOR:2016}
Soheil Kolouri, Akif~B. Tosun, John~A. Ozolek, and Gustavo~K. Rohde.
\newblock A continuous linear optimal transport approach for pattern analysis
  in image datasets.
\newblock {\em Pattern Recognition}, 51:453--462, 2016.

\bibitem{cifar10}
Alex Krizhevsky.
\newblock Learning {M}ultiple {L}ayers of {F}eatures from {T}iny {I}mages.
\newblock Technical report, 2009.

\bibitem{Kuhn:1955}
H.~W. Kuhn.
\newblock The {H}ungarian method for the assignment problem.
\newblock {\em Naval Res. Logist. Quart.}, 2:83--97, 1955.

\bibitem{Levy:2015}
Bruno L\'{e}vy.
\newblock A numerical algorithm for {$L^2$} semi-discrete optimal transport in
  3{D}.
\newblock {\em ESAIM Math. Model. Numer. Anal.}, 49(6):1693--1715, 2015.

\bibitem{LMS18}
Matthias Liero, Alexander Mielke, and Giuseppe Savar\'{e}.
\newblock Optimal entropy-transport problems and a new
  {H}ellinger-{K}antorovich distance between positive measures.
\newblock {\em Invent. Math.}, 211(3):969--1117, 2018.

\bibitem{LingOkada:2007}
Haibin Ling and Kazunori Okada.
\newblock An {E}fficient {E}arth {M}over's {D}istance {A}lgorithm for {R}obust
  {H}istogram {C}omparison.
\newblock {\em IEEE Transactions on Pattern Analysis and Machine Intelligence},
  29:840--853, 2007.

\bibitem{HXD:2018}
Huidong Liu, Xianfeng GU, and Dimitris Samaras.
\newblock A {T}wo-{S}tep {C}omputation of the {E}xact {GAN} {W}asserstein
  {D}istance.
\newblock In Jennifer Dy and Andreas Krause, editors, {\em Proceedings of the
  35th International Conference on Machine Learning}, volume~80 of {\em
  Proceedings of Machine Learning Research}, pages 3159--3168. PMLR, 10--15 Jul
  2018.

\bibitem{Liuetal:2020}
Xiaofeng Liu, Wenxuan Ji, Jane You, Georges El~Fakhri, and Jonghye Woo.
\newblock Severity-aware semantic segmentation with reinforced wasserstein
  training.
\newblock In {\em 2020 IEEE/CVF Conference on Computer Vision and Pattern
  Recognition (CVPR)}, pages 12563--12572, 2020.

\bibitem{LP:2005}
Gr\'{e}goire Loeper and Francesca Rapetti.
\newblock Numerical solution of the {M}onge-{A}mp\`ere equation by a {N}ewton's
  algorithm.
\newblock {\em C. R. Math. Acad. Sci. Paris}, 340(4):319--324, 2005.

\bibitem{Lott-Villani07}
John Lott and Cédric Villani.
\newblock {Ricci curvature for metric-measure spaces via optimal transport}.
\newblock {\em ArXiv Mathematics e-prints}, December 2004.

\bibitem{pasqualetto}
Danka Lu\v{c}i\'{c} and Enrico Pasqualetto.
\newblock Gamma-convergence of {C}heeger energies with respect to increasing
  distances.
\newblock {\em J. Math. Anal. Appl.}, 515(1):Paper No. 126415, 10, 2022.

\bibitem{LevySchwindt:2018}
Bruno Lévy and Erica~L. Schwindt.
\newblock Notions of optimal transport theory and how to implement them on a
  computer.
\newblock {\em Computers \& Graphics}, 72:135--148, 2018.

\bibitem{WGAN3:2019}
Anton Mallasto, Guido Mont{\'u}far, and Augusto Gerolin.
\newblock How {W}ell {D}o {WGAN}s {E}stimate the {W}asserstein {M}etric?
\newblock {\em arXiv:1910.03875}, 2019.

\bibitem{Merigot:2011}
Quentin Mérigot.
\newblock A {M}ultiscale {A}pproach to {O}ptimal {T}ransport.
\newblock {\em Computer Graphics Forum}, 30(5):1583--1592, 2011.

\bibitem{PeleWerman:2009}
Ofir Pele and Michael Werman.
\newblock Fast and robust {E}arth {M}over's {D}istances.
\newblock In {\em 2009 IEEE 12th International Conference on Computer Vision},
  pages 460--467, 2009.

\bibitem{PeyreCuturi:2019}
Gabriel Peyré and Marco Cuturi.
\newblock Computational {O}ptimal {T}ransport: {W}ith {A}pplications to {D}ata
  {S}cience.
\newblock {\em Foundations and Trends® in Machine Learning}, 11(5-6):355--607,
  2019.

\bibitem{RAISSI2019686}
M.~Raissi, P.~Perdikaris, and G.E. Karniadakis.
\newblock Physics-informed neural networks: A deep learning framework for
  solving forward and inverse problems involving nonlinear partial differential
  equations.
\newblock {\em Journal of Computational Physics}, 378:686--707, 2019.

\bibitem{RGC:2017}
Aaditya Ramdas, Nicol{\'{a}}s~Garc{\'{\i}}a Trillos, and Marco Cuturi.
\newblock On {W}asserstein {T}wo-{S}ample {T}esting and {R}elated {F}amilies of
  {N}onparametric {T}ests.
\newblock {\em Entropy}, 19(2):47, 2017.

\bibitem{Rubner2000TheEM}
Yossi Rubner, Carlo Tomasi, and Leonidas~J. Guibas.
\newblock The {E}arth {M}over's {D}istance as a {M}etric for {I}mage
  {R}etrieval.
\newblock {\em International Journal of Computer Vision}, 40:99--121, 2000.

\bibitem{santambrogio}
Filippo Santambrogio.
\newblock {\em Optimal transport for applied mathematicians}, volume~87 of {\em
  Progress in Nonlinear Differential Equations and their Applications}.
\newblock Birkh\"{a}user/Springer, Cham, 2015.
\newblock Calculus of variations, PDEs, and modeling.

\bibitem{Savare22}
Giuseppe Savar\'e.
\newblock Sobolev {S}paces in {E}xtended {M}etric-{M}easure {S}paces.
\newblock In {\em New Trends on Analysis and Geometry in Metric Spaces}, pages
  117--276. Springer, 2022.

\bibitem{Schmitzer:2019}
Bernhard Schmitzer.
\newblock Stabilized {S}parse {S}caling {A}lgorithms for {E}ntropy
  {R}egularized {T}ransport {P}roblems.
\newblock {\em SIAM Journal on Scientific Computing}, 41(3):A1443--A1481, 2019.

\bibitem{Schieber:2019}
J{\"o}rn Schrieber.
\newblock {\em Algorithms for {O}ptimal {T}ransport and {W}asserstein
  {D}istances}.
\newblock PhD thesis, Georg-August-Universit{\"a}t G{\"o}ttingen, 2019.

\bibitem{SeguyCuturi:2015}
Vivien Seguy and Marco Cuturi.
\newblock Principal {G}eodesic {A}nalysis for {P}robability {M}easures under
  the {O}ptimal {T}ransport {M}etric.
\newblock In C.~Cortes, N.~Lawrence, D.~Lee, M.~Sugiyama, and R.~Garnett,
  editors, {\em Advances in Neural Information Processing Systems}, volume~28.
  Curran Associates, Inc., 2015.

\bibitem{Shanmugalingam00}
Nageswari Shanmugalingam.
\newblock Newtonian spaces: an extension of {S}obolev spaces to metric measure
  spaces.
\newblock {\em Rev. Mat. Iberoamericana}, 16(2):243--279, 2000.

\bibitem{Sinkhorn:1964}
Richard Sinkhorn.
\newblock A relationship between arbitrary positive matrices and doubly
  stochastic matrices.
\newblock {\em Ann. Math. Statist.}, 35:876--879, 1964.

\bibitem{S22}
Giacomo~Enrico Sodini.
\newblock The general class of {W}asserstein {S}obolev spaces: density of
  cylinder functions, reflexivity, uniform convexity and {C}larkson's
  inequalities.
\newblock {\em Calc. Var. Partial Differential Equations}, 62(7):Paper No. 212,
  41, 2023.

\bibitem{marc}
Giacomo~Enrico Sodini and Lorenzo Dello~Schiavo.
\newblock The {H}ellinger--{K}antorovich metric measure geometry on spaces of
  measures.
\newblock {\em In preparation}, 2024.

\bibitem{sonnleitner2023}
Mathias Sonnleitner and Mario Ullrich.
\newblock On the power of iid information for linear approximation.
\newblock {\em arXiv:2310.12740}, 2023.

\bibitem{WGAN4:2021}
Jan Stanczuk, Christian Etmann, Lisa~Maria Kreusser, and Carola-Bibiane
  Sch{\"o}nlieb.
\newblock Wasserstein {GAN}s work because they fail (to approximate the
  {W}asserstein distance).
\newblock {\em arXiv:2103.01678}, 2021.

\bibitem{Sturm06I}
Karl-Theodor Sturm.
\newblock On the geometry of metric measure spaces. {I}.
\newblock {\em Acta Math.}, 196(1):65--131, 2006.

\bibitem{Sturm06II}
Karl-Theodor Sturm.
\newblock On the geometry of metric measure spaces. {II}.
\newblock {\em Acta Math.}, 196(1):133--177, 2006.

\bibitem{Tarjan:1997}
Robert~E. Tarjan.
\newblock Dynamic trees as search trees via euler tours, applied to the network
  simplex algorithm.
\newblock {\em Math. Program.}, 78(2):169–177, aug 1997.

\bibitem{TSG_2017-2019__35__197_0}
David Tewodrose.
\newblock A survey on {s}pectral embeddings and their application in data
  analysis.
\newblock {\em S\'eminaire de th\'eorie spectrale et g\'eom\'etrie},
  35:197--244, 2017-2019.

\bibitem{Tropp:2012}
Joel~A. Tropp.
\newblock User-friendly tail bounds for sums of random matrices.
\newblock {\em Found. Comput. Math.}, 12(4):389--434, 2012.

\bibitem{Villani03}
C{\'e}dric Villani.
\newblock {\em Topics in optimal transportation}, volume~58 of {\em Graduate
  Studies in Mathematics}.
\newblock American Mathematical Society, Providence, RI, 2003.

\bibitem{Villani:09}
C\'{e}dric Villani.
\newblock {\em Optimal transport}, volume 338 of {\em Grundlehren der
  mathematischen Wissenschaften [Fundamental Principles of Mathematical
  Sciences]}.
\newblock Springer-Verlag, Berlin, 2009.
\newblock Old and new.

\bibitem{WSBOR:2013}
Wei Wang, Dejan Slepčev, Saurav Basu, John Ozolek, and Gustavo Rohde.
\newblock A {L}inear {O}ptimal {T}ransportation {F}ramework for {Q}uantifying
  and {V}isualizing {V}ariations in {S}ets of {I}mages.
\newblock {\em International journal of computer vision}, 101:254--269, 06
  2013.

\bibitem{Zangetal:20}
Yaohua Zang, Gang Bao, Xiaojing Ye, and Haomin Zhou.
\newblock Weak adversarial networks for high-dimensional partial differential
  equations.
\newblock {\em Journal of Computational Physics}, 411:109409, 2020.

\bibitem{Zeidler:90}
E.~Zeidler.
\newblock {\em Nonlinear functional analysis and its applications. {II}/{B}}.
\newblock Springer-Verlag, New York, 1990.

\end{thebibliography}
\end{document}